\documentclass[12pt,reqno]{amsart}
\usepackage{amssymb}
\usepackage{mathtools}
\usepackage{txfonts}
\usepackage{eucal}
\usepackage{extarrows}

\usepackage[all]{xy}

\usepackage{tikz}

\usepackage{graphicx}
\usepackage{float}

\usepackage{xspace}

\usepackage{amsmath}

\usepackage[a4paper,body={16.3cm,22.8cm},centering]{geometry}
\usepackage[colorlinks,final,backref=page,hyperindex,pdftex]{hyperref}

\newcommand{\nc}{\newcommand}
\newcommand{\delete}[1]{}

	\nc{\mlabel}[1]{\label{#1}}  % Use this to suppress names
	\nc{\mcite}[1]{\cite{#1}}  % Use this to suppress names
	\nc{\mref}[1]{\ref{#1}}  % Use this to suppress names
	\nc{\mbibitem}[1]{\bibitem{#1}} % Use this to show number name
\delete{
	\nc{\mlabel}[1]{\label{#1}  % Use the next two lines to show names
		{\hfill \hspace{1cm}{\small\tt{{\ }\hfill(#1)}}}}
	\nc{\mcite}[1]{\cite{#1}{\small{\tt{{\ }(#1)}}}}  % Use this lines to show names
	\nc{\mref}[1]{\ref{#1}{{\tt{{\ }(#1)}}}}  % Use this lines to show names
	\nc{\mbibitem}[1]{\bibitem[\bf #1]{#1}} % Use this to show name
}
\newtheorem{theorem}{Theorem}[section]
\newtheorem{prop}[theorem]{Proposition}

\newtheorem{coro}[theorem]{Corollary}

\theoremstyle{definition}
\newtheorem{defn}[theorem]{Definition}
\newtheorem{remark}[theorem]{Remark}
\newtheorem{exam}[theorem]{Example}
\newtheorem{prop-def}{Proposition-Definition}[section]

\newcommand\cal[1]{\mathcal{#1}}

\newcommand\alphlist{a,b,c,d,e,f,g,h,i,j,k,l,m,n,o,p,q,r,s,t,u,v,w,x,y,z}
\newcommand\Alphlist{A,B,C,D,E,F,G,H,I,J,K,L,M,N,O,P,Q,R,S,T,U,V,W,X,Y,Z}
\newcommand\getcmds[3]{\expandafter\newcommand\csname #2#1\endcsname{#3{#1}}}
\makeatletter
\@for\x:=\alphlist\do{\expandafter\getcmds\expandafter{\x}{frak}{\mathfrak}}% \fraka,\frakb,...
\@for\x:=\Alphlist\do{\expandafter\getcmds\expandafter{\x}{frak}{\mathfrak}}% \frakA,\frakB,...
\makeatother

\nc{\bfk}{{\bf k}}
\font\cyr=wncyr10

\newfont{\scyr}{wncyr10 scaled 550}
\nc{\sha}{\mbox{\cyr X}}
\nc{\ssha}{\mbox{\bf \scyr X}}

\nc{\id}{\mathrm{id}}
\nc{\Id}{\mathrm{Id}}
\nc{\lbar}[1]{\overline{#1}}
\nc{\ot}{\otimes}
\nc{\dep}{\mathrm{dep}}

\nc{\tred}[1]{\textcolor{red}{#1}} \nc{\tgreen}[1]{\textcolor{green}{#1}}
\nc{\tblue}[1]{\textcolor{blue}{#1}} \nc{\tpurple}[1]{\textcolor{purple}{#1}}

\nc{\li}[1]{\tpurple{\underline{Li:}#1 }}
\nc{\liadd}[1]{\tpurple{#1}}
\nc{\xing}[1]{\tblue{\underline{Xing:}#1 }}
\nc{\dominique}[1]{\tblue{\underline{Dominique: }#1 }}
\nc{\yuan}[1]{\tred{\underline{Yuan:}#1 }}
\nc{\markus}[1]{\tred{\underline{Markus:} #1}}
\nc\hu[1]{\tgreen{\underline{Huhu:}#1}}

\usetikzlibrary{calc}
\newlength\xch
\newsavebox\dbox
\sbox\dbox{\tikz{\fill (0,0) circle (0.05cm);}}
\newif\ifqdd
\newif\ifzdd
\nc{\dnx}{\Delta_n A} \nc{\dx}{\Delta A} \nc{\dgp}{{\rm deg_{P}}}
\nc{\dgt}{{\rm deg_{T}}} \nc{\dg}{{\rm deg}} \nc{\ida}{ID($A$)} \nc{\tu}{\tilde{u}} \nc{\tv}{\tilde{v}}
\nc{\nr}{\calr_n} \nc{\nz}{\calz_n} \nc{\fun}{\cala_{n,d}}
\nc{\fbase}{\calb} \nc{\LF}{\mathrm{RF}} \nc{\FFA}{\mathrm{LF}} \nc{\irr}{\mathrm{Irr}}
\nc{\result}{\bfk\mathrm{Irr}(S_n)}  \nc{\I}{I_{\mathrm{ID},n}^0}
\nc{\nrs}{\calr_n^\star} \nc{\ii}{\mathrm{I}} \nc{\iii}{\mathrm{II}}
\nc{\intl}{{\rm int}}\nc{\ws}[1]{{#1}}\nc{\deleted}[1]{\delete{#1}}\nc{\plas}{placements\xspace}

\nc{\bim}[1]{#1}  \nc{\shaop}{\sha_{\Omega}^{+}}  \nc{\shao}{\sha_{\Omega}}
\nc{\bbim}[2]{#1 #2} \nc{\bbbim}[2]{#1,\, #2} \nc{\RBF}{{\rm RBF}}
\nc{\frb}{F_{\RB}} \nc{\shaf}{\ssha_{\tiny{\Omega}}} \nc{\sham}{\diamond_{\tiny{\Omega}}}
\nc{\lf}{\lfloor} \nc{\rf}{\rfloor} \nc{\shan}{\ssha_{\lambda}}
\nc{\rlex}{{\rm {lex}}} \nc{\bb}{\Box} \nc{\ra}{\rightarrow}
\nc{\e}{{\rm {e}}}
\nc{\DDF}{\mathrm{DD}(X,\,\Omega)}\nc{\DTF}{\mathrm{DT}(X,\,\Omega)} \nc{\DT}{\mathrm{DT}'(\Omega,\,V)}
\nc{\bra}{\mathrm{bra}} \nc{\bre}{\mathrm{bre}}
\nc{\dec}{\mathrm{dec}} \nc{\diamondw}{\diamond_{w}}
\nc{\type}{\mathrm{type}}

\nc\caF[1]{\cal{F}_{#1}(X,\,\Omega)}
\nc\calt{\cal{T}(X,\,\Omega)} \nc\caltn{\cal{T}_n(X,\,\Omega)}
\nc\caltbin{\cal{T}_b(X,\,\Omega)}
\nc\calta{\cal{T}_0(X,\,\Omega)}
\nc\caltb{\cal{T}_1(X,\,\Omega)}
\nc\caltc{\cal{T}_2(X,\,\Omega)}
\nc\caltd{\cal{T}_3(X,\,\Omega)}
\nc\caltm{\cal{T}_m(X,\,\Omega)}
\nc\calf{\cal{F}(X,\,\Omega)}
\nc\fram{\frak{M}(\Omega,\, X)}
\nc\shaw{\sha^{NC}_w(\Omega,\, X)}
\nc\dw{\diamond_w} \nc\dl{\diamond_\ell}
\nc\shal{\sha^{NC}_\ell(X,\, \Omega)} \nc\shav{\sha^{NC}_w(\Omega,\, V)} \nc\shat{\sha^{NC,1}_w(\Omega,\, T^{+}(V))}
\nc{\cfo}{\cal{F}(X,\,\Omega)}

\nc{\lar}{\varinjlim}
\nc\XO{(X,\,\Omega)}
\def\cxo#1#2;{\cal{#1}#2\XO}
\def\cxob#1#2;{\cal{#1}#2_b\XO}
\nc\lrf[2]{B_{#2}^+(#1)}
\nc{\fd}{\mathrm{\text{typed angularly decorated planar rooted trees}}}
\nc{\rb}{\mathrm{RBFWs}} \nc{\dfw}{\mathrm{DFW{(X)}}} \nc{\tfw}{\mathrm{TFW{(X)}}}
\nc{\tfv}{\mathrm{TFW{(V)}}} \nc{\rbf}{\mathrm{RBF}}
\nc{\db}{\mathrm{db}}
\nc{\st}{\mathrm{st}}

\def\Ve#1,#2,#3;{\vee_{#1,\,(#2,\,#3)}}
\def\bigv#1;#2;#3;{\bigvee\nolimits_{#1}^{#2;\,#3}}

\nc{\Irr}{\mathrm{Irr}} \nc{\lc}{\lfloor} \nc{\rc}{\rfloor}
\nc{\rswx}{\frak{M}( \Omega_R\sqcup \Omega_S, X)}
\nc{\rswxs}{\frak{M}^\star( \Omega_R\sqcup \Omega_S, X)}
\nc{\Dl}{\leq_{_{{\rm Dl}}}} \nc{\Dll}{<_{_{{\rm Dl}}}} \nc{\bbs}{\mathbb{S}}
\nc{\orbsa}{$\Omega$-Rota-Baxter system\xspace}
\nc{\orbsas}{$\Omega$-Rota-Baxter systems\xspace}
\nc{\mrbs}{matching Rota-Baxter system\xspace}
\nc{\mrbss}{matching Rota-Baxter systems\xspace}

\nc\prbsla[4]{{R}_{#1}\left(#3\right)R_{#2}\left(#4\right)}
\nc\prbsra[4]{{R}_{#1\rightarrow#2}\left(R_{#1\rhd#2}\left(#3\right)#4\right)+{R}_{#1\leftarrow#2}\left(#3S_{#1\lhd#2}\left(#4\right)\right)}
\nc\prbslb[4]{{S}_{#1}\left(#3\right)S_{#2}\left(#4\right)}
\nc\prbsrb[4]{{S}_{#1\rightarrow#2}\left(R_{#1\rhd#2}\left(#3\right)#4\right)+{S}_{#1\leftarrow#2}\left(#3S_{#1\lhd#2}\left(#4\right)\right)}

\nc\rbsla[4]{\lc #3 \rc ^{R}_{#1} \lc #4 \rc ^R_{#2}}
\nc\rbslb[4]{\lc #3\rc ^{S}_{#1}  \lc #4 \rc ^S_{#2}}
\nc\rbsray[4]{\lc \lc #3 \rc^R_{#1\rhd#2} #4 \rc ^{R}_{#1\rightarrow#2}}
\nc\rbsraz[4]{\lc #3 \lc #4\rc^S_{#1\lhd#2}\rc ^{R}_{#1\leftarrow#2}}
\nc\rbsrby[4]{\lc \lc #3\rc ^R_{#1\rhd#2}#4\rc ^{S}_{#1\rightarrow#2}}
\nc\rbsrbz[4]{\lc #3 \lc #4 \rc ^S_{#1\lhd#2}\rc ^{S}_{#1\leftarrow#2}}
\nc\rbsrac[4]{\lc #3 \lc #4\rc^R_{#1\lhd#2}\rc ^{R}_{#1\leftarrow#2}}

\nc\rbslq[4]{\lc #3 \rc ^{Q}_{#1} \lc #4 \rc ^Q_{#2}}
\nc\rbslt[4]{\lc #3\rc ^{T}_{#1}  \lc #4 \rc ^T_{#2}}
\nc\rbsrqy[4]{\lc \lc #3 \rc^R_{#1\rhd#2} #4 \rc ^{Q}_{#1\rightarrow#2}}
\nc\rbsrqz[4]{\lc #3 \lc #4\rc^S_{#1\lhd#2}\rc ^{Q}_{#1\leftarrow#2}}
\nc\rbsrty[4]{\lc \lc #3\rc ^R_{#1\rhd#2}#4\rc ^{T}_{#1\rightarrow#2}}
\nc\rbsrtz[4]{\lc #3 \lc #4 \rc ^S_{#1\lhd#2}\rc ^{T}_{#1\leftarrow#2}}

\nc{\obr}[1]{\lc #1 \rc_\omega^R} \nc{\obs}[1]{\lc #1 \rc_\omega^S} \nc{\obq}[1]{\lc #1 \rc_\omega^*}
\nc{\obqa}[1]{\lc #1 \rc_\alpha^*} \nc{\obqb}[1]{\lc #1 \rc_\beta^*}
\nc{\obra}[1]{\lc #1 \rc_\alpha^R} \nc{\obrb}[1]{\lc #1 \rc_\beta^R}
\nc{\obsa}[1]{\lc #1 \rc_\alpha^S} \nc{\obsb}[1]{\lc #1 \rc_\beta^S}

\begin{document}

\title[BiHom-$\Omega$-associative algebras and related structures]{BiHom-$\Omega$-associative algebras and related structures}
%\title[BiHom-$\Omega$-associative algebras, BiHom-$\Omega$-dendriform algebras, BiHom-$\Omega$-pre-Lie algebras and BiHom $\Omega$-Lie algebras]{BiHom-$\Omega$-associative algebras, BiHom-$\Omega$-dendriform algebras, BiHom-$\Omega$-pre-Lie algebras and BiHom $\Omega$-Lie algebras}

\author{Jiaqi Liu
}
\address{School of Mathematics and Statistics, Henan University, Henan, Kaifeng 475004, P.\,R. China}
\email{liujiaqi@henu.edu.cn
}
\author{Yuanyuan Zhang$^{*}$
}
\footnotetext{* Corresponding author.}
\address{School of Mathematics and Statistics, Henan University, Henan, Kaifeng 475004, P.\,R. China}
\email{zhangyy17@henu.edu.cn
}
%
%\author{Huhu Zhang} \address{School of Mathematics and Statistics,
	%	Lanzhou University, Lanzhou, 730000, P.\,R. China}
%\email{zhanghh20@lzu.edu.cn}
%
%\author{Xing Gao$^{*}$}\thanks{*Corresponding author}
%\address{School of Mathematics and Statistics, Lanzhou University,
	%	Lanzhou, 730000, China; School of Mathematics and Statistics
	%	Qinghai Nationalities University, Xining, 810007, China; Gansu Provincial Research Center for Basic Disciplines of Mathematics and Statistics, Lanzhou, 730070, China}
%\email{gaoxing@lzu.edu.cn}

\date{\today}

\begin{abstract}
In this paper, we first propose the concepts of BiHom-$\Omega$-associative algebras, BiHom-$\Omega$-dendriform algebras, BiHom-$\Omega$-pre-Lie algebras and BiHom-$\Omega$-Lie algebras. We then obtain a new BiHom-$\Omega$-associative (resp. Lie) algebra by defining a new multiplication on a BiHom-$\Omega$-associative (resp. Lie) algebra with the Rota-Baxter family of weight $\lambda$.
% the main question we want to answer is the following. In what circumstances can we construct a new BiHom-$\Omega$-associative (resp. Lie) algebra from a Rota-Baxter family of weight $\lambda$ on the BiHom-$\Omega$-associative (resp. Lie) algebra?
In addition, we generalize the classical relationships of associative algebras, pre-Lie algebras, dendriform algebras and Lie algebras to the BiHom-$\Omega$ version.
\end{abstract}

\makeatletter
\@namedef{subjclassname@2020}{\textup{2020} Mathematics Subject Classification}
\makeatother
\subjclass[2020]{
	%15A30, %Algebraic systems of matrices
	16W99, %Rings and algebras with additional structure
	16S10, %Rings determined by universal properties (free algebras, coproducts, adjunction of inverses, etc.)
	%13P10, %Grobner bases; other bases for ideals and modules
	%16S15, %Finite generation, finite presentability, normal forms
	%12H05, %Differential algebra
	%08A70, %application of universal algebra to computer science
	%08B20, %free algebras
	%16R99 %Rings with polynomial identities\none of the above, but in this section
	%06F05,  %Ordered semigroups and monoids
	%20M05 %Free semigroups, generators and relations, word problems
	%16S50 %Endomorphism rings; matrix rings
	%16T10, %Bialgebras
	%16T05,  %Hopf algebras and their applications
	%16T30,  %Connections with combinatorics
	%17B60, %Lie (super)algebras associated with other structures (associative,Jordan, etc.)
	%17D25      %Lie-admissible algebras
	%81T15  %Perturbative methods of renormalization
	%81R10,  %Infinite-dimensional groups and algebras motivated by physics, including Virasoro, Kac-Moody, $W$-algebras and other current algebras and their representations [See also 17B65, 17B67, 22E65, 22E67, 22E70]
	%81R15  %    Operator algebra methods [See also 46Lxx, 81T05]
}

\keywords{BiHom-$\Omega$-associative algebra, BiHom-$\Omega$-dendriform algebra, BiHom-$\Omega$-pre-Lie algebra, BiHom-$\Omega$-Lie algebra, Rota-Baxter family}

\maketitle

\tableofcontents

\setcounter{section}{0}

\allowdisplaybreaks

%========================================================================
\section{Introduction}
%%%%%%%%%%%%%%%%%%%%%%%%%%%%%%

\subsubsection*{BiHom-algebras.}
The origin of Hom-structures may be found in the physics literature around 1990, concerning
$q$-deformations of algebras of vector fields, especially Witt and Virasoro algebras, see for instance~\cite{AizawaSaito,CZ,CMPJ,LKQ,ms}. In the last few years, many articles have generalized the classical algebraic structures to Hom-algebraic structure, see for instance \cite{AC2,AC3,BM,CG,CWZ,HSS}. A generalization has been given in~\cite{gmmp}, where the construction
of a Hom-category including a group action led to the concept of BiHom-type algebras. Yau posed the twisting principle, a main tool for constructing (Bi)Hom-type algebras. Up to now, many concepts of BiHom-type algebras have been proposed, such as BiHom-associative algebras~\cite{gmmp}, BiHom-Lie algebras~\cite{gmmp}, BiHom-(tri)dendriform algebras~\cite{BiHom-dendri}, BiHom-pre-Lie algebras~\cite{bihomprelie} and BiHom-PostLie algebras~\cite{BiHom-PostLie}.\\

\subsubsection*{$\Omega$-algebras.}
The notion of associative algebras relative to a commutative semigroup $\Omega$ first proposed by Aguiar in \cite{Aguiar}. The proposal of this concept has received widespread attention. Later, Das and Zhang et.al referred to $\Omega$-relative associative algebras as $\Omega$-associative algebras in articles~\cite[Definition 4.1]{Das} and~\cite{Zhang}, respectively. In this paper, we will continue to refer to $\Omega$-relative algebras as $\Omega$-algebras. In fact, $\Omega$-type algebras are a generalization of the original type algebras which correspond to the case in which the semigroup $\Omega$ is a single point set. Prior to this study, there was another generalization of algebraic structures, namely family algebra. Rota-Baxter family algebras~\cite{GL} were the first example of family algebraic structures. In recent, the concepts of (tri)dendriform family algebras~\cite{tri-family} and pre-Lie family algebras~\cite{prelief} were proposed by Zhang, Gao and Manchon. Let us take the dendriform algebra as an example to illustrate the relationship between $\Omega$-algebras and family algebras. In the definition of $\Omega$-algebraic structures, if the operation $\prec_{\alpha,\,\beta}$ is independent of $\alpha$ and the operation $\succ_{\alpha,\,\beta}$ is independent of $\beta$ in the $\Omega$-dendriform algebra $(D,\prec_{\alpha,\,\beta},\succ_{\alpha,\,\beta})_{\alpha,\,\beta\in \Omega}$, then $D$ reduces to a dendriform family algebra. In the definition of morphisms, the family algebra is a special case in the $\Omega$-algebra where the map $f_{\alpha}$ is independent of $\alpha$.\\

\subsubsection*{BiHom-$\Omega$-algebras.}
Inspired by BiHom-associative algebras and $\Omega$-associative algebras, we propose the concept of BiHom-$\Omega$-associative algebras. It's not only a promotion of BiHom-associative algebra, but also a generlization of $\Omega$-associative algebra. On the one hand, when the semigroup $\Omega$ is a trivial semigroup with one single element, a BiHom-$\Omega$-associative algebra is precisely a BiHom-associative algebra. On the other hand, when the structure maps of a BiHom-$\Omega$-associative algebra are all identity maps, the BiHom-$\Omega$-associative algebra reduces to an $\Omega$-associative algebra. In this paper, we mainly introduce the second generalized method. In order to better observe the relationships among different BiHom-$\Omega$-type algebras studied in this paper, we give the following commutative diagram.

	
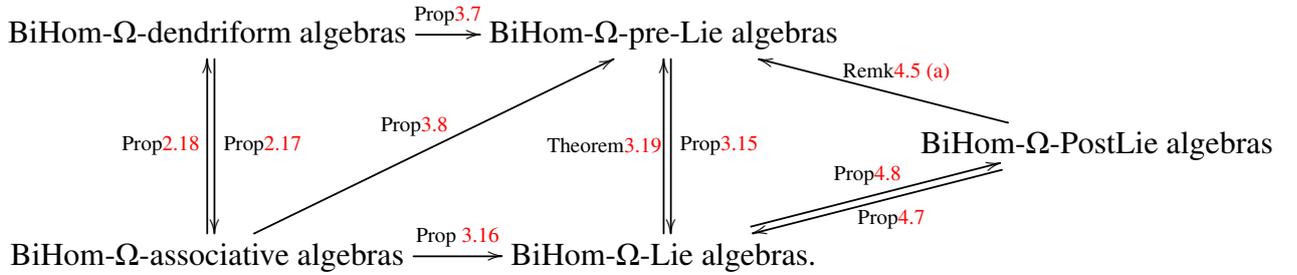
\[\xymatrix{
	\text{BiHom-$\Omega$-dendriform algebras}\ar[r]^-{\text{Prop\ref{BHdendpreLie}}}\ar@<.5ex>[dd]^{\text{Prop\ref{denf-to-Omegaasso}}}_{\text{Prop\ref{asso-den} } }& \text{BiHom-$\Omega$-pre-Lie algebras}\ar@<.5ex>[dd]^{\text{Prop\ref{pLf-to-OmegaLie}}}_{\text{Theorem\ref{lrprelie}}} \\
	 \;& \;& \text{BiHom-$\Omega$-PostLie algebras} \ar[lu]_{\text{Remk\ref{PostLie-Lie}~\ref{item:prelie}}}\ar[ld]^{\text{Prop\ref{interm}}}\\
	\text{BiHom-$\Omega$-associative algebras}\ar[r]^-{\text{Prop \ref{commutator}}}\ar[uur]^{\text{Prop\ref{asso-preLie}}}&\text{BiHom-$\Omega$-Lie algebras.}\ar@<.5ex>[ur]^{\text{Prop\ref{comgen}}}
	\ar"3,1";"1,1"\ar"3,2";"1,2".
}\]\\

\subsubsection*{The outline of this paper.}
In Section \ref{sec2}, we first propose the concepts of BiHom-$\Omega$-associative algebras and BiHom-$\Omega$-dendriform algebras, then we study the relationship between them (Proposition~\ref{denf-to-Omegaasso} and~\ref{asso-den}). Also we prove that a new BiHom-$\Omega$-associative algebra can be induced by a Rota-Baxter family of weight $\lambda$ on the BiHom-$\Omega$-associative algebra (Theorem~\ref{asso-RBO-asso}).
In Section \ref{sec3}, we mainly introduce BiHom-$\Omega$-pre-Lie algebras and BiHom-$\Omega$-Lie algebras and we also give the links between them (Proposition \ref{pLf-to-OmegaLie} and Theorem \ref{lrprelie}). Moreover, similar to Theorem~\ref{asso-RBO-asso}, we obtain a new BiHom-$\Omega$-Lie algebra by defining a new multiplication on a BiHom-$\Omega$-Lie algebra with the Rota-Baxter family of weight $\lambda$ (Theorem \ref{Proposition:BiHLie}). In Section \ref{sec5}, we first introduce the concept of BiHom-$\Omega$-PostLie algebra, then by studying the relationship between BiHom-$\Omega$-PostLie algebras and BiHom-$\Omega$-Lie algebras, we get Proposition~\ref{interm} and Proposition~\ref{comgen}, which are the generalizations of Theorem \ref{Proposition:BiHLie} and Theorem \ref{lrprelie}, respectively. Finally, we simply introduce the concept of BiHom-$\Omega$-pre-Possion algebras.\\

$\mathbf{Notation.}$ Throughout this paper, we fix a commutative unitary ring \bfk, which will be the base
ring of all algebras as well as linear maps. By an algebra we mean a
unitary associative noncommutative algebra, unless the contrary is specified. Denote by $\Omega $ a semigroup, unless otherwise specified. For the composition of two maps $ p $ and $ q $, we will write either $ p\circ q $ or simply $ pq. $

%%%%%%%%%%%%%%%%%%%%%%%%%%%%%%%%%%%%%%
\section{BiHom-$\Omega$-associative algebras and BiHom-$\Omega$-dendriform algebras}\label{sec2}
%%%%%%%%%%%%%%%%%%%%%%%%%%%%%
\setcounter{equation}{0}
%%%%%%%%%%%%%%%%%%%%%%%%%%%%
In this section, we mainly introduce the definitions of BiHom-$\Omega$-associative algebras and BiHom-$\Omega$-dendriform algebras, then we give some results of them.

\subsection{BiHom-$\Omega$-associative algebras}
In this subsection, we first introduce the BiHom version of $\Omega$-associative algebras and its Yau twist property, then we prove that a new BiHom-$\Omega$-associative algebra can be induced by a Rota-Baxter family of weight $\lambda$ on the BiHom-$\Omega$-associative algebra.
Now let's recall the related concepts of $\Omega$-associative algebras.

\begin{defn}\cite{Aguiar}
	An $\mathbf{\Omega}$-$\mathbf{associative \, algebra}$ $ (A, \cdot_{\alpha,\,\beta})_{\alpha,\,\beta\in \Omega} $ is a vector space $A$ equipped with a family of operations $(\cdot_{\alpha,\,\beta}: A \times A \rightarrow A)_{\alpha,\,\beta  \in \Omega} $ such that
	\begin{align}\label{Omega-ass-alg}
		(x\cdot_{\alpha,\,\beta} y)\cdot_{\alpha\,\beta,\,\gamma}z=x\cdot_{\alpha,\,\beta\,\gamma}(y\cdot_{\beta,\,\gamma}z),
	\end{align}
	for all $ x,y,z \in A,\, \alpha,\,\beta,\,\gamma \in \Omega.$
\end{defn}
\begin{defn}\cite{Aguiar}
	Let $ (A,\cdot_{\alpha,\,\beta})_{\alpha,\,\beta\in \Omega} $ and $ (A',\cdot'_{\alpha,\,\beta})_{\alpha,\,\beta\in \Omega} $ be two $\Omega$-associative algebras. A family of linear maps $ (f_{\alpha})_{\alpha\in \Omega}: A\rightarrow A' $ is called an $\mathbf{\Omega}$-$\mathbf{associative \, algebra \, morphism}$ if
	\begin{align}\label{Omega-ass-morphism}
		f_{\alpha\,\beta}(x\cdot_{\alpha,\,\beta}y)=f_{\alpha}(x)\cdot'_{\alpha,\,\beta}f_{\beta}(y),
	\end{align}
	for all $ x,y\in A,\,\alpha,\,\beta\in \Omega. $
\end{defn}

Inspired by the concepts of BiHom-associative algebras \cite{gmmp} and $\Omega$-associative algebras \cite{Aguiar}, now we introduce the BiHom-$\Omega$-associative algebras.

\begin{defn}\label{BiHom-Omega-asso}
	A $\mathbf{BiHom}$-$\mathbf{\Omega}$-$\mathbf{associative \, algebra}$ $( A,\bullet_{\alpha,\,\beta},p_{\alpha},q_{\alpha} )_{\alpha,\,\beta \in \Omega} $ is a vector space $A$ equipped with two commuting families of linear maps $p_{\alpha} , q_{\alpha}:A\rightarrow A$ and a family of bilinear maps $( \bullet_{\alpha,\,\beta}
	:A\otimes A\rightarrow A ) _{\alpha,\,\beta \in \Omega}$ such that
%		\[p \circ q =q \circ p,\]
		\begin{align}\label{multiplicativity}
			p_{\alpha\,\beta} (x \bullet_{\alpha,\,\beta }y) =p_{\alpha}(x)\bullet_{\alpha,\,\beta}p_{\beta}(y)\;,\; q_{\alpha\,\beta}(x\bullet_{\alpha,\,\beta}y)=q_{\alpha}(x)\bullet_{\alpha,\,\beta}q_{\beta}(y)
			,\quad \text{(multiplicativity)}
		\end{align}
		\begin{align}\label{BiH-Omega-assoity}
			p_{\alpha}(x)\bullet_{\alpha,\,\beta\,\gamma}(y\bullet_{\beta,\,\gamma}z)=(x\bullet_{\alpha,\,\beta}y)\bullet_{\alpha\,\beta,\,\gamma}q_{\gamma}(z),\quad \text{(BiHom-$\Omega$-associativity)}
		\end{align}
	for all $x,
	y, z\in A,\,\alpha,\,\beta,\,\gamma \in \Omega$. The maps $p_{\alpha} $ and $q_{\alpha} $ (in this order) are called the structure maps of $A$.
\end{defn}

Obviously, if the structure maps of BiHom-$\Omega$-associative algebra $ (A,\bullet_{\alpha,\,\beta},p_{\alpha},q_{\alpha})_{\alpha,\,\beta\in \Omega} $ are the identity maps, then $A$ reduces to an $\Omega$-associative algebra.

\begin{defn}
	Let $(A, \bullet_{\alpha,\,\beta}, p_{\alpha},q_{\alpha})_{\alpha,\,\beta \in \Omega}$ and $(A', \bullet'_{\alpha,\,\beta}, p'_{\alpha},q'_{\alpha})_{\alpha,\,\beta \in \Omega}$ be two BiHom-$\Omega$-associative algebras. A family of linear maps $(f_{\alpha})_{\alpha \in \Omega}:A\rightarrow A'$ is called a $\mathbf{BiHom}$-$\mathbf{\Omega}$-$\mathbf{associative \, algebra}$ $\mathbf{ morphism}$ if $(f_{\alpha})_{\alpha\in \Omega}$ is an $\Omega$-associative algebra morphism and
	\[p'_{\alpha}\circ f_{\alpha}=f_{\alpha}\circ p_{\alpha} ,\quad q'_{\alpha}\circ f_{\alpha}=f_{\alpha}\circ q_{\alpha}, \quad\text{for all } \alpha\in \Omega.\]
\end{defn}

We give an example for BiHom-$\Omega$-associative algebras as follows.
\begin{exam}
	Let maps $ c: \Omega \times \Omega\rightarrow \mathbf{k},\;\rightthreetimes : \Omega \times \mathbf{k}\rightarrow \mathbf{k},\;\text{and}\;\leftthreetimes: \mathbf{k}\times \Omega \rightarrow \mathbf{k}$ satisfy the following conditions
	\[\alpha\,\beta\rightthreetimes 1_{k}=(\alpha\rightthreetimes 1_{k})(\beta\rightthreetimes1_{k}),\quad 1_{k}\leftthreetimes\alpha\,\beta=(1_{k}\leftthreetimes\alpha)(1_{k}\leftthreetimes\beta),\]
	\[c(\alpha,\,\beta)(1_{k}\leftthreetimes \gamma)c(\alpha\,\beta,\,\gamma)=c(\alpha,\,\beta\,\gamma)(\alpha\rightthreetimes1_{k})c(\beta,\,\gamma),\]
	for all $\alpha,\,\beta,\,\gamma\in \Omega,$ and $1_{k}$ is the unit of $\mathbf{k}$. We define the operations on the 2-dimensional unital space $\mathbf{k}\lbrace e_{1},e_{2}\rbrace$:
	\begin{align*}
		p_{\alpha}(e_{1}):=(\alpha\rightthreetimes1_{k})e_{1},\quad p_{\alpha}(e_{2}):=(\alpha\rightthreetimes 1_{k})e_{2},\\
		q_{\alpha}(e_{1}):=(1_{k}\leftthreetimes\alpha)e_{1},\quad q_{\alpha}(e_{2}):=(1_{k}\leftthreetimes\alpha)e_{1},\\
		e_{1}\bullet_{\alpha,\,\beta}e_{1}:=c(\alpha,\,\beta)e_{1},\quad
		e_{1}\bullet_{\alpha,\,\beta}e_{2}:=c(\alpha,\,\beta)e_{1},\\
		e_{2}\bullet_{\alpha,\,\beta}e_{1}:=c(\alpha,\,\beta)e_{2},\quad
		e_{2}\bullet_{\alpha,\,\beta}e_{2}:=c(\alpha,\,\beta)e_{2}.
	\end{align*}
Then $ (\mathbf{k}\lbrace e_{1},\,e_{2}\rbrace,\bullet_{\alpha,\,\beta},p_{\alpha},q_{\alpha})_{\alpha,\,\beta\in \Omega} $ is a BiHom-$\Omega$-associative algebra.
\end{exam}

Now we show that BiHom-$\Omega$-associative algebras can be obtained from the classical $\Omega$-associative algebras as follows.

\begin{prop}\label{asso-Yautwist}
	Let $(A, \cdot_{\alpha,\,\beta} )_{\alpha,\,\beta \in \Omega}$ be an $\Omega$-associative algebra. If $p_{\alpha} , q_{\alpha} :A\rightarrow A$ are two commuting $\Omega$-associative algebra morphisms and we define the multiplication on $A$ by
	\[x\bullet_{\alpha,\,\beta}y:=p_{\alpha}(x)\cdot_{\alpha,\,\beta} q_{\beta}(y), \quad \text{ for all }x, y\in A,\,\alpha,\,\beta\in \Omega.\] Then $(A, \bullet_{\alpha,\,\beta} , p_{\alpha} , q_{\alpha} )_{\alpha,\,\beta \in \Omega}$ is a BiHom-$\Omega$-associative algebra,
	called the Yau twist of $(A, \cdot_{\alpha,\,\beta} )_{\alpha,\,\beta \in\Omega}$.
\end{prop}
\begin{proof}
	First, we prove the multiplicativity property. For $ x,y,z\in A,\,\alpha,\,\beta,\,\gamma\in \Omega, $
	we have
	\begin{align*}
		p_{\alpha\,\beta}(x\bullet_{\alpha,\,\beta}y)&=p_{\alpha\,\beta}(p_{\alpha}(x)\cdot_{\alpha,\,\beta}q_{\beta}(y))\\
		&=p_{\alpha}^{2}(x)\cdot_{\alpha,\,\beta}p_{\beta}q_{\beta}(y)\hspace{1cm}\text{(by Eq.~(\ref{Omega-ass-morphism})})\\
		&=p_{\alpha}^{2}(x)\cdot_{\alpha,\,\beta}q_{\beta}p_{\beta}(y)\hspace{1cm}\text{(by $p_{\beta}\circ q_{\beta}=q_{\beta}\circ p_{\beta}$)}\\
		&=p_{\alpha}(x)\bullet_{\alpha,\,\beta}p_{\beta}(y).
	\end{align*}
Similarly, we get $ q_{\alpha\,\beta}(x\bullet_{\alpha,\,\beta}y)=q_{\alpha}(x)\bullet_{\alpha,\,\beta}q_{\beta}(y).$ Next, we prove the BiHom-$\Omega$-associativity, we have
\begin{align*}
	p_{\alpha}(x)\bullet_{\alpha,\,\beta\,\gamma}(y\bullet_{\beta,\,\gamma}z)&=p_{\alpha}(x)\bullet_{\alpha,\,\beta\,\gamma}(p_{\beta}(y)\cdot_{\beta,\,\gamma}q_{\gamma}(z))\\
	&=p_{\alpha}^{2}(x)\cdot_{\alpha,\,\beta\,\gamma}q_{\beta\,\gamma}(p_{\beta}(y)\cdot_{\beta,\,\gamma}q_{\gamma}(z))\\
	&=p_{\alpha}^{2}(x)\cdot_{\alpha,\,\beta\,\gamma}(q_{\beta}p_{\beta}(y)\cdot_{\beta,\,\gamma}q_{\gamma}^{2}(z))\hspace{1cm}\text{(by $(q_{\alpha})_{\alpha\in \Omega}$ satisfying Eq.~(\ref{Omega-ass-morphism})}\\
	&=(p_{\alpha}^{2}(x)\cdot_{\alpha,\,\beta}q_{\beta}p_{\beta}(y))\cdot_{\alpha\,\beta,\,\gamma}q_{\gamma}^{2}(z)\hspace{1cm}\text{(by Eq.~(\ref{Omega-ass-alg}))}\\
	&=p_{\alpha\,\beta}(p_{\alpha}(x)\cdot_{\alpha,\,\beta}q_{\beta}(y))\cdot_{\alpha\,\beta,\,\gamma}q_{\gamma}^{2}(z)\\
	&=p_{\alpha\,\beta}(x\bullet_{\alpha,\,\beta}y)\cdot_{\alpha\,\beta,\,\gamma}q_{\gamma}^{2}(z)\\
	&=(x\bullet_{\alpha,\,\beta}y)\bullet_{\alpha\,\beta,\,\gamma}q_{\gamma}(z).
\end{align*}
This completes the proof.
\end{proof}

The Yau twisting procedure for BiHom-$\Omega$-associative algebras admits a more general form, which we state in the next result.

\begin{prop}
	Let $(A, \bullet_{\alpha,\,\beta} ,p_{\alpha} , q_{\alpha} )_{\alpha,\,\beta \in \Omega}$ be a BiHom-$\Omega$-associative algebra. If $p'_{\alpha}, q'_{\alpha}:
	A\rightarrow A$ are two BiHom-$\Omega$-associative algebra morphisms and any two families of the maps $p_{\alpha}, q_{\alpha} ,
	p'_{\alpha}, q'_{\alpha}$ commute with each other. Define the multiplication on $A$ by
	\[x \bullet'_{\alpha,\,\beta}y:=p'_{\alpha}(x)\bullet_{\alpha,\,\beta} q'_{\beta}(y),	\]
	for all $ x, y\in A,\,\alpha,\,\beta\in \Omega,$ then $(A, \bullet'_{\alpha,\,\beta}, p_{\alpha} \circ p'_{\alpha}, q_{\alpha} \circ q'_{\alpha})_{\alpha,\,\beta \in \Omega}$
	is a BiHom-$\Omega$-associative algebra.
\end{prop}
\begin{proof}
	For $ x,y,z\in A,\,\alpha,\,\beta,\,\gamma\in \Omega ,$ the multiplicativity is obvious. Now we only need to prove the BiHom-$\Omega$-associativity.
	\begin{align*}
		p_{\alpha}\circ p_{\alpha}'(x)\bullet'_{\alpha,\,\beta\,\gamma}(y\bullet'_{\beta,\,\gamma}z)&=p_{\alpha}'p_{\alpha}p_{\alpha}'(x)\bullet_{\alpha,\,\beta\,\gamma}q_{\beta\,\gamma}'(y\bullet'_{\beta,\,\gamma}z)\\
		&=p_{\alpha}(p_{\alpha}')^{2}(x)\bullet_{\alpha,\,\beta\,\gamma}q_{\beta\,\gamma}'(p_{\beta}'(y)\bullet_{\beta,\,\gamma}q_{\gamma}'(z))\hspace{1cm}\text{(by $p_{\alpha}\circ p_{\alpha}'=p_{\alpha}\circ 'p_{\alpha}$})\\
		&=p_{\alpha}(p_{\alpha}')^{2}(x)\bullet_{\alpha,\,\beta\,\gamma}(q_{\beta}'p_{\beta}'(y)\bullet_{\beta,\,\gamma}(q_{\gamma}')^{2}(z))\\
		&\hspace{1cm}\text{(by $q_{\beta\,\gamma}'$ being a BiHom-$\Omega$-associative algebra morphism)})\\
		&=((p_{\alpha}')^{2}(x)\bullet_{\alpha,\,\beta}q_{\beta}'p_{\beta}'(y))\bullet_{\alpha\,\beta,\,\gamma}q_{\gamma}(q_{\gamma}')^{2}(z)\hspace{1cm}\text{(by Eq.~(\ref{BiH-Omega-assoity})})\\
		&=(p_{\alpha}'p_{\alpha}'(x)\bullet_{\alpha,\,\beta}p_{\beta}'q_{\beta}'(y))\bullet_{\alpha\,\beta,\,\gamma}q_{\gamma}'q_{\gamma}q_{\gamma}'(z)\\
		&\hspace{1cm}\text{(by $q_{\alpha},p_{\alpha}',q_{\alpha}'$ commuting with each other})\\
		&=p_{\alpha\,\beta}'(p_{\alpha}'(x)\bullet_{\alpha,\,\beta}q_{\beta}'(y))\bullet_{\alpha\,\beta,\,\gamma}q_{\gamma}'q_{\gamma}q_{\gamma}'(z)\\
		&\hspace{1cm}\text{(by $p_{\alpha\,\beta}'$ being a BiHom-$\Omega$-associative algebra morphism)})\\		&=p_{\alpha\,\beta}'(x\bullet'_{\alpha,\,\beta}y)\bullet_{\alpha\,\beta,\,\gamma}q_{\gamma}'q_{\gamma} q_{\gamma}'(z)\\
		&=(x\bullet'_{\alpha,\,\beta}y)\bullet'_{\alpha\,\beta,\,\gamma}q_{\gamma}\circ q_{\gamma}'(z).
	\end{align*}
This completes the proof.
\end{proof}

Let $\lambda\in \bf{k}$, a Rota-Baxter family of weight $\lambda$ on the BiHom-$\Omega$-algebra is defined as follows.

\begin{defn}
	Let $\Omega$ be a semigroup. A $\mathbf{Rota}$-$\mathbf{Baxter \, family \, of \, weight \,\lambda}$ on the BiHom-$\Omega$-algebra $ (A,\mu_{\alpha,\,\beta}, p_{\alpha}, q_{\alpha})_{\alpha,\,\beta\in \Omega} $
	is a collection of linear operators $ (R_{\alpha})_{\alpha\in \Omega} $ on $A$ such that
	\begin{align}\label{RB-Eq}
		\mu_{\alpha,\,\beta}(R_{\alpha}(x),R_{\beta}(y))=R_{\alpha\,\beta}(\mu_{\alpha,\,\beta}(R_{\alpha}(x),y)+\mu_{\alpha,\,\beta}(x,R_{\beta}(y))+\lambda \mu_{\alpha,\,\beta}(x,y)),
	\end{align}
	for all $ x,y\in A,\,\alpha,\,\beta\in \Omega $.

If further $(R_{\alpha})_{\alpha\in\Omega}$ commute with the structure maps, then $ (A,\mu_{\alpha,\,\beta}, p_{\alpha}, q_{\alpha}, R_{\alpha})_{\alpha,\,\beta\in \Omega} $ is called a $\bf{Rota\text{-}Baxter \, family \, BiHom\text{-}\Omega\text{-}algebra \, of \, weight \, \lambda}$.
\end{defn}

 The main purpose of the following result is to show that a new BiHom-$\Omega$-associative algebra can be constructed by the Rota-Baxter family of weight $\lambda$ on the BiHom-$\Omega$-associative algebra.

\begin{theorem}\label{asso-RBO-asso}
	Let $(A, \bullet_{\alpha,\,\beta} ,p_{\alpha}, q_{\alpha})_{\alpha,\,\beta\in \Omega}$ be a BiHom-$\Omega$-associative algebra. If $ (R_{\alpha})_{\alpha\in \Omega} $ is a Rota-Baxter family of weight $\lambda$ on $A$ satisfying
\[R_{\alpha}\circ p_{\alpha} =p_{\alpha} \circ R_{\alpha} \; \text{and} \; R_{\alpha}\circ q_{\alpha} =q_{\alpha} \circ R_{\alpha}.\]
Define a new operation on $A$ by
	\[x\star_{\alpha ,\,\beta}y:=x\bullet_{\alpha,\,\beta }R_{\beta}(y)+R_{\alpha}(x)\bullet_{\alpha,\,\beta }y+\lambda x\bullet_{\alpha,\,\beta}y,\]
	for all $x, y\in A,\,\alpha,\,\beta\in \Omega,$ then $(A, \star_{\alpha ,\,\beta}, p_{\alpha}, q_{\alpha})_{\alpha,\,\beta\in \Omega}$ is a BiHom-$\Omega$-associative algebra.
\end{theorem}
\begin{proof}
	For any $ x,y,z\in A,$ $\alpha,\,\beta,\,\gamma \in \Omega,$ we have $ p_{\alpha}\circ q_{\alpha}=q_{\alpha}\circ p_{\alpha} $ and
	\begin{align*}
		p_{\alpha\,\beta}(x \star_{\alpha ,\, \beta}y)&=p_{\alpha\,\beta}(x\bullet_{\alpha,\,\beta }R_{\beta}(y)+R_{\alpha}(x)\bullet_{\alpha,\,\beta }y+\lambda x\bullet_{\alpha,\,\beta}y)\\
		&=p_{\alpha\,\beta}(x\bullet_{\alpha,\,\beta}R_{\beta}(y))+p_{\alpha\,\beta}(R_{\alpha}(x)\bullet_{\alpha,\,\beta}y)+\lambda p_{\alpha\,\beta}(x\bullet_{\alpha,\,\beta}y)\\
		%&\hspace{1cm}\text{(by $(p_{\alpha})_{\alpha\in \Omega}$ being a family of linear maps)}\\
		&=p_{\alpha}(x)\bullet_{\alpha,\,\beta}p_{\beta}R_{\beta}(y)+p_{\alpha}R_{\alpha}(x)\bullet_{\alpha,\,\beta}p_{\beta}(y)+\lambda p_{\alpha}(x)\bullet_{\alpha,\,\beta}p_{\beta}(y)\\
		&\hspace{1cm}\text{(by Eq.~(\ref{multiplicativity}))}\\
		&=p_{\alpha}(x)\bullet_{\alpha,\,\beta}R_{\beta}(p_{\beta}(y))+R_{\alpha}(p_{\alpha}(x))\bullet_{\alpha,\,\beta}p_{\beta}(y)+\lambda p_{\alpha}(x)\bullet_{\alpha,\,\beta}p_{\beta}(y)\\
		&=p_{\alpha}(x)\star_{\alpha ,\, \beta}p_{\beta}(y).
	\end{align*}
	Similarly, we get $q_{\alpha\,\beta}(x \star_{\alpha ,\,\beta} y )=q_{\alpha}(x)\star_{\alpha ,\, \beta}q_{\beta}(y).$ On the one hand, we have
	\begin{align*}
		(x\star_{\alpha ,\, \beta}y)\star_{\alpha\beta,\,\gamma}q_{\gamma}(z)=&(x\bullet_{\alpha,\,\beta}R_{\beta}(y)+R_{\alpha}(x)\bullet_{\alpha,\,\beta}y+\lambda x\bullet_{\alpha,\,\beta}y)\star_{\alpha\beta,\,\gamma}q_{\gamma}(z)\\
=&(x\bullet_{\alpha,\,\beta}R_{\beta}(y)+R_{\alpha}(x)\bullet_{\alpha,\,\beta}y+\lambda x\bullet_{\alpha,\,\beta}y)\bullet_{\alpha\beta,\,\gamma}R_{\gamma}q_{\gamma}(z)\\
		&+R_{\alpha\,\beta}(x\bullet_{\alpha,\,\beta}R_{\beta}(y)+R_{\alpha}(x)\bullet_{\alpha,\,\beta}y+\lambda x\bullet_{\alpha,\,\beta}y)\bullet_{\alpha\beta,\,\gamma}q_{\gamma}(z)\\
		&+\lambda(x\bullet_{\alpha,\,\beta}R_{\beta}(y)+R_{\alpha}(x)\bullet_{\alpha,\,\beta}y+\lambda x\bullet_{\alpha,\,\beta}y)\bullet_{\alpha\beta,\,\gamma}q_{\gamma}(z)\\
		=&(x\bullet_{\alpha,\,\beta}R_{\beta}(y))\bullet_{\alpha\beta,\,\gamma}q_{\gamma}R_{\gamma}(z)+(R_{\alpha}(x)\bullet_{\alpha,\,\beta}y)\bullet_{\alpha\beta,\,\gamma}q_{\gamma}R_{\gamma}(z)\\
		&+\lambda(x\bullet_{\alpha,\,\beta}y)\bullet_{\alpha\beta,\,\gamma}q_{\gamma}R_{\gamma}(z)+(R_{\alpha}(x)\bullet_{\alpha,\,\beta}R_{\beta}(y))\bullet_{\alpha\beta,\,\gamma}q_{\gamma}(z)\\
		&+\lambda(x\bullet_{\alpha,\,\beta}R_{\beta}(y))\bullet_{\alpha\beta,\,\gamma}q_{\gamma}(z)+\lambda(R_{\alpha}(x)\bullet_{\alpha,\,\beta}y)\bullet_{\alpha\beta,\,\gamma}q_{\gamma}(z)\\
		&+\lambda^{2}(x\bullet_{\alpha,\,\beta}y)\bullet_{\alpha\beta,\,\gamma}q_{\gamma}(z)\hspace{1cm}\text{(by Eq.~(\ref{RB-Eq}) and $ R_{\gamma}\circ q_{\gamma}=q_{\gamma}\circ R_{\gamma}$)}\\
		=&p_{\alpha}(x)\bullet_{\alpha,\,\beta\gamma}(R_{\beta}(y)\bullet_{\beta,\,\gamma}R_{\gamma}(z))+p_{\alpha}R_{\alpha}(x)\bullet_{\alpha,\,\beta\gamma}(y\bullet_{\beta,\,\gamma}R_{\gamma}(z))\\
		&+\lambda p_{\alpha}(x)\bullet_{\alpha,\,\beta\gamma}(y \bullet_{\beta,\,\gamma}R_{\gamma}(z))+p_{\alpha}R_{\alpha}(x)\bullet_{\alpha,\,\beta\gamma}(R_{\beta}(y)\bullet_{\beta,\,\gamma}z)\\
		&+\lambda p_{\alpha}(x)\bullet_{\alpha,\,\beta\gamma}(R_{\beta}(y)\bullet_{\beta,\,\gamma}z)+\lambda p_{\alpha}R_{\alpha}(x)\bullet_{\alpha,\,\beta\gamma}(y\bullet_{\beta,\,\gamma}z)\\
		&+\lambda^{2}p_{\alpha}(x)\bullet_{\alpha,\,\beta\gamma}(y\bullet_{\beta,\,\gamma}z).\hspace{0.3cm} \text{(by Eq.~(\ref{BiH-Omega-assoity}))}
	\end{align*}
	On the other hand, we have
	\begin{align*}
		p_{\alpha}(x)\star_{\alpha,\,\beta\gamma}(y \star_{\beta,\,\gamma}z)=&p_{\alpha}(x)\star_{\alpha,\,\beta\,\gamma}(y\bullet_{\beta,\,\gamma}R_{\gamma}(z)+R_{\beta}(y)\bullet_{\beta,\,\gamma}z+\lambda y\bullet_{\beta,\,\gamma}z)\\
		=&p_{\alpha}(x)\bullet_{\alpha,\,\beta\gamma}R_{\beta \gamma}(y\bullet_{\beta,\,\gamma}R_{\gamma}(z)+R_{\beta}(y)\bullet_{\beta,\,\gamma}z+\lambda y\bullet_{\beta,\,\gamma}z)\\
		&+R_{\alpha}p_{\alpha}(x)\bullet_{\alpha,\,\beta\gamma}(y\bullet_{\beta,\,\gamma}R_{\gamma}(z)+R_{\beta}(y)\bullet_{\beta,\,\gamma}z+\lambda y\bullet_{\beta,\,\gamma}z)\\
		&+\lambda p_{\alpha}(x)\bullet_{\alpha,\,\beta\gamma}(y\bullet_{\beta,\,\gamma}R_{\gamma}(z)+R_{\beta}(y)\bullet_{\beta,\,\gamma}z+\lambda y\bullet_{\beta,\,\gamma}z)\\
		=&p_{\alpha}(x)\bullet_{\alpha,\,\beta\gamma}(R_{\beta}(y)\bullet_{\beta,\,\gamma}R_{\gamma}(z))+p_{\alpha}R_{\alpha}(x)\bullet_{\alpha,\,\beta\gamma}(y\bullet_{\beta,\,\gamma}R_{\gamma}(z))\\
		&+p_{\alpha}R_{\alpha}(x)\bullet_{\alpha,\,\beta\gamma}(R_{\beta}(y)\bullet_{\beta,\,\gamma}z)+\lambda p_{\alpha}R_{\alpha}(x)\bullet_{\alpha,\,\beta\gamma}(y\bullet_{\beta,\,\gamma}z)\\
		&+\lambda p_{\alpha}(x)\bullet_{\alpha,\,\beta\gamma}(y\bullet_{\beta,\,\gamma}R_{\gamma}(z))+\lambda p_{\alpha}(x)\bullet_{\alpha,\,\beta\gamma}(R_{\beta}(y)\bullet_{\beta,\,\gamma}z)\\
		&+\lambda^{2}p_{\alpha}(x)\bullet_{\alpha,\,\beta\gamma}(y\bullet_{\beta,\,\gamma}z)\hspace{1cm} \text{(by Eq.~(\ref{RB-Eq})} \text{ and } R_{\alpha}\circ q_{\alpha}=q_{\alpha}\circ R_{\alpha} ).
	\end{align*}
	By comparing the items of both sides, we get $ (x\star_{\alpha ,\, \beta}y)\star_{\alpha\beta,\,\gamma}q_{\gamma}(z)=p_{\alpha}(x)\star_{\alpha,\,\beta\gamma}(y \star_{\beta,\,\gamma}z). $
\end{proof}

\begin{remark}
	In Theorem \ref{asso-RBO-asso}, we notice that $(R_{\alpha})_{\alpha\in \Omega}$ is also a Rota-Baxter family of weight $\lambda $ for the BiHom-$\Omega$-associative algebra
	$(A, \star_{\alpha , \,\beta}, p_{\alpha} ,q_{\alpha})_{\alpha,\,\beta\in \Omega}$.
\end{remark}

%%%%%%%%%%%%%%%%%%%%
\subsection{BiHom-$\Omega$-dendriform algebras}\label{sub2.2}
%%%%%%%%%%%%%%%%%%%%

In this subsection, we mainly introduce the BiHom-$\Omega$-dendriform algebra, which is a generalization of $\Omega$-dendriform algebras~\cite{Aguiar}. Further, we give the relationship between BiHom-$\Omega$-dendriform algebras and BiHom-$\Omega$-associative algebras. For this, let us first briefly recall the definition of $\Omega$-dendriform algebras.%We begin by recalling the $\Omega$-dendriform algebras.

\begin{defn}\cite{Aguiar}
	An $\mathbf{\Omega}$-$\mathbf{dendriform \, algebra}$ $ (A,\prec_{\alpha,\,\beta},\succ_{\alpha,\,\beta})_{\alpha,\,\beta\in \Omega} $ is a vector space $A$ equipped with two families of bilinear maps $ (\prec_{\alpha,\,\beta},\,\succ_{\alpha,\,\beta}: A\times A\rightarrow A)_{\alpha,\,\beta\in \Omega} $ such that
	\begin{align}
		(x\prec_{\alpha,\,\beta}y)\prec_{\alpha\,\beta,\,\gamma}z
=&x\prec_{\alpha,\,\beta\,\gamma}(y\prec_{\beta,\,\gamma}z
+y\succ_{\beta,\,\gamma}z),\label{Ome-den-1}\\
		(x\succ_{\alpha,\,\beta}y)\prec_{\alpha\,\beta,\,\gamma}z
=&x\succ_{\alpha,\,\beta\,\gamma}(y\prec_{\beta,\,\gamma}z),\label{Ome-den-2}\\
		(x\prec_{\alpha,\,\beta}y+x\succ_{\alpha,\,\beta}y)\succ_{\alpha\,\beta,\,\gamma}z=&x \succ_{\alpha,\,\beta\,\gamma}(y\succ_{\beta,\,\gamma}z)\label{Ome-den-3}
	\end{align}
for all $ x,y,z\in A,\,\alpha,\,\beta,\,\gamma\in \Omega.$
\end{defn}
\begin{defn}
Let $ (A,\prec_{\alpha,\,\beta},\succ_{\alpha,\,\beta})_{\alpha,\,\beta\in \Omega} $ and  $ (A',\prec_{\alpha,\,\beta}',\succ_{\alpha,\,\beta}')_{\alpha,\,\beta\in \Omega} $ be two $\Omega$-dendriform algebras. A family of linear maps $ (f_{\alpha})_{\alpha\in\Omega}: A\rightarrow A' $ is called an $\mathbf{\Omega}$-$\mathbf{dendriform \, algebra \, morphism}$ if
\begin{align}\label{Omega-den-morphism}
	f_{\alpha\,\beta}(x\prec_{\alpha,\,\beta}y)=f_{\alpha}(x)\prec'_{\alpha,\,\beta}f_{\beta}(y),\quad f_{\alpha\,\beta}(x\succ_{\alpha,\,\beta}y)=f_{\alpha}(x)\succ'_{\alpha,\,\beta}f_{\beta}(y),
\end{align}
for all $ x,y\in A,\,\alpha,\,\beta\in \Omega. $

\end{defn}
% As is known to all, the concept of (tri)dendriform family algebra has been introduced in \cite{tri-family}.
 Now, we generalize the above definitions of the $\Omega$-dendriform algebra to BiHom-verision.

\begin{defn}\label{BiHom-den-fam}
	A $\mathbf{BiHom}$-$\mathbf{\Omega}$-$\mathbf{dendriform \, algebra}$ $(A, \prec_{\alpha,\,\beta} , \succ_{\alpha,\,\beta} , p_{\alpha} , q_{\alpha} )_{\alpha,\,\beta\in \Omega}$ is a vector space $A$ equipped with two families of bilinear maps $( \prec_{\alpha,\,\beta} , \succ_{\alpha,\,\beta} :A\otimes A\rightarrow A )_{\alpha,\,\beta\in \Omega}$ and two commuting families of linear maps $p_{\alpha} , q_{\alpha} :A\rightarrow A$ such that
%		p_{\alpha} \circ q_{\alpha} =q_{\alpha} \circ p_{\alpha}, \label{BiHomdend1}
	\begin{align}
		&p_{\alpha\,\beta}(x\prec_{\alpha,\,\beta} y)=p_{\alpha} (x)\prec_{\alpha,\,\beta} p_{\beta} (y), ~
		p_{\alpha\,\beta}(x\succ_{\alpha,\,\beta} y)=p_{\alpha} (x)\succ_{\alpha,\,\beta} p_{\beta}(y), \label{BiHomdend2} \\
		&q_{\alpha\,\beta}(x\prec_{\alpha,\,\beta} y)=q_{\alpha} (x)\prec_{\alpha,\,\beta} q_{\beta}(y), ~
		q_{\alpha\,\beta}(x\succ_{\alpha,\,\beta} y)=q_{\alpha} (x)\succ_{\alpha,\,\beta} q_{\beta}(y), \label{BiHomdend3} \\
		&(x\prec_{\alpha,\,\beta} y)\prec_{\alpha\,\beta,\,\gamma} q_{\gamma}(z)=p_{\alpha} (x)\prec_{\alpha,\,\beta\,\gamma} (y\prec_{\beta,\,\gamma} z+y\succ_{\beta,\,\gamma} z),  \label{BiHomdend4} \\
		&(x\succ_{\alpha,\,\beta} y)\prec_{\alpha\,\beta,\,\gamma} q_{\gamma}(z)=p_{\alpha} (x)\succ_{\alpha,\,\beta\,\gamma} (y\prec_{\beta,\,\gamma} z), \label{BiHomdend5} \\
		&p_{\alpha}(x)\succ_{\alpha,\,\beta\,\gamma} (y\succ_{\beta,\,\gamma} z)=(x\prec_{\alpha,\,\beta} y+x\succ _{\alpha,\,\beta}y)\succ_{\alpha\,\beta,\,\gamma} q_{\gamma}(z), \label{BiHomdend6}
	\end{align}
 for all $x,\, y,\, z\in A,\,\alpha,\,\beta \in \Omega.$ The maps $p_{\alpha}$ and $q_{\alpha} $ (in this order) are called the structure maps
	of $A$.
\end{defn}
\begin{defn}
	Let $ (A,\prec_{\alpha,\,\beta},\succ_{\alpha,\,\beta},p_{\alpha},q_{\alpha})_{\alpha,\,\beta\in \Omega} $ and $ (A',\prec'_{\alpha,\,\beta},\succ'_{\alpha,\,\beta},p'_{\alpha},q'_{\alpha})_{\alpha,\,\beta\in \Omega} $ be two BiHom-$\Omega$-dendriform algebras. A family of linear maps $ (f_{\alpha})_{\alpha\in \Omega}:A\rightarrow A' $ is called a $\mathbf{BiHom}$-$\mathbf{\Omega}$-$\mathbf{dendriform \, algebra \, morphism}$ if $ (f_{\alpha})_{\alpha\in \Omega} $ is an $\Omega$-dendriform algebra morphism and
	\[f_{\alpha}\circ p_{\alpha}=p'_{\alpha}\circ f_{\alpha},\quad f_{\alpha}\circ q_{\alpha}=q'_{\alpha}\circ f_{\alpha},\quad \text{ for all } \alpha\in \Omega.\]
\end{defn}

Inspired by~\cite[Proposition 2.2]{BiHom-dendri}, here we characterize the Yau twisting procedure for BiHom-$\Omega$-dendriform algebras as follows.

\begin{prop}\label{dendri-Yautwist}
	Let $ (A,\prec_{\alpha,\,\beta},\succ_{\alpha,\,\beta})_{\alpha,\,\beta\in \Omega} $ be an $\Omega$-dendriform algebra. If $ p_{\alpha},q_{\alpha}: A \rightarrow A $ are two commuting $\Omega$-dendriform algebra morphisms. Define the operations by
	\[x\prec_{\alpha,\,\beta}'y:=p_{\alpha}(x)\prec_{\alpha,\,\beta}q_{\beta}(y),\quad x\succ_{\alpha,\,\beta}'y:=p_{\alpha}(x)\succ_{\alpha,\,\beta}q_{\beta}(y),\]
	for all $ x,y\in A,\,\alpha,\,\beta\in \Omega, $ then $ (A,\prec_{\alpha,\,\beta}',\succ_{\alpha,\,\beta}',p_{\alpha},q_{\alpha})_{\alpha,\,\beta\in \Omega} $ is a BiHom-$\Omega$-dendriform algebra, called the Yau twist of $ (A,\prec_{\alpha,\,\beta},\succ_{\alpha,\,\beta})_{\alpha,\,\beta\in \Omega} $.
\end{prop}
\begin{proof}
	For $ x,y,z\in A,\,\alpha,\,\beta,\,\gamma\in \Omega, $ we need to prove the structure maps satisfying Eqs.~(\ref{BiHomdend2})-(\ref{BiHomdend3}). Here we prove
	\begin{align*}
		p_{\alpha\,\beta}(x\prec_{\alpha,\,\beta}'y)=&p_{\alpha\,\beta}(p_{\alpha}(x)\prec_{\alpha,\,\beta}q_{\beta}(y))\\
		=&p_{\alpha}^{2}(x)\prec_{\alpha,\,\beta}p_{\beta}q_{\beta}(y)\hspace{1cm}\text{(by $p_{\alpha\,\beta}$ satisfying Eq.~(\ref{Omega-den-morphism})})\\
		=&p_{\alpha}p_{\alpha}(x)\prec_{\alpha,\,\beta}q_{\beta}p_{\beta}(y)\hspace{1cm}\text{(by $p_{\beta}\circ q_{\beta}=q_{\beta}\circ p_{\beta}$})\\
		=&p_{\alpha}(x)\prec_{\alpha,\,\beta}'p_{\beta}(y).
	\end{align*}
Other relations in Eqs.~(\ref{BiHomdend2})-(\ref{BiHomdend3}) are similar to prove.
%Similarly, we get \[p_{\alpha\,\beta}(x\succ_{\alpha,\,\beta}' y)=p_{\alpha} (x)\succ_{\alpha,\,\beta}' p_{\beta}(y),\]
%\[q_{\alpha\,\beta}(x\prec_{\alpha,\,\beta}' y)=q_{\alpha} (x)\prec_{\alpha,\,\beta}'q_{\beta}(y), \]
%\[q_{\alpha\,\beta}(x\succ_{\alpha,\,\beta}' y)=q_{\alpha} (x)\succ_{\alpha,\,\beta}' q_{\beta}(y).\]
 Next, we only need to prove Eq.~(\ref{BiHomdend4}) and Eqs.~(\ref{BiHomdend5})-(\ref{BiHomdend6}) are similar to prove by using Eqs.~\eqref{Ome-den-2}-\eqref{Ome-den-3}.
	\begin{align*}
	(x\prec_{\alpha,\,\beta}'y)\prec_{\alpha\,\beta,\,\gamma}'q_{\gamma}(z)=&p_{\alpha\,\beta}(x\prec_{\alpha,\,\beta}'y)\prec_{\alpha\,\beta,\,\gamma}q_{\gamma}^{2}(z)\\
	=&p_{\alpha\,\beta}(p_{\alpha}(x)\prec_{\alpha,\,\beta}q_{\beta}(y))\prec_{\alpha\,\beta,\,\gamma}q_{\gamma}^{2}(z)\\
	=&(p_{\alpha}^{2}(x)\prec_{\alpha,\,\beta}p_{\beta}q_{\beta}(y))\prec_{\alpha\,\beta,\,\gamma}q_{\gamma}^{2}(z)\hspace{1cm}\text{(by $p_{\alpha\,\beta}$ satisfying Eq.~(\ref{Omega-den-morphism})})\\
	=&p_{\alpha}^{2}(x)\prec_{\alpha,\,\beta\,\gamma}(p_{\beta}q_{\beta}(y)\prec_{\beta,\,\gamma}q_{\gamma}^{2}(z)+p_{\beta}q_{\beta}(y)\succ_{\beta,\,\gamma}q_{\gamma}^{2}(z))\hspace{0.5cm}\text{(by Eq.~(\ref{Ome-den-1})})\\
	=&p_{\alpha}^{2}(x)\prec_{\alpha,\,\beta\,\gamma}(q_{\beta}p_{\beta}(y)\prec_{\beta,\,\gamma}q_{\gamma}^{2}(z)+q_{\beta}p_{\beta}(y)\succ_{\beta,\,\gamma}q_{\gamma}^{2}(z))\\
	&\hspace{1cm}\text{(by $p_{\beta}\circ q_{\beta}=q_{\beta}\circ p_{\beta}$})\\
	=&p_{\alpha}^{2}(x)\prec_{\alpha,\,\beta\,\gamma}q_{\beta\,\gamma}(p_{\beta}(y)\prec_{\beta,\,\gamma}q_{\gamma}(z)+p_{\beta}(y)\succ_{\beta,\,\gamma}q_{\gamma}(z))\\
	&\hspace{1cm}\text{(by $q_{\beta\,\gamma}$ satisfying Eq.~(\ref{Omega-den-morphism})})\\
	=&p_{\alpha}^{2}(x)\prec_{\alpha,\,\beta\,\gamma}q_{\beta\,\gamma}(y\prec_{\beta,\,\gamma}'z+y\succ_{\beta,\,\gamma}'z)\\
	=&p_{\alpha}(x)\prec_{\alpha,\,\beta\,\gamma}'(y\prec_{\beta,\,\gamma}'z+y\succ_{\beta,\,\gamma}'z).
	\end{align*}

%Similarly, we get \[(x\succ_{\alpha,\,\beta}' y)\prec_{\alpha\,\beta,\,\gamma}' q_{\gamma}(z)=p_{\alpha} (x)\succ_{\alpha,\,\beta\,\gamma}' (y\prec_{\beta,\,\gamma}' z),\]
%\[p_{\alpha}(x)\succ_{\alpha,\,\beta\,\gamma}' (y\succ_{\beta,\,\gamma}' z)=(x\prec_{\alpha,\,\beta}' y+x\succ _{\alpha,\,\beta}'y)\succ_{\alpha\,\beta,\,\gamma} 'q_{\gamma}(z).\]
This completes the proof.
\end{proof}

%Next, we have the $\Omega$-version of \cite{BiHom-dendri} Remark 2.3, which is a more general form of Yau-twisting.

The following result is a more general of the Yau twisting procedure for BiHom-$\Omega$-dendriform algebras and the proof is similar to Proposition~\ref{dendri-Yautwist}.

\begin{prop}\label{bihomden-bihomden}
	If $ (A,\prec_{\alpha,\,\beta},\succ_{\alpha,\,\beta},p_{\alpha},q_{\alpha})_{\alpha,\,\beta\in \Omega} $ is a BiHom-$\Omega$-dendriform algebra with $ p'_{\alpha},q'_{\alpha}:A\rightarrow A $ are two BiHom-$\Omega$-dendriform algebra morphisms and any two families of the maps $p_{\alpha},q_{\alpha},p'_{\alpha},q'_{\alpha}$ commute with each other. Define the multiplications on $A$ by
	\[x\prec'_{\alpha,\,\beta}y:=p'_{\alpha}(x)\prec_{\alpha,\,\beta}q'_{\beta}(y),\;\;\;x\succ'_{\alpha,\,\beta}y:=p'_{\alpha}(x)\succ_{\alpha,\,\beta}q'_{\beta}(y),\]
	for all $x,y\in A,\,\alpha,\,\beta\in\Omega,$ then $ (A,\prec'_{\alpha,\,\beta},\succ'_{\alpha,\,\beta},p_{\alpha}\circ p'_{\alpha},q_{\alpha}\circ q'_{\alpha})_{\alpha,\,\beta\in \Omega} $ is a BiHom-$\Omega$-dendriform algebra.
\end{prop}

Next, we will introduce the relationship between BiHom-$\Omega$-associative algebras and BiHom-$\Omega$-dendriform algebras.

\begin{prop}\label{denf-to-Omegaasso}
		Let $ (A, \prec_{\alpha,\,\beta}, \succ_{\alpha,\,\beta}, p_{\alpha}, q_{\alpha})_{\alpha,\,\beta\in \Omega} $ be a BiHom-$\Omega$-dendriform algebra. If we define the multiplication by
		\[x \bullet_{\alpha,\,\beta} y := x \prec_{\alpha,\,\beta}y +x\succ_{\alpha,\,\beta}y, \]
 for all $x, y \in A,\,\alpha,\,\beta \in \Omega,$ then $ (A, \bullet_{\alpha,\,\beta}, p_{\alpha},q_{\alpha})_{\alpha,\,\beta\in\Omega} $ is a BiHom-$\Omega$-associative algebra.

\end{prop}
\begin{proof}
For any $ x,\,y,\,z\in A,\,\alpha,\,\beta,\,\gamma\in\Omega,$ owing to the commutativity, we get $ p_{\alpha}\circ q_{\alpha}=q_{\alpha}\circ p_{\alpha} $ easily.
First, we prove the multiplicativity in Definition \ref{BiHom-Omega-asso}, we have
 \begin{align*}
 	p_{\alpha\,\beta}(x\bullet_{\alpha,\,\beta}y)=&p_{\alpha\,\beta}(x\prec_{\alpha,\,\beta}y+x\succ_{\alpha,\,\beta}y)\\
 	=&p_{\alpha\,\beta}(x\prec_{\alpha,\,\beta}y)+p_{\alpha\,\beta}(x\succ_{\alpha,\,\beta}y)\hspace{1cm} \text{(by $(p_{\alpha})_{\alpha\in \Omega}$ being a family of linear maps)}\\
 	=&p_{\alpha}(x)\prec_{\alpha,\,\beta}p_{\beta}(y)+p_{\alpha}(x)\succ_{\alpha,\,\beta}p_{\beta}(y)\hspace{1cm} \text{ (by Eq.~(\ref{BiHomdend2})})\\
 	=&p_{\alpha}(x)\bullet_{\alpha,\,\beta}p_{\beta}(y).
 \end{align*}
Similarly, we have
$q_{\alpha\,\beta}(x\bullet_{\alpha,\,\beta}y)=q_{\alpha}(x)\bullet_{\alpha,\,\beta}q_{\beta}(y).$
Next, we prove the BiHom-$\Omega$-associativity in Definition \ref{BiHom-Omega-asso}, we have
\begin{align*}
	p_{\alpha}(x)\bullet_{\alpha,\,\beta\gamma}(y\bullet_{\beta,\,\gamma}z)=&p_{\alpha}(x)\bullet_{\alpha,\,\beta\gamma}(y\prec_{\beta,\,\gamma}z+y\succ_{\beta,\,\gamma}z)\\
	=&p_{\alpha}(x)\prec_{\alpha,\,\beta\,\gamma}(y\prec_{\beta,\,\gamma}z+y\succ_{\beta,\,\gamma}z)+p_{\alpha}(x)\succ_{\alpha,\,\beta\,\gamma}(y\prec_{\beta,\,\gamma}z+y\succ_{\beta,\,\gamma}z)\\
	=&(x\prec_{\alpha,\,\beta}y)\prec_{\alpha\,\beta,\,\gamma}q_{\gamma}(z)+(x\succ_{\alpha,\,\beta}y)\prec_{\alpha\,\beta,\,\gamma}q_{\gamma}(z)\\
	&+(x\prec_{\alpha,\,\beta}y+x\succ_{\alpha,\,\beta}y)\succ_{\alpha\,\beta,\,\gamma}q_{\gamma}(z)\hspace{1cm} \text{(by Eqs.~(\ref{BiHomdend4})-(\ref{BiHomdend6}))}\\
	=&(x\prec_{\alpha,\,\beta}y+x\succ_{\alpha,\,\beta}y)\prec_{\alpha\,\beta,\,\gamma}q_{\gamma}(z)+(x\prec_{\alpha,\,\beta}y+x\succ_{\alpha,\,\beta}y)\succ_{\alpha\,\beta,\,\gamma}q_{\gamma}(z)\\
	=&(x\prec_{\alpha,\,\beta}y+x\succ_{\alpha,\,\beta}y)\bullet_{\alpha\beta,\,\gamma}q_{\gamma}(z)\\
	=&(x\bullet_{\alpha,\,\beta}y)\bullet_{\alpha\beta,\,\gamma}q_{\gamma}(z).
\end{align*}
This completes the proof.
\end{proof}

 It's well known that Rota-Baxter algebras induce dendriform family algebras, now we generalize this property to BiHom-$\Omega$ version.

\begin{prop}\label{asso-den}
	Let $ (A,\bullet_{\alpha,\,\beta},p_{\alpha},q_{\alpha})_{\alpha,\,\beta\in \Omega} $ be a BiHom-$\Omega$-associative algebra.
	\begin{enumerate}
		\item\label{asso-RB-den-1} If $ (R_{\alpha})_{\alpha\in \Omega} $ is a Rota-Baxter family of weight 0 on $A$ and
		\[p_{\alpha}\circ R_{\alpha}=R_{\alpha}\circ p_{\alpha},\quad q_{\alpha}\circ R_{\alpha}=R_{\alpha}\circ q_{\alpha}.\] Define the operations by
		\[x\prec_{\alpha,\,\beta}y:=x\bullet_{\alpha,\,\beta} R_{\beta}(y),\quad x\succ_{\alpha,\,\beta}y:=R_{\alpha}(x)\bullet_{\alpha,\,\beta}y,\]
		for all $ x,y\in A,\,\alpha,\,\beta\in \Omega, $ then $ (A,\prec_{\alpha,\,\beta},\succ_{\alpha,\,\beta},p_{\alpha},q_{\alpha})_{\alpha,\,\beta\in \Omega} $ is a BiHom-$\Omega$-dendriform algebra.
		\item\label{asso-RB-den-2} If $ (R_{\alpha})_{\alpha\in \Omega} $ is a Rota-Baxter family of weight $\lambda$ on $A$ and
		\[p_{\alpha}\circ R_{\alpha}=R_{\alpha}\circ p_{\alpha},\quad q_{\alpha}\circ R_{\alpha}=R_{\alpha}\circ q_{\alpha}.\] Define the operations by
		\[x\prec'_{\alpha,\,\beta}y:=x\bullet_{\alpha,\,\beta} R_{\beta}(y)+\lambda x\bullet_{\alpha,\,\beta}y,\quad x\succ'_{\alpha,\,\beta}y:=R_{\alpha}(x)\bullet_{\alpha,\,\beta}y,\]
		for all $ x,y\in A,\,\alpha,\,\beta\in \Omega, $ then $ (A,\prec'_{\alpha,\,\beta},\succ'_{\alpha,\,\beta},p_{\alpha},q_{\alpha})_{\alpha,\,\beta\in \Omega} $ is a BiHom-$\Omega$-dendriform algebra.
	\end{enumerate}
\end{prop}
\begin{proof}
	We just prove Item \ref{asso-RB-den-2}. Item \ref{asso-RB-den-1} can be proved in the same way. To verify the relations in Eqs.~(\ref{BiHomdend2})-(\ref{BiHomdend6}), for any $ x,y,z\in A,\,\alpha,\,\beta,\,\gamma\in \Omega,$ first of all, we have
	\begin{align*}
		p_{\alpha\,\beta}(x\prec'_{\alpha,\,\beta}y)=&p_{\alpha\,\beta}(x\bullet_{\alpha,\,\beta}R_{\beta}(y)+\lambda x\bullet_{\alpha,\,\beta}y)\\
		=&p_{\alpha}(x)\bullet_{\alpha,\,\beta}p_{\beta}R_{\beta}(y)+\lambda p_{\alpha}(x)\bullet_{\alpha,\,\beta}p_{\beta}(y)\hspace{1cm}\text{(by Eq.~(\ref{multiplicativity}))}\\
		=&p_{\alpha}(x)\bullet_{\alpha,\,\beta}R_{\beta}p_{\beta}(y)+\lambda p_{\alpha}(x)\bullet_{\alpha,\,\beta}p_{\beta}(y)\hspace{1cm}\text{(by $p_{\beta}\circ R_{\beta}=R_{\beta}\circ p_{\beta}$})\\
		=&p_{\alpha}(x)\prec'_{\alpha,\,\beta}p_{\beta}(y),
	\end{align*}
Other relations in Eqs.~(\ref{BiHomdend2})-(\ref{BiHomdend3}) can be similarly proved.
%Similarly, we get
%\[p_{\alpha\,\beta}(x\succ'_{\alpha,\,\beta}y)=p_{\alpha}(x)\succ'_{\alpha,\,\beta}p_{\beta}(y),\]
%\[q_{\alpha\,\beta}(x\prec'_{\alpha,\,\beta}y)=q_{\alpha}(x)\prec'_{\alpha,\,\beta}q_{\beta}(y),\]
%\[q_{\alpha\,\beta}(x\succ'_{\alpha,\,\beta}y)=q_{\alpha}(x)\succ'_{\alpha,\,\beta}q_{\beta}(y).\]
Next, for Eqs.~(\ref{BiHomdend4})-(\ref{BiHomdend6}), we compute
\begin{align*}
	(x\prec'_{\alpha,\,\beta}y)\prec'_{\alpha\,\beta,\,\gamma}q_{\gamma}(z)=&(x\bullet_{\alpha,\,\beta}R_{\beta}(y)+\lambda x\bullet_{\alpha,\,\beta}y)\prec'_{\alpha\,\beta,\,\gamma}q_{\gamma}(z)\\
	=&(x\bullet_{\alpha,\,\beta}R_{\beta}(y)+\lambda x\bullet_{\alpha,\,\beta}y)\bullet_{\alpha\,\beta,\,\gamma}R_{\gamma}q_{\gamma}(z)+\lambda (x\bullet_{\alpha,\,\beta}R_{\beta}(y)\\
	&+\lambda x\bullet_{\alpha,\,\beta}y)\bullet_{\alpha\,\beta,\,\gamma}q_{\gamma}(z)\\
	=&(x\bullet_{\alpha,\,\beta}R_{\beta}(y))\bullet_{\alpha\,\beta,\,\gamma}q_{\gamma}R_{\gamma}(z)+\lambda (x\bullet_{\alpha,\,\beta}y)\bullet_{\alpha\,\beta,\,\gamma}q_{\gamma}R_{\gamma}(z)\\
	&+\lambda (x\bullet_{\alpha,\,\beta}R_{\beta}(y))\bullet_{\alpha\,\beta,\,\gamma}q_{\gamma}(z)+\lambda^{2}(x\bullet_{\alpha,\,\beta}y)\bullet_{\alpha\,\beta,\,\gamma}q_{\gamma}(z)\\
	&\hspace{1cm} \text{(by $R_{\gamma}\circ q_{\gamma}= q_{\gamma}\circ R_{\gamma}$)}\\
	=&p_{\alpha}(x)\bullet_{\alpha,\,\beta\,\gamma}(R_{\beta}(y)\bullet_{\beta,\,\gamma}R_{\gamma}(z))+\lambda p_{\alpha}(x)\bullet_{\alpha,\,\beta\,\gamma}(y\bullet_{\beta,\,\gamma}R_{\gamma}(z))\\
	&+\lambda p_{\alpha}(x)\bullet_{\alpha,\,\beta\,\gamma}(R_{\beta}(y)\bullet_{\beta,\,\gamma}z)+\lambda^{2}p_{\alpha}(x)\bullet_{\alpha,\,\beta\,\gamma}(y\bullet_{\beta,\,\gamma}z)\hspace{1cm}\text{(by Eq.~(\ref{BiH-Omega-assoity}))}\\
	=&p_{\alpha}(x)\bullet_{\alpha,\,\beta\,\gamma}R_{\beta\,\gamma}(R_{\beta}(y)\bullet_{\beta,\,\gamma}z+y\bullet_{\beta,\,\gamma}R_{\gamma}(z)+\lambda y\bullet_{\beta,\,\gamma}z)\\
	&+\lambda p_{\alpha}(x)\bullet_{\alpha,\,\beta\,\gamma}(R_{\beta}(y)\bullet_{\beta,\,\gamma}z+y\bullet_{\beta,\,\gamma}R_{\gamma}(z)+\lambda y\bullet_{\beta,\,\gamma}z)\hspace{1cm}\text{(by Eq.~(\ref{RB-Eq}))}\\
	=&p_{\alpha}(x)\prec'_{\alpha,\,\beta\,\gamma}(R_{\beta}(y)\bullet_{\beta,\,\gamma}z+y\bullet_{\beta,\,\gamma}R_{\gamma}(z)+\lambda y\bullet_{\beta,\,\gamma}z)\\
	=&p_{\alpha}(x)\prec'_{\alpha,\,\beta\,\gamma}(y\succ'_{\beta,\,\gamma}z+y\prec'_{\beta,\,\gamma}z),
\end{align*}
\begin{align*}
	(x\succ'_{\alpha,\,\beta}y)\prec'_{\alpha\,\beta,\,\gamma}q_{\gamma}(z)=&(R_{\alpha}(x)\bullet_{\alpha,\,\beta}y)\bullet_{\alpha\,\beta,\,\gamma}R_{\gamma}q_{\gamma}(z)+\lambda (R_{\alpha}(x)\bullet_{\alpha,\,\beta}y)\bullet_{\alpha\,\beta,\,\gamma}q_{\gamma}(z)\\
	=&(R_{\alpha}(x)\bullet_{\alpha,\,\beta}y)\bullet_{\alpha\,\beta,\,\gamma}q_{\gamma}R_{\gamma}(z)+\lambda (R_{\alpha}(x)\bullet_{\alpha,\,\beta}y)\bullet_{\alpha\,\beta,\,\gamma}q_{\gamma}(z)\\
	&\hspace{1cm}(\text{by $R_{\gamma}\circ q_{\gamma}=q_{\gamma}\circ R_{\gamma}$})\\
	=&p_{\alpha}R_{\alpha}(x)\bullet_{\alpha,\,\beta\,\gamma}(y\bullet_{\beta,\,\gamma}R_{\gamma}(z))+\lambda p_{\alpha}R_{\alpha}(x)\bullet_{\alpha,\,\beta\,\gamma}(y\bullet_{\beta,\,\gamma}z)\hspace{1cm}\text{(by Eq.~(\ref{BiH-Omega-assoity}))}\\
	=&R_{\alpha}p_{\alpha}(x)\bullet_{\alpha,\,\beta\,\gamma}(y\bullet_{\beta,\,\gamma}R_{\gamma}(z)+\lambda y\bullet_{\beta,\,\gamma}z)\hspace{1.7cm}\text{(by $p_{\alpha}\circ R_{\alpha}=R_{\alpha}\circ p_{\alpha}$)}\\
	=&p_{\alpha}(x)\succ'_{\alpha,\,\beta\,\gamma}(y\prec'_{\beta,\,\gamma}z).
\end{align*}
\begin{align*}
	p_{\alpha}(x)\succ'_{\alpha,\,\beta\,\gamma}(y\succ'_{\beta,\,\gamma}z)=&p_{\alpha}(x)\succ'_{\alpha,\,\beta\,\gamma}(R_{\beta}(y)\bullet_{\beta,\,\gamma}z)=R_{\alpha}p_{\alpha}(x)\bullet_{\alpha,\,\beta\,\gamma}(R_{\beta}(y)\bullet_{\beta,\,\gamma}z)\\
	=&p_{\alpha}R_{\alpha}(x)\bullet_{\alpha,\,\beta\,\gamma}(R_{\beta}(y)\bullet_{\beta,\,\gamma}z)\hspace{1cm}\text{( by }R_{\alpha}\circ p_{\alpha}=p_{\alpha}\circ R_{\alpha})\\
	=&(R_{\alpha}(x)\bullet_{\alpha,\,\beta}R_{\beta}(y))\bullet_{\alpha\,\beta,\,\gamma}q_{\gamma}(z)\hspace{1cm}\text{(by Eq.~(\ref{BiH-Omega-assoity}))}\\
	=&R_{\alpha\,\beta}(x\bullet_{\alpha,\,\beta}R_{\beta}(y)+R_{\alpha}(x)\bullet_{\alpha,\,\beta}y+\lambda x\bullet_{\alpha,\,\beta}y)\bullet_{\alpha\,\beta,\,\gamma}q_{\gamma}(z)\hspace{1cm}\text{(by Eq.~(\ref{RB-Eq}))}\\
	=&(x\bullet_{\alpha,\,\beta}R_{\beta}(y)+\lambda x\bullet_{\alpha,\,\beta}y+R_{\alpha}(x)\bullet_{\alpha,\,\beta}y)\succ'_{\alpha\,\beta,\,\gamma}q_{\gamma}(z)\\
	=&(x\prec'_{\alpha,\,\beta} y+x\succ' _{\alpha,\,\beta}y)\succ'_{\alpha\,\beta,\,\gamma} q_{\gamma}(z).
\end{align*}
This completes the proof.
\end{proof}

%%%%%%%%%%%%%%%
\section{BiHom-$\Omega$-pre-Lie algebras and BiHom-$\Omega$-Lie algebras}\label{sec3}

In this section, we assume that $\Omega$ is a commutative semigroup.
First, we introduce the definitions of BiHom-$\Omega$-pre-Lie algebras and BiHom-$\Omega$-Lie algebras, then we obtain a BiHom-$\Omega$ analogue of the classical result of Aguiar~\cite{aguiar}.

\subsection{BiHom-$\Omega$-pre-Lie algebras}
The concept of $\Omega$-pre-Lie algebras was proposed in \cite{Aguiar}, as a generalization of pre-Lie algebras invented by Gerstenhaber and Vinberg~\cite{preLie-1,preLie-2}.

\begin{defn}\cite{Aguiar}
	An $\mathbf{\Omega}$-$\mathbf{pre}$-$\mathbf{Lie \, algebra}$ $ (A,\rhd_{\alpha,\,\beta})_{\alpha,\,\beta\in \Omega} $ is a vector space $A$ equipped with a family of bilinear maps $(\rhd_{\alpha,\,\beta} :A\otimes A\rightarrow A )_{\alpha,\,\beta \in \Omega}$ such that
	\begin{align}\label{Omega-prelie-alg}
		x\rhd_{\alpha,\,\beta\gamma} (y\rhd_{\beta,\,\gamma} z)-(x\rhd_{\alpha,\,\beta} y)\rhd_{\alpha\beta,\,\gamma} z=y\rhd_{\beta,\,\alpha\gamma} (x\rhd_{\alpha,\gamma} z)-(y\rhd_{\beta,\,\alpha} x)\rhd_{\beta\alpha,\gamma} z,
	\end{align}
	for all $x, y, z\in A ,\,\alpha,\,\beta,\,\gamma \in \Omega $.
\end{defn}

\begin{defn}
Let $ (A,\rhd_{\alpha,\,\beta})_{\alpha,\,\beta\in \Omega} $ and $ (A',\rhd'_{\alpha,\,\beta})_{\alpha,\,\beta\in \Omega} $ be two $\Omega$-pre-Lie algebras. A family of linear maps $ (f_{\alpha})_{\alpha\in \Omega}: A \rightarrow A' $ is called an $\mathbf{\Omega}$-$\mathbf{pre}$-$\mathbf{Lie \, algebra \, morphism}$ if
\begin{align}\label{Omega-prelie-morphism}
	f_{\alpha\,\beta}(x\rhd_{\alpha,\,\beta}y)=f_{\alpha}(x)\rhd'_{\alpha,\,\beta}f_{\beta}(y),
\end{align}
for all $ x,y\in A,\,\alpha,\,\beta\in \Omega. $
\end{defn}

 BiHom-pre-Lie algebras, as the BiHom version of classical pre-Lie algebras, has been proposed in \cite{bihomprelie}. Similarly, we give the BiHom version of $\Omega$-pre-Lie algebras.

\begin{defn}
   A $\mathbf{BiHom}$-$\mathbf{\Omega}$-$\mathbf{pre}$-$\mathbf{Lie \, algebra}$ $(A, \blacktriangleright_{\alpha,\,\beta} , p_{\alpha} , q_{\alpha} )_{\alpha,\,\beta\in \Omega}$ is a vector space $A$ equipped with a family of bilinear maps $ (\blacktriangleright_{\alpha,\,\beta} :A\otimes A\rightarrow A )_{\alpha,\,\beta \in \Omega}$ and two commuting $\Omega$-pre-Lie algebra morphisms $p_{\alpha} , q_{\alpha} :A\rightarrow A$ such that

	$ p_{\alpha}q_{\alpha} (x)\blacktriangleright_{\alpha,\,\beta\gamma} (p_{\beta} (y)\blacktriangleright_{\beta,\,\gamma} z)-(q_{\alpha} (x)\blacktriangleright_{\alpha,\,\beta} p_{\beta} (y))\blacktriangleright_{\alpha\beta,\,\gamma} q_{\gamma} (z) $
	\begin{align}\label{BHO-prelie-alg}
		&=p_{\beta} q_{\beta} (y)\blacktriangleright_{\beta,\,\alpha\gamma} (p_{\alpha}(x)\blacktriangleright_{\alpha,\gamma} z)-(q_{\beta}(y)\blacktriangleright_{\beta,\,\alpha} p_{\alpha}(x))\blacktriangleright_{\beta\alpha,\gamma} q_{\gamma}(z),
	\end{align}
for all $x, y, z\in A, \,\alpha,\,\beta,\,\gamma\in \Omega$. The maps $p_{\alpha} $ and $q_{\alpha}$ (in this order) are called the structure maps
of $A$. 	
\end{defn}

	\begin{defn}
	Let $(A, \blacktriangleright_{\alpha,\,\beta} , p_{\alpha}, q_{\alpha} )_{\alpha,\,\beta \in \Omega}$ and $(A', \blacktriangleright_{\alpha,\,\beta} ', p'_{\alpha}, q'_{\alpha})_{\alpha,\,\beta\in \Omega}$ be two BiHom-$\Omega$-pre-Lie algebras. A family of linear maps $(f_{\alpha})_{\alpha\in \Omega}:A\rightarrow A'$ is called a $\mathbf{BiHom}$-$\mathbf{\Omega}$-$\mathbf{pre}$-$\mathbf{Lie \, algebra}$ $\mathbf{morphism}$ if $(f_{\alpha})_{\alpha\in\Omega}$ is an $\Omega$-pre-Lie algebra morphism and
	\[f_{\alpha}\circ p_{\alpha} =p'_{\alpha}\circ f_{\alpha} ,\quad f_{\alpha}\circ q_{\alpha} =q'_{\alpha}\circ f_{\alpha},\quad \text{for all }\alpha\in \Omega.\]
\end{defn}

Inspired by \cite{bihomprelie}, we have the similar property of $\Omega$-pre-Lie algebras as follows.

\begin{prop}\label{Yautwist-preLie}
	Let $(A, \rhd_{\alpha,\,\beta} )_{\alpha,\,\beta \in \Omega}$ be an $\Omega$-pre-Lie algebra. If $p_{\alpha} , q_{\alpha} :A\rightarrow A$ are two commuting $\Omega$-pre-Lie algebra morphisms. Define the multiplications on $A$ by
	\[x\blacktriangleright_{\alpha,\,\beta}y:=p_{\alpha}(x)\rhd_{\alpha,\,\beta} q_{\beta}(y), \]
for all $x, y\in A,\,\alpha,\,\beta\in \Omega,$	then $(A, \blacktriangleright_{\alpha,\,\beta} , p_{\alpha} , q_{\alpha} )_{\alpha,\,\beta \in \Omega}$ is a BiHom-$\Omega$-pre-Lie algebra,
	called the Yau twist of $(A, \rhd_{\alpha,\,\beta} )_{\alpha,\,\beta \in\Omega}$.
\end{prop}
\begin{proof}
	 For any $ x,y,z\in A,\,\alpha,\,\beta,\,\gamma\in \Omega ,$ we only need to prove Eq.~(\ref{BHO-prelie-alg}).
	\begin{align*}
		p&_{\alpha}q_{\alpha}(x)\blacktriangleright_{\alpha,\,\beta\,\gamma} (p_{\beta}(y)\blacktriangleright_{\beta,\,\gamma} z)-(q_{\alpha} (x)\blacktriangleright_{\alpha,\,\beta} p_{\beta} (y))\blacktriangleright_{\alpha\,\beta,\,\gamma} q_{\gamma} (z)\\
		=&p_{\alpha}q_{\alpha}(x)\blacktriangleright_{\alpha,\,\beta\,\gamma}( p_{\beta}^2( y) \rhd_{\beta,\,\gamma}q_{\gamma}(z))-(p_{\alpha}q_{\alpha}(x)\rhd_{\alpha,\,\beta }q_{\beta}p_{\beta}(y))\blacktriangleright_{\alpha\,\beta,\,\gamma}q_{\gamma}(z)\\
		=&p_{\alpha}^{2}q_{\alpha}(x)\rhd_{\alpha,\,\beta\,\gamma}q_{\beta\,\gamma}(p_{\beta}^{2}(y)\rhd_{\beta,\,\gamma}q_{\gamma}(z))-p_{\alpha\,\beta}(p_{\alpha}q_{\alpha}(x)\rhd_{\alpha,\,\beta}q_{\beta}p_{\beta}(y))\rhd_{\alpha\,\beta,\,\gamma}q_{\gamma}^{2}(z)\\
		=&p_{\alpha} ^2q_{\alpha} (x)\rhd_{\alpha,\,\beta\,\gamma} (q_{\beta}p_{\beta} ^2 (y)\rhd_{\beta,\,\gamma} q_{\gamma} ^2(z))-(p_{\alpha}^2q_{\alpha} (x)\rhd_{\alpha,\,\beta}
		p_{\beta}q_{\beta}p_{\beta} (y))\rhd_{\alpha\,\beta,\,\gamma} q_{\gamma}^2(z)\\
		&\hspace{1cm}\text{(by $q_{\beta\,\gamma},p_{\alpha\,\beta}$ satisfying Eq.~(\ref{Omega-prelie-morphism})})\\
		=&q_{\alpha}p_{\alpha}^{2}(x)\rhd_{\alpha,\,\beta\,\gamma}(p_{\beta}^{2}q_{\beta}(y)\rhd_{\beta,\,\gamma}q_{\gamma}^{2}(z))-(q_{\alpha}p_{\alpha}^{2}(x)\rhd_{\alpha,\,\beta}p_{\beta}^{2}q_{\beta}(y))\rhd_{\alpha\,\beta,\,\gamma}q_{\gamma}^{2}(z)\\
		&\hspace{1cm}\text{(by $ p_{\alpha}\circ q_{\alpha}=q_{\alpha}\circ p_{\alpha} $})\\
		=&p_{\beta}^{2}q_{\beta}(y)\rhd_{\beta,\,\alpha\,\gamma}(q_{\alpha}p_{\alpha}^{2}(x)\rhd_{\alpha,\,\gamma}q_{\gamma}^{2}(z))-(p_{\beta}^{2}q_{\beta}(y)\rhd_{\beta,\,\alpha}q_{\alpha}p_{\alpha}^{2}(x))\rhd_{\beta\,\alpha,\,\gamma}q_{\gamma}^{2}(z)\\
		&\hspace{1cm}\text{(by $(A,\rhd_{\alpha,\,\beta})_{\alpha,\,\beta\in \Omega}$ satisfying Eq.~(\ref{Omega-prelie-alg}))}\\
		=&p_{\beta}^{2}q_{\beta}(y)\rhd_{\beta,\,\alpha\,\gamma}(q_{\alpha}p_{\alpha}^{2}(x)\rhd_{\alpha,\,\gamma}q_{\gamma}^{2}(z))-(p_{\beta}^{2}q_{\beta}(y)\rhd_{\beta,\,\alpha}p_{\alpha}q_{\alpha}p_{\alpha}(x))\rhd_{\beta\,\alpha,\,\gamma}q_{\gamma}^{2}(z)\\
		&\hspace{1cm}\text{(by $ q_{\alpha}\circ p_{\alpha}=p_{\alpha}\circ q_{\alpha} $})\\
		=&p_{\beta}^{2}q_{\beta}(y)\rhd_{\beta,\,\alpha\,\gamma}q_{\alpha\,\gamma}(p_{\alpha}^{2}(x)\rhd_{\alpha,\,\gamma}q_{\gamma}(z))-p_{\beta\,\alpha}(p_{\beta}q_{\beta}(y)\rhd_{\beta,\,\alpha}q_{\alpha}p_{\alpha}(x))\rhd_{\beta\,\alpha,\,\gamma}q_{\gamma}^{2}(z)\\
		&\hspace{1cm}\text{(by $p_{\alpha},q_{\alpha}$ satisfying Eq.~(\ref{Omega-prelie-morphism})})\\
		=&p_{\beta}q_{\beta}(y)\blacktriangleright_{\beta,\,\alpha\,\gamma}(p_{\alpha}^{2}(x)\rhd_{\alpha,\,\gamma}q_{\gamma}(z))-(p_{\beta}q_{\beta}(y)\rhd_{\beta,\,\alpha}q_{\alpha}p_{\alpha}(x))\blacktriangleright_{\beta\,\alpha,\,\gamma}q_{\gamma}(z)\\
		=&p_{\beta}q_{\beta}(y)\blacktriangleright_{\beta,\,\alpha\,\gamma}(p_{\alpha}(x)\blacktriangleright_{\alpha,\,\gamma} z)-(q_{\beta}(y)\blacktriangleright_{\beta,\,\alpha} p_{\alpha}(x))\blacktriangleright_{\beta\,\alpha,\,\gamma} q_{\gamma}(z).
	\end{align*}
	This completes the proof.
\end{proof}

The following result is a more general of the Yau twisting procedure for BiHom-$\Omega$-pre-Lie algebras and the proof is similar to Proposition~\ref{Yautwist-preLie}.

\begin{prop}
	Let $(A, \blacktriangleright_{\alpha,\,\beta} ,p_{\alpha} , q_{\alpha} )_{\alpha,\,\beta \in \Omega}$ be a BiHom-$\Omega$-pre-Lie algebra. If $p'_{\alpha}, q'_{\alpha}:
	A\rightarrow A$ are two BiHom-$\Omega$-pre-Lie algebra morphisms and any two families of the maps $p_{\alpha}, q_{\alpha} ,
	p'_{\alpha}, q'_{\alpha}$ commute with each other. Define the multiplication on $A$ by
	\[x \blacktriangleright'_{\alpha,\,\beta}y:=p'_{\alpha}(x)\blacktriangleright_{\alpha,\,\beta} q'_{\beta}(y),\]
for all $x, y\in A,\,\alpha,\,\beta\in \Omega,$	then $(A, \blacktriangleright'_{\alpha,\,\beta}, p_{\alpha} \circ p'_{\alpha}, q_{\alpha} \circ q'_{\alpha})_{\alpha,\,\beta \in \Omega}$
	is a BiHom-$\Omega$-pre-Lie algebra.
\end{prop}

The concept of BiHom-$\Omega$-dendriform algebras has been given in Definition~\ref{BiHom-den-fam}. When the structure maps are bijective, we get a BiHom-$\Omega$-pre-Lie algebra from the BiHom-$\Omega$-dendriform algebra as follows.

\begin{prop} \label{BHdendpreLie}
	Let $(A, \prec_{\alpha,\,\beta} , \succ_{\alpha,\,\beta} , p_{\alpha} , q_{\alpha} )_{\alpha,\,\beta\in \Omega}$ be a BiHom-$\Omega$-dendriform algebra. If $p_{\alpha} ,\,q_{\alpha} $
	are bijective and we define the operation by
	\begin{eqnarray*}
		&&x\blacktriangleright_{\alpha,\,\beta} y:=x\succ_{\alpha,\,\beta} y-(p_{\beta}^{-1}q_{\beta}(y))\prec_{\beta,\,\alpha} (p_{\alpha}q_{\alpha}^{-1}(x)), \;\;\;\;\;
%		x\lhd_{\omega} y=x\prec_{\omega} y-(p ^{-1}q (y))\succ_{\omega} (p q ^{-1}(x)).
	\end{eqnarray*}
for all $x, y\in A,\,\alpha,\,\beta\in \Omega$. Then $(A, \blacktriangleright_{\alpha,\,\beta} , p_{\alpha} , q_{\alpha} )_{\alpha,\,\beta\in \Omega}$  %(respectively $(A, \lhd_{\omega} , p , q )_{\omega\in \Omega}$)
	is a BiHom-$\Omega$-pre-Lie algebra.
\end{prop}

\begin{proof}
	For $ x,\,y,\,z\in A,\,\alpha,\,\beta,\,\gamma\in \Omega,$ we only need to check Eq.~(\ref{BHO-prelie-alg}).
 On the one hand, we have
	\begin{align*}
		p&_{\alpha}q_{\alpha}(x)\blacktriangleright_{\alpha,\,\beta\,\gamma}(p_{\beta}(y)\blacktriangleright_{\beta,\,\gamma}z)-(q_{\alpha}(x)\blacktriangleright_{\alpha,\,\beta}p_{\beta}(y))\blacktriangleright_{\alpha\,\beta,\,\gamma}q_{\gamma}(z)\\
		=&p_{\alpha}q_{\alpha}(x)\succ_{\alpha,\,\beta\,\gamma}(p_{\beta}(y)\blacktriangleright_{\beta,\,\gamma}z)-(p_{\beta\,\gamma}^{-1}q_{\beta\,\gamma}(p_{\beta}(y)\blacktriangleright_{\beta,\,\gamma}z))\prec_{\beta\,\gamma,\,\alpha}(p_{\alpha}q_{\alpha}^{-1}p_{\alpha}q_{\alpha}(x))\\
		&-(q_{\alpha}(x)\blacktriangleright_{\alpha,\,\beta}p_{\beta}(y))\succ_{\alpha\,\beta,\,\gamma}q_{\gamma}(z)+p_{\gamma}^{-1}q_{\gamma}^{2}(z)\prec_{\gamma,\,\alpha\,\beta}p_{\alpha\,\beta}q_{\alpha\,\beta}^{-1}(q_{\alpha}(x)\blacktriangleright_{\alpha,\,\beta}p_{\beta}(y))\\
		=&p_{\alpha}q_{\alpha}(x)\succ_{\alpha,\,\beta\,\gamma}(p_{\beta}(y)\succ_{\beta,\,\gamma}z)-p_{\beta\,\gamma}^{-1}q_{\beta\,\gamma}(p_{\beta}(y)\succ_{\beta,\,\gamma}z)\prec_{\beta\,\gamma,\,\alpha}p_{\alpha}^{2}(x)\\
		&-p_{\alpha}q_{\alpha}(x)\succ_{\alpha,\,\beta\,\gamma}(p_{\gamma}^{-1}q_{\gamma}(z)\prec_{\gamma,\,\beta}p_{\beta}^{2}q_{\beta}^{-1}(y))+p_{\beta\,\gamma}^{-1}q_{\beta\,\gamma}(p_{\gamma}^{-1}q_{\gamma}(z)\prec_{\gamma,\,\beta}p_{\beta}^{2}q_{\beta}^{-1}(y))\prec_{\beta\,\gamma,\,\alpha}p_{\alpha}^{2}(x)\\
		&-(q_{\alpha}(x)\succ_{\alpha,\,\beta}p_{\beta}(y))\succ_{\alpha\,\beta,\,\gamma}q_{\gamma}(z)+p_{\gamma}^{-1}q_{\gamma}^{2}(z)\prec_{\gamma,\,\alpha\,\beta}p_{\alpha\,\beta}q_{\alpha\,\beta}^{-1}(q_{\alpha}(x)\succ_{\alpha,\,\beta}p_{\beta}(y))\\
		&+(q_{\beta}(y)\prec_{\beta,\,\alpha}p_{\alpha}(x))\succ_{\alpha\,\beta,\,\gamma}q_{\gamma}(z)-p_{\gamma}^{-1}q_{\gamma}^{2}(z)\prec_{\gamma,\,\alpha\,\beta}p_{\alpha\,\beta}q_{\alpha\,\beta}^{-1}(q_{\beta}(y)\prec_{\beta,\,\alpha}p_{\alpha}(x))\\
		=&p_{\alpha}q_{\alpha}(x)\succ_{\alpha,\,\beta\,\gamma}(p_{\beta}(y)\succ_{\beta,\,\gamma}z)-(q_{\beta}(y)\succ_{\beta,\,\gamma}p_{\gamma}^{-1}q_{\gamma}(z))\prec_{\beta\,\gamma,\,\alpha}p_{\alpha}^{2}(x)\\
		&-p_{\alpha}q_{\alpha}(x)\succ_{\alpha,\,\beta\,\gamma}(p_{\gamma}^{-1}q_{\gamma}(z)\prec_{\gamma,\,\beta}p_{\beta}^{2}q_{\beta}^{-1}(y))+(p_{\gamma}^{-2}q_{\gamma}^{2}(z)\prec_{\gamma,\,\beta}p_{\beta}(y))\prec_{\beta\,\gamma,\,\alpha}p_{\alpha}^{2}(x)\\
		&-(q_{\alpha}(x)\succ_{\alpha,\,\beta}p_{\beta}(y))\prec_{\alpha\,\beta,\,\gamma}q_{\gamma}(z)+p_{\gamma}^{-1}q_{\gamma}^{2}(z)\prec_{\gamma,\,\alpha\,\beta}(p_{\alpha}(x)\succ_{\alpha,\,\beta}p_{\beta}^{2}q_{\beta}^{-1}(y))\\
		&+(q_{\beta}(y)\prec_{\beta,\,\alpha}p_{\alpha}(x))\succ_{\alpha\,\beta,\,\gamma}q_{\gamma}(z)-p_{\gamma}^{-1}q_{\gamma}^{2}(z)\prec_{\gamma,\,\alpha\,\beta}(p_{\beta}(y)\prec_{\beta,\,\alpha}p_{\alpha}^{2}q_{\alpha}^{-1}(x))\\
		&\hspace{1cm}\text{(by $ p_{\beta\,\gamma}^{-1},q_{\beta\,\gamma},p_{\alpha\,\beta},q_{\alpha\,\beta}^{-1} $ satisfying Eq.~(\ref{Omega-den-morphism})})\\
        =&(q_{\alpha}(x)\prec_{\alpha,\,\beta}p_{\beta}(y))\succ_{\alpha\,\beta,\,\gamma}q_{\gamma}(z)-p_{\beta}q_{\beta}(y)\succ_{\beta,\,\gamma\,\alpha}(p_{\gamma}^{-1}q_{\gamma}(z)\prec_{\gamma,\,\alpha}q_{\alpha}^{-1}p_{\alpha}^{2}(x))\\
        &-(q_{\alpha}(x)\succ_{\alpha,\,\gamma}p_{\gamma}^{-1}q_{\gamma}(z))\prec_{\alpha\,\gamma,\,\beta}p_{\beta}^{2}(y)+p_{\gamma}^{-1}q_{\gamma}^{2}(z)\prec_{\gamma,\,\alpha\,\beta}(p_{\beta}(y)\succ_{\beta,\,\alpha}q_{\alpha}^{-1}p_{\alpha}^{2}(x))\\
        &+p_{\gamma}^{-1}q_{\gamma}^{2}(z)\prec_{\gamma,\,\alpha\,\beta}(p_{\alpha}(x)\succ_{\alpha,\,\beta}p_{\beta}^{2}q_{\beta}^{-1}(y))+(q_{\beta}(y)\prec_{\beta,\,\alpha}p_{\alpha}(x))\succ_{\alpha\,\beta,\,\gamma}q_{\gamma}(z).\hspace{0.3cm}\text{(by Eqs.~(\ref{BiHomdend4})-(\ref{BiHomdend6}))}		
	\end{align*}
On the other hand, we have
\begin{align*}
	p&_{\beta}q_{\beta}(y)\blacktriangleright_{\beta,\,\alpha\,\gamma}(p_{\alpha}(x)\blacktriangleright_{\alpha,\,\gamma}z)-(q_{\beta}(y)\blacktriangleright_{\beta,\,\alpha}p_{\alpha}(x))\blacktriangleright_{\beta\,\alpha,\,\gamma}q_{\gamma}(z)\\
	=&p_{\beta}q_{\beta}(y)\succ_{\beta,\,\alpha\,\gamma}(p_{\alpha}(x)\blacktriangleright_{\alpha,\,\gamma}z)-p_{\alpha\,\gamma}^{-1}q_{\alpha\,\gamma}(p_{\alpha}(x)\blacktriangleright_{\alpha,\,\gamma}z)\prec_{\alpha\,\gamma,\,\beta}p_{\beta}q_{\beta}^{-1}p_{\beta}q_{\beta}(y)\\
	&-(q_{\beta}(y)\blacktriangleright_{\beta,\,\alpha}p_{\alpha}(x))\succ_{\beta\,\alpha,\,\gamma}q_{\gamma}(z)+p_{\gamma}^{-1}q_{\gamma}^{2}(z)\prec_{\gamma,\,\beta\,\alpha}p_{\beta\,\alpha}q_{\beta\,\alpha}^{-1}(q_{\beta}(y)\blacktriangleright_{\beta,\,\alpha}p_{\alpha}(x))\\
	=&p_{\beta}q_{\beta}(y)\succ_{\beta,\,\alpha\,\gamma}(p_{\alpha}(x)\succ_{\alpha,\,\gamma}z)-p_{\alpha\,\gamma}^{-1}q_{\alpha\,\gamma}(p_{\alpha}(x)\succ_{\alpha,\,\gamma}z)\prec_{\alpha\,\gamma,\,\beta}p_{\beta}^{2}(y)\\
	&-p_{\beta}q_{\beta}(y)\succ_{\beta,\,\alpha\,\gamma}(p_{\gamma}^{-1}q_{\gamma}(z)\prec_{\gamma,\,\alpha}p_{\alpha}^{2}q_{\alpha}^{-1}(x))+p_{\alpha\,\gamma}^{-1}q_{\alpha\,\gamma}(p_{\gamma}^{-1}q_{\gamma}(z)\prec_{\gamma,\,\alpha}p_{\alpha}^{2}q_{\alpha}^{-1}(x))\prec_{\alpha\,\gamma,\,\beta}p_{\beta}^{2}(y)\\
	&-(q_{\beta}(y)\succ_{\beta,\,\alpha}p_{\alpha}(x))\succ_{\beta\,\alpha,\,\gamma}q_{\gamma}(z)+p_{\gamma}^{-1}q_{\gamma}^{2}(z)\prec_{\gamma,\,\beta\,\alpha}p_{\beta\,\alpha}q_{\beta\,\alpha}^{-1}(q_{\beta}(y)\succ_{\beta,\,\alpha}p_{\alpha}(x))\\
	&+(q_{\alpha}(x)\prec_{\alpha,\,\beta}p_{\beta}(y))\succ_{\alpha\,\beta,\,\gamma}q_{\gamma}(z)-p_{\gamma}^{-1}q_{\gamma}^{2}(z)\prec_{\gamma,\,\alpha\,\beta}p_{\alpha\,\beta}q_{\alpha\,\beta}^{-1}(q_{\alpha}(x)\prec_{\alpha,\,\beta}p_{\beta}(y))\\
	=&p_{\beta}q_{\beta}(y)\succ_{\beta,\,\alpha\,\gamma}(p_{\alpha}(x)\succ_{\alpha,\,\gamma}z)-(q_{\alpha}(x)\succ_{\alpha,\,\gamma}p_{\gamma}^{-1}q_{\gamma}(z))\prec_{\alpha\,\gamma,\,\beta}p_{\beta}^{2}(y)\\
	&-p_{\beta}q_{\beta}(y)\succ_{\beta,\,\alpha\,\gamma}(p_{\gamma}^{-1}q_{\gamma}(z)\prec_{\gamma,\,\alpha}p_{\alpha}^{2}q_{\alpha}^{-1}(x))+(p_{\gamma}^{-2}q_{\gamma}^{2}(z)\prec_{\gamma,\,\alpha}p_{\alpha}(x))\prec_{\alpha\,\gamma,\,\beta}p_{\beta}^{2}(y)\\
	&-(q_{\beta}(y)\succ_{\beta,\,\alpha}p_{\alpha}(x))\succ_{\beta\,\alpha,\,\gamma}q_{\gamma}(z)+p_{\gamma}^{-1}q_{\gamma}^{2}(z)\prec_{\gamma,\,\beta\,\alpha}(p_{\beta}(y)\succ_{\beta,\,\alpha}p_{\alpha}^{2}q_{\alpha}^{-1}(x))\\
	&+(q_{\alpha}(x)\prec_{\alpha,\,\beta}p_{\beta}(y))\succ_{\alpha\,\beta,\,\gamma}q_{\gamma}(z)-p_{\gamma}^{-1}q_{\gamma}^{2}(z)\prec_{\gamma,\,\alpha\,\beta}(p_{\alpha}(x)\prec_{\alpha,\,\beta}p_{\beta}^{2}q_{\beta}^{-1}(y))\\
	&\hspace{1cm}\text{(by $ p_{\alpha\,\gamma}^{-1},q_{\alpha\,\gamma},p_{\alpha\,\beta},q_{\alpha\,\beta}^{-1} $ satisfying Eqs.~(\ref{BiHomdend2})-(\ref{BiHomdend3})})\\
	=&(q_{\beta}(y)\prec_{\beta,\,\alpha}p_{\alpha}(x))\succ_{\beta\,\alpha,\,\gamma}q_{\gamma}(z)-(q_{\alpha}(x)\succ_{\alpha,\,\gamma}p_{\gamma}^{-1}q_{\gamma}(z))\prec_{\alpha\,\gamma,\,\beta}p_{\beta}^{2}(y)\\
	&-p_{\beta}q_{\beta}(y)\succ_{\beta,\,\alpha\,\gamma}(p_{\gamma}^{-1}q_{\gamma}(z)\prec_{\gamma,\,\alpha}p_{\alpha}^{2}q_{\alpha}^{-1}(x))+p_{\gamma}^{-1}q_{\gamma}^{2}(z)\prec_{\gamma,\,\alpha\,\beta}(p_{\alpha}(x)\succ_{\alpha,\,\beta}q_{\beta}^{-1}p_{\beta}^{2}(y))\\
	&+p_{\gamma}^{-1}q_{\gamma}^{2}(z)\prec_{\gamma,\,\beta\,\alpha}(p_{\beta}(y)\succ_{\beta,\,\alpha}p_{\alpha}^{2}q_{\alpha}^{-1}(x))+(q_{\alpha}(x)\prec_{\alpha,\,\beta}p_{\beta}(y))\succ_{\alpha\,\beta,\,\gamma}q_{\gamma}(z)\hspace{0.3cm}\text{(by Eq.~(\ref{BiHomdend4}) and Eq.~(\ref{BiHomdend6}))}	\\
	=&(q_{\beta}(y)\prec_{\beta,\,\alpha}p_{\alpha}(x))\succ_{\beta\,\alpha,\,\gamma}q_{\gamma}(z)-(q_{\alpha}(x)\succ_{\alpha,\,\gamma}p_{\gamma}^{-1}q_{\gamma}(z))\prec_{\alpha\,\gamma,\,\beta}p_{\beta}^{2}(y)\\
	&-p_{\beta}q_{\beta}(y)\succ_{\beta,\,\alpha\,\gamma}(p_{\gamma}^{-1}q_{\gamma}(z)\prec_{\gamma,\,\alpha}q_{\alpha}^{-1}p_{\alpha}^{2}(x))+p_{\gamma}^{-1}q_{\gamma}^{2}(z)\prec_{\gamma,\,\alpha\,\beta}(p_{\alpha}(x)\succ_{\alpha,\,\beta}p_{\beta}^{2}q_{\beta}^{-1}(y))\\
	&+p_{\gamma}^{-1}q_{\gamma}^{2}(z)\prec_{\gamma,\,\beta\,\alpha}(p_{\beta}(y)\succ_{\beta,\,\alpha}q_{\alpha}^{-1}p_{\alpha}^{2}(x))+(q_{\alpha}(x)\prec_{\alpha,\,\beta}p_{\beta}(y))\succ_{\alpha\,\beta,\,\gamma}q_{\gamma}(z).\hspace{0.3cm}\text{(by $ p_{\alpha}q_{\alpha}^{-1}=q_{\alpha}^{-1}p_{\alpha} $})
\end{align*}
By comparing the items of both sides, we get
\begin{align*}
	&p_{\alpha}q_{\alpha}(x)\blacktriangleright_{\alpha,\,\beta\,\gamma}(p_{\beta}(y)\blacktriangleright_{\beta,\,\gamma}z)-(q_{\alpha}(x)\blacktriangleright_{\alpha,\,\beta}p_{\beta}(y))\blacktriangleright_{\alpha\,\beta,\,\gamma}q_{\gamma}(z)\\
	=&p_{\beta}q_{\beta}(y)\blacktriangleright_{\beta,\,\alpha\,\gamma}(p_{\alpha}(x)\blacktriangleright_{\alpha,\,\gamma}z)-(q_{\beta}(y)\blacktriangleright_{\beta,\,\alpha}p_{\alpha}(x))\blacktriangleright_{\beta\,\alpha,\,\gamma}q_{\gamma}(z).
\end{align*}
	
\end{proof}

In classical cases, any associative algebra is also a pre-Lie algebra. We generlize this result to BiHom-$\Omega$ version.

\begin{prop}\label{asso-preLie}
If $ (A,\bullet_{\alpha,\,\beta},p_{\alpha},q_{\alpha})_{\alpha,\,\beta\in \Omega} $ is a BiHom-$\Omega$-assocative algebra, then $ (A,\bullet_{\alpha,\,\beta},p_{\alpha},q_{\alpha})_{\alpha,\,\beta\in \Omega} $ is a BiHom-$\Omega$-pre-Lie algebra.
\end{prop}
\begin{proof}
	For any $x,y,z\in A,\,\alpha,\,\beta,\,\gamma\in \Omega,$ we only need to prove Eq.~(\ref{BHO-prelie-alg}). By BiHom-$\Omega$-associativity, we have
	\[p_{\alpha}(x)\bullet_{\alpha,\,\beta\,\gamma}(y\bullet_{\beta,\,\gamma}z)=(x\bullet_{\alpha,\,\beta}y)\bullet_{\alpha\,\beta,\,\gamma}q_{\gamma}(z),\]
	by replacing $x$ with $q_{\alpha}(x)$ and $y$ with $ p_{\beta}(y) $, we get
	\[p_{\alpha}q_{\alpha}(x)\bullet_{\alpha,\,\beta\,\gamma}(p_{\beta}(y)\bullet_{\beta,\,\gamma}z)-(q_{\alpha}(x)\bullet_{\alpha,\,\beta}p_{\beta}(y))\bullet_{\alpha\,\beta,\,\gamma}q_{\gamma}(z)=0.\]
	
\noindent	Similarly, we get
	\[p_{\beta}q_{\beta}(y)\bullet_{\beta,\,\alpha\,\gamma}(p_{\alpha}(x)\bullet_{\alpha,\,\gamma}z)-(q_{\beta}(y)\bullet_{\beta,\,\alpha}p_{\alpha}(x))\bullet_{\beta\,\alpha,\,\gamma}q_{\gamma}(z)=0.\]
	Thus, we get Eq.~(\ref{BHO-prelie-alg}). This completes the proof.
\end{proof}

\subsection{BiHom-$\Omega$-Lie algebras}
In this subsection, we first introduce the concept of BiHom-$\Omega$-Lie algebras, then we prove that a new BiHom-$\Omega$-Lie algebra can be induced by a Rota-Baxter family of weight $\lambda$ on the BiHom-$\Omega$-Lie algebra. Finally, we generalize the classical relationships of associative algebras, pre-Lie algebras and Lie algebras to the BiHom-$\Omega$ version.

\begin{defn}\cite{Aguiar}
	An $\bf{\Omega}$-$\mathbf{Lie \, algebra}$ $ (L,[\cdot,\cdot]_{\alpha,\,\beta})_{\alpha,\,\beta\in \Omega} $ is a vector space $L$ equipped with a family of binary operations $ ([\cdot,\cdot]_{\alpha,\,\beta}: L \times L\rightarrow L)_{\alpha,\,\beta\in \Omega} $ such that
	\begin{align}
		&[x,y]_{\alpha,\,\beta}=-[y,x]_{\beta,\,\alpha},\label{Omega-Lie-skew}\\
		[x,[y,z]_{\beta,\,\gamma}]_{\alpha,\,\beta\,\gamma}&+[y,[z,x]_{\gamma,\,\alpha}]_{\beta,\,\gamma\,\alpha}+[z,[x,y]_{\alpha,\,\beta}]_{\gamma,\,\alpha\,\beta}=0,\label{Omega-Lie-Jacobi}
	\end{align}
	for all $ x,y,z\in L,\,\alpha,\,\beta,\,\gamma\in \Omega $.
\end{defn}
\begin{defn}
	Let $ (L,[\cdot,\cdot]_{\alpha,\,\beta})_{\alpha,\,\beta\in \Omega} $ and $ (L',[\cdot,\cdot]'_{\alpha,\,\beta})_{\alpha,\beta\in \Omega} $ be two $\Omega$-Lie algebras. A family of linear maps $ (f_{\alpha})_{\alpha\in \Omega}: L \rightarrow L' $ is called an $\mathbf{\Omega}$-$\mathbf{Lie \, algebra \, morphism}$ if 	
	\begin{align}\label{Omega-Lie-morphism}
		f_{\alpha\,\beta}([x, y]_{\alpha,\,\beta})=[f_{\alpha}(x), f_{\beta}(y)]'_{\alpha,\,\beta},
	\end{align}
	for all $x, y\in L ,\,\alpha,\,\beta\in \Omega$.
\end{defn}

Now, we generlize the above definitions of the $\Omega$-Lie algebra to BiHom version.

\begin{defn}\label{Def-BiHomLie}
	A $\mathbf{BiHom}$-$\mathbf{\Omega}$-$\mathbf{Lie \, algebra}$ $( L,
	\left\lbrace\cdot , \cdot \right\rbrace_{\alpha,\,\beta} ,p_{\alpha} ,q_{\alpha} )_{\alpha,\,\beta\in \Omega} $ is a vector space $L$ equipped with a family of bilinear maps $ (\lbrace\cdot , \cdot \rbrace_{\alpha,\,\beta}
	:L\times L\rightarrow L)_{\alpha,\,\beta\in \Omega}$ and two commuting $\Omega$-Lie algebra morphisms $p_{\alpha},q_{\alpha} :L\rightarrow L$
	such that
	\begin{align}\label{BHO-skew}
		\lbrace q_{\alpha}(x),p_{\beta}(y)\rbrace_{\alpha,\,\beta}=-\lbrace q_{\beta}(y),p_{\alpha}(x)\rbrace_{\beta,\,\alpha},
		\quad \text{ (BiHom-$\Omega$-skew-symmetry)}
	\end{align}
		\begin{align*}
		\lbrace q_{\alpha}^{2}(x) ,\lbrace q_{\beta}(y)
		,p_{\gamma}(z)\rbrace_{\beta,\,\gamma} \rbrace_{\alpha,\,\beta\gamma} +\lbrace q_{\beta}^{2}(y) ,\lbrace q_{\gamma}(z)
		,p_{\alpha}(x) \rbrace_{\gamma,\alpha} \rbrace_{\beta,\,\gamma\alpha} +\lbrace q_{\gamma}^{2}(z) ,\lbrace q_{\alpha}(x) ,p_{\beta}(y
		) \rbrace_{\alpha,\,\beta} \rbrace_{\gamma,\alpha\beta} =0,
	\end{align*}
\begin{align}\label{BHO-Jacobi}
	\hspace{7cm}\text{(BiHom-$\Omega$-Jacobi condition)}
\end{align}
	for all $x, y, z\in L,\,\alpha,\,\beta,\,\gamma\in \Omega$.	The maps $p_{\alpha} $ and $q_{\alpha} $ (in this order) are called the structure maps
	of $L$.
\end{defn}
\begin{defn}
	Let	$ ( L, \lbrace\cdot , \cdot \rbrace_{\alpha,\,\beta} ,p_{\alpha} ,q_{\alpha} )_{\alpha,\,\beta\in \Omega}$
	and
	$( L^{\prime},\lbrace\cdot , \cdot \rbrace_{\alpha,\,\beta}^{\prime },p ^{\prime}_{\alpha},q
	^{\prime}_{\alpha})_{\alpha,\,\beta\in \Omega}$ be two BiHom-$\Omega$-Lie algebras. A family of linear maps $ (f_{\alpha})_{\alpha\in \Omega}:L \rightarrow L' $ is called a $\mathbf{BiHom}$-$\mathbf{\Omega}$-$\mathbf{Lie \, algebra \, morphism}$ if $f_{\alpha}$ is an $\Omega$-Lie algebra morphism and
	\[p_{\alpha}^{\prime }\circ f_{\alpha}=f_{\alpha}\circ p_{\alpha} ,\quad q_{\alpha}^{\prime }\circ f_{\alpha}=f_{\alpha}\circ q_{\alpha},\quad \text{for all } \alpha\in \Omega.\]
\end{defn}

\noindent Similar to~\cite[Proposition 3.16]{gmmp}, here we get the similar property of $\Omega$-Lie algebras as follows.

\begin{prop}\label{Ome-Lie-Yautwist}
	Let $ (L,[\cdot , \cdot ]_{\alpha,\,\beta} )_{\alpha,\,\beta\in \Omega} $ be an $\Omega$-Lie algebra and let $p_{\alpha} ,q_{\alpha} :L\rightarrow L$ be two families of commuting $\Omega$-Lie algebra morphisms. We define the bilinear maps $ \left\{\cdot , \cdot \right\} _{\alpha,\,\beta}:L\times L\rightarrow L $ by
	\[\left\{ x, y\right\}_{\alpha,\,\beta} :=[ p_{\alpha}(x),q_{\beta}(y)]_{\alpha,\,\beta},  \quad\text{ for all }x,y\in L, \,\alpha,\,\beta\in \Omega.\]
	Then $(L, \left\{\cdot , \cdot \right\}_{\alpha,\,\beta}, p_{\alpha} ,
	q_{\alpha} )_{\alpha,\,\beta\in \Omega}$ is a BiHom-$\Omega$-Lie algebra, called the Yau twist of $(L,[\cdot , \cdot ]_{\alpha,\,\beta} )_{\alpha,\,\beta\in \Omega} $.
\end{prop}
\begin{proof}
	For $ x,y,z\in L,\,\alpha,\,\beta,\,\gamma\in \Omega,$ owing to the commutativity, we get $ p_{\alpha}\circ q_{\alpha}=q_{\alpha} \circ p_{\alpha}.$ First, we prove $p_{\alpha},q_{\alpha}$ are the $\Omega$-Lie algebra endomorphisms on $ (L,\lbrace \cdot,\cdot\rbrace_{\alpha,\,\beta})_{\alpha,\,\beta\in \Omega} $, we have
	\begin{align*}
		p_{\alpha\,\beta}( \lbrace x, y \rbrace _{\alpha,\,\beta} )=&p_{\alpha\,\beta}([p_{\alpha}(x),q_{\beta}(y)]_{\alpha,\,\beta})\\
		=&[p_{\alpha}^{2}(x),p_{\beta}q_{\beta}(y)]_{\alpha,\,\beta}\hspace{1cm}\text{(by $p_{\alpha\,\beta}$ satisfying Eq.~(\ref{Omega-Lie-morphism})})\\
		=&[p_{\alpha}^{2}(x),q_{\beta}p_{\beta}(y)]_{\alpha,\,\beta}\hspace{1cm}\text{(by $p_{\beta}\circ q_{\beta}=q_{\beta}\circ p_{\beta}$ })\\
		=&\lbrace p_{\alpha}(x), p_{\beta}(y) \rbrace _{\alpha,\,\beta}.
	\end{align*}

\noindent Similarly, we get
$ q_{\alpha\,\beta}(\lbrace x, y \rbrace _{\alpha,\,\beta} )=\lbrace q_{\alpha}(x), q_{\beta}(y) \rbrace _{\alpha,\,\beta}.$ Next, we prove the BiHom-$\Omega$-skew-symmetry, we have %for operations $ \lbrace \cdot , \cdot \rbrace_{\alpha,\,\beta} $.
%	On the one hand,we have
	\begin{align*}
		\lbrace q_{\alpha}(x), p_{\beta}(y) \rbrace_{\alpha,\,\beta}=&[p_{\alpha}q_{\alpha}(x), q_{\beta}p_{\beta}(y)]_{\alpha,\,\beta}\\
		=&-[q_{\beta}p_{\beta}(y),p_{\alpha}q_{\alpha}(x)]_{\beta,\,\alpha}\hspace{1cm}( \text{by Eq.~(\ref{Omega-Lie-skew})} )\\
		=&-[p_{\beta}q_{\beta}(y), q_{\alpha}p_{\alpha}(x)]_{\beta,\,\alpha}\hspace{1cm}( \text{by $ p_{\alpha}\circ q_{\alpha}=q_{\alpha} \circ p_{\alpha} $})\\
		=&-\lbrace q_{\beta}(y), p_{\alpha}(x) \rbrace _{\beta ,\,\alpha}.
	\end{align*}
Finally, we prove the BiHom-$\Omega$-Jcaobi condition, we have
	\begin{align*}
		\lbrace& q_{\alpha}^{2}(x), \lbrace q_{\beta}(y), p_{\gamma}(z)\rbrace _{\beta,\,\gamma} \rbrace_{\alpha,\,\beta\gamma}+\lbrace q_{\beta}^{2}(y), \lbrace q_{\gamma}(z), p_{\alpha}(x) \rbrace_{\gamma,\,\alpha} \rbrace_{\beta,\,\gamma\,\alpha}+\lbrace q_{\gamma}^{2}(z), \lbrace q_{\alpha}(x), p_{\beta}(y) \rbrace_{\alpha,\,\beta} \rbrace_{\gamma,\,\alpha\,\beta}\\
		=&\lbrace q_{\alpha}^{2}(x),[p_{\beta}q_{\beta}(y),q_{\gamma}p_{\gamma}(z)]_{\beta,\,\gamma}\rbrace_{\alpha,\,\beta\,\gamma}+\lbrace q_{\beta}^{2}(y),[p_{\gamma}q_{\gamma}(z),q_{\alpha}p_{\alpha}(x)]_{\gamma,\,\alpha}\rbrace_{\beta,\,\gamma\,\alpha}\\
		&+\lbrace q_{\gamma}^{2}(z),[p_{\alpha}q_{\alpha}(x),q_{\beta}p_{\beta}(y)]_{\alpha,\,\beta}\rbrace_{\gamma,\,\alpha\,\beta}\\
		=&\lbrace q_{\alpha}^{2}(x), [p_{\beta}q_{\beta}(y), p_{\gamma}q_{\gamma}(z)]_{\beta,\,\gamma}\rbrace_{\alpha,\,\beta\,\gamma}+\lbrace q_{\beta}^{2}(y), [p_{\gamma}q_{\gamma}(z), p_{\alpha}q_{\alpha}(x)]_{\gamma,\,\alpha}\rbrace_{\beta,\,\gamma\,\alpha}\\
		&+\lbrace q_{\gamma}^{2}(z), [p_{\alpha}q_{\alpha}(x), p_{\beta}q_{\beta}(y)]_{\alpha,\,\beta}\rbrace_{\gamma,\,\alpha\,\beta}\hspace{1cm}\text{(by $ p_{\alpha}\circ q_{\alpha}=q_{\alpha}\circ p_{\alpha} $)}\\
		=&[p_{\alpha}q_{\alpha}^{2}(x),q_{\beta\,\gamma}([p_{\beta}q_{\beta}(y),p_{\gamma}q_{\gamma}(z)]_{\beta,\,\gamma})]_{\alpha,\,\beta\,\gamma}+[p_{\beta}q_{\beta}^{2}(y),q_{\gamma\,\alpha}([p_{\gamma}q_{\gamma}(z),p_{\alpha}q_{\alpha}(x)]_{\gamma,\,\alpha})]_{\beta,\,\gamma\,\alpha}\\
		&+[p_{\gamma}q_{\gamma}^{2}(z),q_{\alpha\,\beta}([p_{\alpha}q_{\alpha}(x),q_{\beta}p_{\beta}(y)]_{\alpha,\,\beta})]_{\gamma,\,\alpha\,\beta}\\
		=&[p_{\alpha}q_{\alpha}^{2}(x), [p_{\beta}q_{\beta}^{2}(y), p_{\gamma}q_{\gamma}^{2}(z)]_{\beta,\,\gamma}]_{\alpha,\,\beta\,\gamma}+[p_{\beta}q_{\beta}^{2}(y), [p_{\gamma}q_{\gamma}^{2}(z), p_{\alpha}q_{\alpha}^{2}(x)]_{\gamma,\,\alpha}]_{\beta,\,\gamma\,\alpha}\\
		&+[p_{\gamma}q_{\gamma}^{2}(z), [p_{\alpha}q_{\alpha}^{2}(x), p_{\beta}q_{\beta}^{2}(y)]_{\alpha,\,\beta}]_{\gamma,\,\alpha\,\beta}\hspace{0.5cm}\text{(by $q_{\beta\,\gamma},q_{\gamma\,\alpha},q_{\alpha\,\beta}$ satisfying Eq.~(\ref{Omega-Lie-morphism})})\\
		=&0. \hspace{2cm} \text{(by Eq.~(\ref{Omega-Lie-Jacobi}))}
	\end{align*}
\end{proof}

Inspired by~\cite[Claim 3.17]{bihomprelie}, we have the following result and the proof is similar to Proposition~\ref{Ome-Lie-Yautwist}.

\begin{prop}
	Let $(L,\lbrace \cdot , \cdot \rbrace _{\alpha,\,\beta} ,p_{\alpha} ,q_{\alpha} )_{\alpha,\,\beta\in \Omega} $ be a
	BiHom-$\Omega$-Lie algebra. If $p_{\alpha}^{\prime } ,\, q_{\alpha}^{\prime }:L\rightarrow L$ are two BiHom-$\Omega$-Lie algebra morphisms and the maps $p_{\alpha}, q_{\alpha} , p_{\alpha}^{\prime }, q_{\alpha}^{\prime }$ commute with each other. Define the bilinear maps $ \langle\cdot,\cdot\rangle_{\alpha,\,\beta}: L \times L \rightarrow L $ by
	\[ \langle x,y\rangle_{\alpha,\,\beta}:=\lbrace p_{\alpha}'(x),q_{\beta}'(y)\rbrace _{\alpha,\,\beta},\]
 for all $ x,y\in L,\,\alpha,\,\beta\in \Omega,$ then $(L,\langle \cdot,\cdot \rangle_{\alpha,\,\beta},
	p_{\alpha}\circ p_{\alpha}^{\prime },q_{\alpha}\circ q_{\alpha}^{\prime })_{\alpha,\,\beta \in \Omega} $ is a
	BiHom-$\Omega$-Lie algebra.
\end{prop}

The following result precises the link between BiHom-$\Omega$-pre-Lie algebras and BiHom-$\Omega$-Lie algebras.

\begin{prop}\label{pLf-to-OmegaLie}
	Let $(A, \blacktriangleright_{\alpha,\,\beta} , p_{\alpha} , q_{\alpha} )_{\alpha,\,\beta \in \Omega}$ be a BiHom-$\Omega$-pre-Lie algebra. If $p_{\alpha},\,q_{\alpha} $ are bijective and we define $\lbrace \cdot ,\cdot  \rbrace _{\alpha,\,\beta}:A\times A\rightarrow A$ by \[\lbrace x, y\rbrace _{\alpha,\,\beta}:=x\blacktriangleright_{\alpha,\,\beta} y-(p_{\beta}^{-1}q_{\beta}(y))\blacktriangleright_{\beta,\,\alpha} (p_{\alpha} q_{\alpha}^{-1}(x)),\]
	for all $x, y\in A,\,\alpha,\,\beta\in\Omega$, then $(A, \lbrace \cdot ,\cdot  \rbrace _{\alpha,\,\beta} , p_{\alpha} , q_{\alpha} )_{\alpha,\,\beta\in\Omega}$ is a BiHom-$\Omega$-Lie algebra.
\end{prop}
\begin{proof}
		For any	$x, y, z\in A, \;\alpha,\,\beta,\,\gamma\in \Omega$, we have
		\begin{align*}
		p_{\alpha\,\beta}(\lbrace x,y\rbrace _{\alpha,\,\beta})=&p_{\alpha\,\beta}(x\blacktriangleright_{\alpha,\,\beta}y-p_{\beta}^{-1}q_{\beta}(y)\blacktriangleright_{\beta,\,\alpha}p_{\alpha}q_{\alpha}^{-1}(x))\\
		=&p_{\alpha\,\beta}(x\blacktriangleright_{\alpha,\,\beta}y)-p_{\alpha\,\beta}(p_{\beta}^{-1}q_{\beta}(y)\blacktriangleright_{\beta,\,\alpha}p_{\alpha}q_{\alpha}^{-1}(x))\\
		=&p_{\alpha}(x)\blacktriangleright_{\alpha,\,\beta}p_{\beta}(y)-p_{\beta}p_{\beta}^{-1}q_{\beta}(y)\blacktriangleright_{\beta,\,\alpha}p_{\alpha}^{2}q_{\alpha}^{-1}(x)\hspace{1cm}\text{(by $p_{\alpha\,\beta}$ satisfying Eq.~(\ref{Omega-prelie-morphism}))}\\
		=&p_{\alpha}(x)\blacktriangleright_{\alpha,\,\beta}p_{\beta}(y)-p_{\beta}^{-1}q_{\beta}p_{\beta}(y)\blacktriangleright_{\beta,\,\alpha}p_{\alpha}q_{\alpha}^{-1}p_{\alpha}(x)\hspace{1cm}\text{(by $ p_{\alpha}\circ q_{\alpha}=q_{\alpha}\circ p_{\alpha}$)}\\
		=&\lbrace p_{\alpha}(x),p_{\beta}(y)\rbrace _{\alpha,\,\beta}.
	\end{align*}
	Similarly, we get $ q_{\alpha\,\beta}(\lbrace x,y\rbrace _{\alpha,\,\beta})=\lbrace q_{\alpha}(x),q_{\beta}(y)\rbrace _{\alpha,\,\beta}.$
	Now we prove the BiHom-$\Omega$-skew-symmetry, we have
	\begin{align*}
		\lbrace q_{\alpha}(x),p_{\beta}(y)\rbrace _{\alpha,\,\beta}=&q_{\alpha}(x)\blacktriangleright_{\alpha,\,\beta}p_{\beta}(y)-p_{\beta}^{-1}q_{\beta}p_{\beta}(y)\blacktriangleright_{\beta,\,\alpha}p_{\alpha}q_{\alpha}^{-1}q_{\alpha}(x)\\
		=&q_{\alpha}(x)\blacktriangleright_{\alpha,\,\beta}p_{\beta}(y)-q_{\beta}(y)\blacktriangleright_{\beta,\,\alpha}p_{\alpha}(x)\\
		=&-(q_{\beta}(y)\blacktriangleright_{\beta,\,\alpha}p_{\alpha}(x)-p_{\alpha}^{-1}q_{\alpha}p_{\alpha}(x)\blacktriangleright_{\alpha,\,\beta}p_{\beta}q_{\beta}^{-1}q_{\beta}(y))\\
		=&-\lbrace q_{\beta}(y),p_{\alpha}(x)\rbrace _{\beta,\,\alpha}.
	\end{align*}

\noindent Next, we are going to prove the BiHom-$\Omega$-Jacobi condition, we have
	\begin{align*}
		\lbrace &q_{\alpha}^2(x),\lbrace q_{\beta}(y),p_{\gamma}(z)\rbrace _{\beta,\,\gamma}\rbrace _{\alpha,\,\beta\gamma}+\lbrace q_{\beta}^2(y),\lbrace q_{\gamma}(z),p_{\alpha}(x)\rbrace _{\gamma,\alpha}\rbrace _{\beta,\,\gamma\alpha}+\lbrace q_{\gamma}^2(z),\lbrace q_{\alpha}(x),p_{\beta}(y)\rbrace _{\alpha,\,\beta}\rbrace _{\gamma,\alpha\beta}\\
		=&\lbrace q_{\alpha}^{2}(x),q_{\beta}(y)\blacktriangleright_{\beta,\,\gamma}p_{\gamma}(z)-p_{\gamma}^{-1}q_{\gamma}p_{\gamma}(z)\blacktriangleright_{\gamma,\,\beta}p_{\beta}q_{\beta}^{-1}q_{\beta}(y)\rbrace _{\alpha,\,\beta\,\gamma}\\
		&+\lbrace q_{\beta}^{2}(y),q_{\gamma}(z)\blacktriangleright_{\gamma,\,\alpha}p_{\alpha}(x)-p_{\alpha}^{-1}q_{\alpha}p_{\alpha}(x)\blacktriangleright_{\alpha,\,\gamma}p_{\gamma}q_{\gamma}^{-1}q_{\gamma}(z)\rbrace _{\beta,\,\gamma\,\alpha}\\
		&+\lbrace q_{\gamma}^{2}(z),q_{\alpha}(x)\blacktriangleright_{\alpha,\,\beta}p_{\beta}(y)-p_{\beta}^{-1}q_{\beta}p_{\beta}(y)\blacktriangleright_{\beta,\,\alpha}p_{\alpha}q_{\alpha}^{-1}q_{\alpha}(x)\rbrace _{\gamma,\,\alpha\,\beta}\\
		=&\lbrace q_{\alpha}^2(x),q_{\beta}(y)\blacktriangleright_{\beta,\,\gamma}p_{\gamma}(z)-q_{\gamma}(z)\blacktriangleright_{\gamma,\,\beta}p_{\beta}(y)\rbrace _{\alpha,\,\beta\gamma}+\lbrace q_{\beta}^2(y),q_{\gamma}(z)\blacktriangleright_{\gamma,\,\alpha}p_{\alpha}(x)-q_{\alpha}(x)\blacktriangleright_{\alpha,\,\gamma}p_{\gamma}(z)\rbrace _{\beta,\,\gamma\alpha}\\
		&+\lbrace q_{\gamma}^2(z),q_{\alpha}(x)\blacktriangleright_{\alpha,\,\beta}p_{\beta}(y)-q_{\beta}(y)\blacktriangleright_{\beta,\,\alpha}p_{\alpha}(x)\rbrace _{\gamma,\alpha\beta}\\
		=&q_{\alpha}^2(x)\blacktriangleright_{\alpha,\,\beta\,\gamma}(q_{\beta}(y)\blacktriangleright_{\beta,\,\gamma}p_{\gamma}(z)-q_{\gamma}(z)\blacktriangleright_{\gamma,\,\beta}p_{\beta}(y))-p_{\beta\,\gamma}^{-1}q_{\beta\,\gamma}(q_{\beta}(y)\blacktriangleright_{\beta,\,\gamma}p_{\gamma}(z)\\
		&-q_{\gamma}(z)\blacktriangleright_{\gamma,\,\beta}p_{\beta}(y))\blacktriangleright_{\beta\,\gamma,\,\alpha}p_{\alpha}q_{\alpha}^{-1}q_{\alpha}^{2}(x)+q_{\beta}^{2}(y)\blacktriangleright_{\beta,\,\gamma\,\alpha}(q_{\gamma}(z)\blacktriangleright_{\gamma,\,\alpha}p_{\alpha}(x)-q_{\alpha}(x)\blacktriangleright_{\alpha,\,\gamma}p_{\gamma}(z))\\
		&-p_{\gamma\,\alpha}^{-1}q_{\gamma\,\alpha}(q_{\gamma}(z)\blacktriangleright_{\gamma,\,\alpha}p_{\alpha}(x)-q_{\alpha}(x)\blacktriangleright_{\alpha,\,\gamma}p_{\gamma}(z))\blacktriangleright_{\gamma\alpha,\,\beta}p_{\beta}q_{\beta}^{-1}q_{\beta}^{2}(y)\\
		&+q_{\gamma}^{2}(z)\blacktriangleright_{\gamma,\,\alpha\,\beta}(q_{\alpha}(x)\blacktriangleright_{\alpha,\,\beta }p_{\beta}(y)-q_{\beta}(y)\blacktriangleright_{\beta,\,\alpha}p_{\alpha}(x))-p_{\alpha\,\beta}^{-1}q_{\alpha\,\beta}(q_{\alpha}(x)\blacktriangleright_{\alpha,\,\beta}p_{\beta}(y)\\
		&-q_{\beta}(y)\blacktriangleright_{\beta,\,\alpha}p_{\alpha}(x))\blacktriangleright_{\alpha\,\beta,\,\gamma}p_{\gamma}q_{\gamma}^{-1}q_{\gamma}^{2}(z)\\
		=&q_{\alpha}^{2}(x)\blacktriangleright_{\alpha,\,\beta\,\gamma}(q_{\beta}(y)\blacktriangleright_{\beta,\,\gamma}p_{\gamma}(z))-q_{\alpha}^{2}(x)\blacktriangleright_{\alpha,\,\gamma\,\beta}(q_{\gamma}(z)\blacktriangleright_{\gamma,\,\beta}p_{\beta}(y))\\
		&-(p_{\beta}^{-1}q_{\beta}^{2}(y)\blacktriangleright_{\beta,\,\gamma}q_{\gamma}(z))\blacktriangleright_{\beta\,\gamma,\,\alpha}p_{\alpha}q_{\alpha}(x)+(p_{\gamma}^{-1}q_{\gamma}^{2}(z)\blacktriangleright_{\gamma,\,\beta}q_{\beta}(y))\blacktriangleright_{\gamma\,\beta,\,\alpha}p_{\alpha}q_{\alpha}(x)\\
		&+q_{\beta}^{2}(y)\blacktriangleright_{\beta,\,\gamma\,\alpha}(q_{\gamma}(z)\blacktriangleright_{\gamma,\,\alpha}p_{\alpha}(x))-q_{\beta}^{2}(y)\blacktriangleright_{\beta,\,\alpha\,\gamma}(q_{\alpha}(x)\blacktriangleright_{\alpha,\,\gamma}p_{\gamma}(z))\\
		&-(p_{\gamma}^{-1}q_{\gamma}^{2}(z)\blacktriangleright_{\gamma,\,\alpha}q_{\alpha}(x))\blacktriangleright_{\gamma\,\alpha,\,\beta}p_{\beta}q_{\beta}(y)+(p_{\alpha}^{-1}q_{\alpha}^{2}(x)\blacktriangleright_{\alpha,\,\gamma}q_{\gamma}(z))\blacktriangleright_{\alpha\,\gamma,\,\beta}p_{\beta}q_{\beta}(y)\\
		&+q_{\gamma}^{2}(z)\blacktriangleright_{\gamma,\,\alpha\,\beta}(q_{\alpha}(x)\blacktriangleright_{\alpha,\,\beta}p_{\beta}(y))-q_{\gamma}^{2}(z)\blacktriangleright_{\gamma,\,\beta\,\alpha}(q_{\beta}(y)\blacktriangleright_{\beta,\,\alpha}p_{\alpha}(x))\\
		&-(p_{\alpha}^{-1}q_{\alpha}^{2}(x)\blacktriangleright_{\alpha,\,\beta }q_{\beta}(y))\blacktriangleright_{\alpha\,\beta,\,\gamma}p_{\gamma}q_{\gamma}(z)+(p_{\beta}^{-1}q_{\beta}^{2}(y)\blacktriangleright_{\beta,\,\alpha}q_{\alpha}(x))\blacktriangleright_{\beta\,\alpha,\,\gamma}p_{\gamma}q_{\gamma}(z)\\
		&\hspace{1.3cm}\text{(by $p_{\beta\,\gamma}^{-1},q_{\beta\,\gamma},p_{\gamma\,\alpha}^{-1},q_{\gamma\,\alpha},p_{\alpha\,\beta}^{-1},q_{\alpha\,\beta}$ satisfying Eq.~(\ref{Omega-prelie-morphism})})\\
		=&p_{\alpha}q_{\alpha}p_{\alpha}^{-1}q_{\alpha}(x)\blacktriangleright_{\alpha,\,\beta\,\gamma}(p_{\beta}p_{\beta}^{-1}q_{\beta}(y)\blacktriangleright_{\beta,\,\gamma}p_{\gamma}(z))-(q_{\alpha}p_{\alpha}^{-1}q_{\alpha}(x)\blacktriangleright_{\alpha,\,\beta}p_{\beta}p_{\beta}^{-1}q_{\beta}(y))\blacktriangleright_{\alpha\,\beta,\,\gamma}q_{\gamma}p_{\gamma}(z)\\
		&-p_{\beta}q_{\beta}p_{\beta}^{-1}q_{\beta}(y)\blacktriangleright_{\beta,\,\alpha\,\gamma}(p_{\alpha}p_{\alpha}^{-1}q_{\alpha}(x)\blacktriangleright_{\alpha,\,\gamma}p_{\gamma}(z))+(q_{\beta}p_{\beta}^{-1}q_{\beta}(y)\blacktriangleright_{\beta,\,\alpha}p_{\alpha}p_{\alpha}^{-1}q_{\alpha}(x))\blacktriangleright_{\beta\,\alpha,\,\gamma}q_{\gamma}p_{\gamma}(z)\\
		&-p_{\alpha}q_{\alpha}p_{\alpha}^{-1}q_{\alpha}(x)\blacktriangleright_{\alpha,\,\gamma\,\beta}(p_{\gamma}p_{\gamma}^{-1}q_{\gamma}(z)\blacktriangleright_{\gamma,\,\beta}p_{\beta}(y))+(q_{\alpha}p_{\alpha}^{-1}q_{\alpha}(x)\blacktriangleright_{\alpha,\,\gamma}p_{\gamma}p_{\gamma}^{-1}q_{\gamma}(z))\blacktriangleright_{\alpha\,\gamma,\,\beta}q_{\beta}p_{\beta}(y)\\
		&+p_{\gamma}q_{\gamma}p_{\gamma}^{-1}q_{\gamma}(z)\blacktriangleright_{\gamma,\,\alpha\,\beta}(p_{\alpha}p_{\alpha}^{-1}q_{\alpha}(x)\blacktriangleright_{\alpha,\,\beta}p_{\beta}(y))-(q_{\gamma}p_{\gamma}^{-1}q_{\gamma}(z)\blacktriangleright_{\gamma,\,\alpha}p_{\alpha}p_{\alpha}^{-1}q_{\alpha}(x))\blacktriangleright_{\gamma\,\alpha,\,\beta}q_{\beta}p_{\beta}(y)\\
		&+p_{\beta}q_{\beta}p_{\beta}^{-1}q_{\beta}(y)\blacktriangleright_{\beta,\,\gamma\,\alpha}(p_{\gamma}p_{\gamma}^{-1}q_{\gamma}(z)\blacktriangleright_{\gamma,\,\alpha}p_{\alpha}(x))-(q_{\beta}p_{\beta}^{-1}q_{\beta}(y)\blacktriangleright_{\beta,\,\gamma}p_{\gamma}p_{\gamma}^{-1}q_{\gamma}(z))\blacktriangleright_{\beta\,\gamma,\,\alpha}q_{\alpha}p_{\alpha}(x)\\
		&-p_{\gamma}q_{\gamma}p_{\gamma}^{-1}q_{\gamma}(z)\blacktriangleright_{\gamma,\,\beta\,\alpha}(p_{\beta}p_{\beta}^{-1}q_{\beta}(y)\blacktriangleright_{\beta,\,\alpha}p_{\alpha}(x))+(q_{\gamma}p_{\gamma}^{-1}q_{\gamma}(z)\blacktriangleright_{\gamma,\,\beta}p_{\beta}p_{\beta}^{-1}q_{\beta}(y))\blacktriangleright_{\gamma\,\beta,\,\alpha}q_{\alpha}p_{\alpha}(x)\\
		&\hspace{1.3cm}\text{(by $p_{\alpha},p_{\alpha}',q_{\alpha},q_{\alpha}'$ commuting with each other})\\
		=&0. \hspace{1cm}\text{(by Eq.~(\ref{BHO-prelie-alg}))}
	\end{align*}
	This completes the proof.
\end{proof}

Similar to associative algebras induce Lie algebras, now we generalize this classical result to
 BiHom-$\Omega$ version.

\begin{prop}\label{commutator}
	Let $ (A,\bullet_{\alpha,\,\beta},p_{\alpha},q_{\alpha})_{\alpha,\,\beta\in \Omega} $ be a BiHom-$\Omega$-associative algebra. If $ p_{\alpha}, q_{\alpha} $ are bijective and we define
	\begin{align*}%\label{commutator-Eq}
	\lbrace x,y\rbrace_{\alpha,\,\beta}:=x\bullet_{\alpha,\,\beta}y-p_{\beta}^{-1}q_{\beta}(y)\bullet_{\beta,\,\alpha}p_{\alpha}q_{\alpha}^{-1}(x),
  \end{align*}
	for all $ x,y\in A,\,\alpha,\,\beta\in \Omega, $ then $ (A,\lbrace \cdot,\cdot\rbrace _{\alpha,\,\beta},p_{\alpha},q_{\alpha})_{\alpha,\,\beta\in \Omega} $ is a BiHom-$\Omega$-Lie algebra.
\end{prop}
\begin{proof}
	It follows from Proposition~\ref{asso-preLie} and \ref{pLf-to-OmegaLie}.
\end{proof}

Similar to Theorem \ref{asso-RBO-asso}, we obtain a new BiHom-$\Omega$-Lie algebra by defining a new binary operation by the Rota-Baxter family of weight $\lambda$ on the BiHom-$\Omega$-Lie algebra.

\begin{theorem}\label{Proposition:BiHLie}
	Let $( L,\lbrace \cdot , \cdot \rbrace_{\alpha,\,\beta}
	,p_{\alpha} ,q_{\alpha})_{\alpha,\,\beta\in \Omega} $ be a BiHom-$\Omega$-Lie
	algebra. If	$( R_{\alpha})_{\alpha\in \Omega}$ is a Rota-Baxter family of weight $\lambda $ on $L$ satisfying \[ R_{\alpha}\circ p_{\alpha} =p_{\alpha} \circ R_{\alpha} \; \text{and} \; R_{\alpha}\circ q_{\alpha} =q_{\alpha} \circ R_{\alpha}.\]
	Define a new operation on $L$ by
%	We define the multiplication on $L$ by
	\begin{equation*}
		\langle x,y\rangle_{\alpha,\,\beta}:=\lbrace R_{\alpha}(x),y\rbrace _{\alpha,\,\beta}+\lbrace x,R_{\beta}(y)\rbrace _{\alpha,\,\beta}+\lambda \lbrace x,y\rbrace _{\alpha,\,\beta},
	\end{equation*}
for all $x,y\in L,\,\alpha,\,\beta\in \Omega$, then $(L,\left\langle \cdot , \cdot \right\rangle_{\alpha,\,\beta} ,p_{\alpha},q_{\alpha} )_{\alpha,\,\beta\in \Omega}$ is a BiHom-$\Omega$-Lie algebra.
\end{theorem}
\begin{proof}
	For all $ x,y,z\in L,\,\alpha,\,\beta,\,\gamma\in \Omega ,$ we have
	\begin{align*}
		\langle p_{\alpha}(x),p_{\beta}(y)\rangle_{\alpha,\,\beta} =&\lbrace R_{\alpha}(p_{\alpha}(x)),p_{\beta}(y)\rbrace _{\alpha,\,\beta}+\lbrace p_{\alpha} (x),R_{\beta}(p_{\beta}(y))\rbrace _{\alpha,\,\beta}+\lambda \lbrace p_{\alpha}(x),p_{\beta}(y)\rbrace _{\alpha,\,\beta} \\
		=&\lbrace p_{\alpha}(R_{\alpha}(x)),p_{\beta}(y)\rbrace _{\alpha,\,\beta}+\lbrace p_{\alpha}(x),p_{\beta}(R_{\beta}(y))\rbrace _{\alpha,\,\beta}+\lambda \lbrace p_{\alpha}(x),p_{\beta}(y)\rbrace _{\alpha,\,\beta}\\
		&\hspace{1cm} \text{(by $R_{\alpha}\circ p_{\alpha} =p_{\alpha} \circ R_{\alpha}$)} \\
		=&p_{\alpha\,\beta}(\lbrace R_{\alpha}(x),y\rbrace _{\alpha,\,\beta})+p_{\alpha\,\beta}(\lbrace x,R_{\beta}(y)\rbrace _{\alpha,\,\beta})+\lambda p_{\alpha\,\beta}(\lbrace x,y\rbrace _{\alpha,\,\beta})\\
		&\hspace{1cm} \text{(by $p_{\alpha\,\beta}$ being a BiHom-$\Omega$-Lie algebra morphism))}\\
		=&p_{\alpha\,\beta}(\lbrace R_{\alpha}(x),y\rbrace _{\alpha,\,\beta}+\lbrace x,R_{\beta}(y)\rbrace _{\alpha,\,\beta}+\lambda\lbrace x,y\rbrace _{\alpha,\,\beta})\\
		=&p_{\alpha\,\beta}(\left\langle x,y\right\rangle_{\alpha,\,\beta} ).
	\end{align*}
	Similarly, we get $ q_{\alpha\,\beta}(\langle x,y\rangle_{\alpha,\,\beta} )=\langle
	q_{\alpha}(x),q_{\beta}(y)\rangle_{\alpha,\,\beta}. $ For the BiHom-$\Omega$-skew-symmetry, we have
	\begin{align*}
		\langle q_{\alpha}(x),p_{\beta}(y)\rangle _{\alpha,\,\beta}=&\lbrace R_{\alpha}(
		q_{\alpha}(x)),p_{\beta} (y)\rbrace _{\alpha,\,\beta}+\lbrace q_{\alpha}(x),R_{\beta}(p_{\beta}(y))\rbrace _{\alpha,\,\beta}+\lambda \lbrace q_{\alpha} (x),p_{\beta} (
		y)\rbrace  _{\alpha,\,\beta}\\
		=&\lbrace q_{\alpha}(R_{\alpha}(x)),p_{\beta}		(y)\rbrace _{\alpha,\,\beta}+\lbrace q_{\alpha}(x),p_{\beta}(R_{\beta}(y))
		\rbrace  _{\alpha,\,\beta}+\lambda \lbrace q_{\alpha} (x),p_{\beta} (y)\rbrace  _{\alpha,\,\beta}\\
		&\hspace{1cm} \text{(by $R_{\alpha}\circ p_{\alpha} =p_{\alpha} \circ R_{\alpha} $ and $ R_{\alpha}\circ q_{\alpha} =q_{\alpha} \circ R_{\alpha}$)}\\
		=&-\lbrace q_{\beta}(y),p_{\alpha}
		(R_{\alpha}(x))\rbrace _{\beta,\,\alpha}-\lbrace q_{\beta}(R_{\beta}(y)),p_{\alpha} (x)\rbrace _{\beta,\,\alpha} -\lambda \lbrace q_{\beta}(y),p_{\alpha} (x)\rbrace _{\beta,\,\alpha}\\
		&\hspace{1cm} \text{(by Eq.~(\ref{BHO-skew}))}
		\\
		=&-\lbrace q_{\beta} (y),R_{\alpha}(p_{\alpha}(x))\rbrace _{\beta,\,\alpha}-\lbrace R_{\beta}(q_{\beta}(y)),p_{\alpha}(x)
		\rbrace _{\beta,\,\alpha}-\lambda \lbrace q_{\beta}(y),p_{\alpha} (x)\rbrace _{\beta,\,\alpha} \\
		&\hspace{1cm} \text{(by $R_{\alpha}\circ p_{\alpha} =p_{\alpha}\circ R_{\alpha} $ and $ R_{\alpha}\circ q_{\alpha} =q_{\alpha} \circ R_{\alpha}$)}\\
		=&-\langle q_{\beta}(y),p_{\alpha}(x)\rangle_{\beta,\,\alpha} .
	\end{align*}

\noindent	Next, we are going to prove that $\left\langle
	\cdot , \cdot \right\rangle_{\alpha,\,\beta} $ satisfies the BiHom-$\Omega$-Jacobi condition, we have
	\begin{align*}
		\langle& q_{\alpha}^{2}( x) ,\langle q_{\beta} ( y) ,p_{\gamma}
		( z) \rangle_{\beta,\,\gamma} \rangle_{\alpha,\,\beta\gamma}\\
		=&\lbrace R_{\alpha}(q_{\alpha} ^{2}(x)),\langle q_{\beta}(
		y) ,p_{\gamma} ( z) \rangle _{\beta,\,\gamma}\rbrace _{\alpha,\,\beta\gamma}+\lbrace q_{\alpha} ^{2}(x),R_{\beta\gamma}(\langle q_{\beta}
		( y) ,p_{\gamma}( z) \rangle _{\beta,\,\gamma})\rbrace _{\alpha,\,\beta\gamma}\\
		&+\lambda \lbrace  q_{\alpha}^{2}(x),\langle q_{\beta}( y) ,p_{\gamma} ( z) \rangle_{\beta,\,\gamma} \rbrace  _{\alpha,\,\beta\gamma}\\
		=&\lbrace R_{\alpha}(q_{\alpha} ^{2}(x)),\langle q_{\beta}(
		y) ,p_{\gamma} ( z) \rangle _{\beta,\,\gamma}\rbrace _{\alpha,\,\beta\gamma}+\lambda \lbrace  q_{\alpha}^{2}(x),\langle q_{\beta}( y) ,p_{\gamma} ( z) \rangle_{\beta,\,\gamma} \rbrace  _{\alpha,\,\beta\gamma}\\
		&+\lbrace q_{\alpha}^{2}(x),R_{\beta\,\gamma}(\lbrace R_{\beta}q_{\beta}(y),p_{\gamma}(z)\rbrace _{\beta,\,\gamma}+\lbrace q_{\beta}(y),R_{\gamma}p_{\gamma}(z)\rbrace _{\beta,\,\gamma}+\lambda\lbrace q_{\beta}(y),p_{\gamma}(z)\rbrace _{\beta,\,\gamma})\rbrace _{\alpha,\,\beta\,\gamma}\\
		=&\lbrace R_{\alpha}(q_{\alpha}^{2}(x)),\langle q_{\beta}\left( y\right) ,p_{\gamma} ( z) \rangle_{\beta,\,\gamma}\rbrace _{\alpha,\,\beta\gamma}+\lbrace q_{\alpha} ^{2}(x),\lbrace
		R_{\beta}(q_{\beta}\left( y\right)) ,R_{\gamma}(p _{\gamma}\left( z\right)) \rbrace _{\beta,\,\gamma} \rbrace _{\alpha,\,\beta\gamma}\\
		&+\lambda \lbrace
		q_{\alpha}^{2}(x),\langle q_{\beta}\left( y\right) ,p_{\gamma} \left( z\right)\rangle _{\beta,\,\gamma}\rbrace _{\alpha,\,\beta\gamma}\hspace{1cm} \text{(by Eq.~(\ref{RB-Eq}))}
		\\
		=&\lbrace R_{\alpha}(q_{\alpha} ^{2}\left( x\right) ),\lbrace R_{\beta}(q_{\beta}\left( y\right) ),p_{\gamma}\left(
		z\right) \rbrace _{\beta,\,\gamma}\rbrace _{\alpha,\,\beta\gamma}+\lbrace R_{\alpha}(q_{\alpha}^{2}\left( x\right) ),\lbrace q _{\beta}\left( y\right) ,R_{\gamma}(p_{\gamma}\left( z\right) )\rbrace _{\beta,\,\gamma}\rbrace _{\alpha,\,\beta\gamma}\\
		&
		+\lbrace R_{\alpha}(q_{\alpha}^{2}\left( x\right) ),\lambda \lbrace q_{\beta}\left( y\right) ,p _{\gamma}\left( z\right) \rbrace _{\beta,\,\gamma}\rbrace _{\alpha,\,\beta\gamma}+
		\lbrace q_{\alpha}^{2}(x),\lbrace  R_{\beta}(q_{\beta} \left( y\right)) ,R_{\gamma}(p_{\gamma} \left( z\right))
		\rbrace  _{\beta,\,\gamma}\rbrace _{\alpha,\,\beta\gamma}\\
		&+\lambda \lbrace q_{\alpha}^{2}\left( x\right) ,\lbrace R_{\beta}(q_{\beta} \left( y\right)
		),p_{\gamma} \left( z\right) \rbrace _{\beta,\,\gamma}\rbrace _{\alpha,\,\beta\gamma}+\lambda \lbrace q _{\alpha}^{2}\left( x\right)
		,\lbrace q_{\beta} \left( y\right) ,R_{\gamma}(p_{\gamma} \left( z\right) )\rbrace _{\beta,\,\gamma}\rbrace _{\alpha,\,\beta\gamma}\\
		&+\lambda \lbrace q_{\alpha}^{2}\left( x\right) ,\lambda \lbrace q_{\beta} \left( y\right) ,p_{\gamma}\left(
		z\right) \rbrace _{\beta,\,\gamma}\rbrace _{\alpha,\,\beta\gamma}.
	\end{align*}

\noindent	Similarly, we have
	\begin{align*}
		\langle& q_{\beta}^{2}\left(y\right) ,\langle q _{\gamma}\left( z\right) ,p_{\alpha}\left( x\right) \rangle_{\gamma,\alpha} \rangle_{\beta,\,\gamma\alpha}\\
		=&\lbrace R_{\beta}(q_{\beta}^{2}\left( y\right) ),\lbrace R_{\gamma}(q	_{\gamma}\left( z\right) ),p_{\alpha}\left( x\right) \rbrace _{\gamma,\alpha}\rbrace _{\beta,\,\gamma\alpha}+\lbrace R_{\beta}(q_{\beta}^{2}\left( y\right)
		),\lbrace q_{\gamma}\left( z\right) ,R_{\alpha}(p_{\alpha}\left( x\right) )\rbrace _{\gamma,\alpha}\rbrace _{\beta,\,\gamma\alpha}\\
		&+\lbrace R_{\beta}(q_{\beta}^{2}\left(
		y\right) ),\lambda \lbrace q_{\gamma} \left( z\right) ,p_{\alpha} \left( x\right) \rbrace _{\gamma,\alpha}\rbrace _{\beta,\,\gamma\alpha}+
		\lbrace q_{\beta}^{2}(y),\lbrace  R_{\gamma}(q_{\gamma} \left( z\right)) ,R_{\alpha}(p_{\alpha}\left( x\right))
		\rbrace _{\gamma,\alpha} \rbrace _{\beta,\,\gamma\alpha} \\
		&+\lambda \lbrace q_{\beta}^{2}\left( y\right) ,\lbrace R_{\gamma}(q_{\gamma} \left( z\right)
		),p_{\alpha}\left( x\right) \rbrace _{\gamma,\alpha}\rbrace _{\beta,\,\gamma\alpha}+\lambda \lbrace q_{\beta}^{2}\left( y\right)
		,\lbrace q_{\gamma}\left( z\right) ,R_{\alpha}(p_{\alpha}\left( x\right) )\rbrace _{\gamma,\alpha}\rbrace _{\beta,\,\gamma\alpha}\\
		&+\lambda \lbrace q_{\beta}^{2}\left( y\right) ,\lambda \lbrace q_{\gamma}\left( z\right) ,p_{\alpha} \left(x\right) \rbrace _{\gamma,\alpha}\rbrace _{\beta,\,\gamma\alpha},
	\end{align*}
	\begin{align*}
		\langle& q_{\gamma}^{2}\left( z\right) ,\langle q_{\alpha}\left( x\right) ,p_{\beta}\left( y\right) \rangle _{\alpha,\,\beta}\rangle_{\gamma,\alpha\beta}\\
		=&\lbrace R_{\gamma}(q_{\gamma}^{2}\left( z\right) ),\lbrace R_{\alpha}(q_{\alpha}\left( x\right) ),p_{\beta}\left( y\right) \rbrace _{\alpha,\,\beta}\rbrace _{\gamma,\alpha\beta}+\lbrace R_{\gamma}(q_{\gamma}^{2}\left( z\right)
		),\lbrace q_{\gamma} \left( x\right) ,R_{\beta}(p_{\beta}\left( y\right) )\rbrace _{\alpha,\,\beta}\rbrace _{\gamma,\alpha\beta}\\
		&+\lbrace R_{\gamma}(q_{\gamma}^{2}\left(z\right) ),\lambda \lbrace q_{\alpha} \left( x\right) ,p_{\beta}\left( y\right) \rbrace _{\alpha,\,\beta}\rbrace _{\gamma,\alpha\beta}+
		\lbrace q_{\gamma}^{2}(z),\lbrace  R_{\alpha}(q_{\alpha}\left( x\right)) ,R_{\beta}(p_{\beta}\left( y\right))
		\rbrace _{\alpha,\,\beta} \rbrace _{\gamma,\alpha\beta} \\
		&+\lambda \lbrace q_{\gamma}^{2}\left( z\right) ,\lbrace R_{\alpha}(q_{\alpha}\left( x\right)
		),p_{\beta}\left( y\right) \rbrace _{\alpha,\,\beta}\rbrace _{\gamma,\alpha\beta}+\lambda \lbrace q_{\gamma}^{2}\left( z\right)
		,\lbrace q_{\alpha}\left( x\right) ,R_{\beta}(p_{\beta}\left(y\right) )\rbrace _{\alpha,\,\beta}\rbrace _{\gamma,\alpha\beta}\\
		&+\lambda \lbrace q_{\gamma}^{2}\left( z\right) ,\lambda \lbrace q_{\alpha} \left( x\right) ,p_{\beta} \left(y\right) \rbrace _{\alpha,\,\beta}\rbrace _{\gamma,\alpha\beta}.
	\end{align*}
	By adding the items, we obtain
	\begin{align*}
		\langle& q_{\alpha}^{2}\left( x\right) ,\langle q_{\beta}\left( y\right) ,p_{\gamma}\left( z\right) \rangle_{\beta,\,\gamma} \rangle_{\alpha,\,\beta\gamma} +\langle q_{\beta}^{2}\left( y\right)
		,\langle q _{\gamma}\left( z\right) ,p_{\alpha}\left( x\right) \rangle _{\gamma,\alpha}\rangle_{\beta,\,\gamma\alpha}
		+\langle q_{\gamma}^{2}\left( z\right) ,\langle q_{\alpha}\left( x\right) ,p_{\beta}\left( y\right) \rangle _{\alpha,\,\beta}\rangle_{\gamma,\alpha\beta}\\
		=&\lbrace R_{\alpha}(q_{\alpha}^{2}\left( x\right) ),\lbrace R_{\beta}(q_{\beta}\left( y\right) ),p_{\gamma}\left(
		z\right) \rbrace _{\beta,\,\gamma}\rbrace _{\alpha,\,\beta\gamma}+\lbrace R_{\alpha}(q_{\alpha}^{2}\left( x\right) ),\lbrace q_{\beta}\left(y\right) ,R_{\gamma}(p_{\gamma}\left( z\right) )\rbrace _{\beta,\,\gamma}\rbrace _{\alpha,\,\beta\gamma}\\
		&
		+\lbrace R_{\alpha}(q_{\alpha}^{2}\left( x\right) ),\lambda \lbrace q_{\beta}\left( y\right) ,p_{\gamma}\left( z\right) \rbrace _{\beta,\,\gamma}\rbrace _{\alpha,\,\beta\gamma}+
		\lbrace q_{\alpha}^{2}(x),\lbrace  R_{\beta}(q_{\beta}\left( y\right)) ,R_{\gamma}(p_{\gamma}\left( z\right))
		\rbrace  _{\beta,\,\gamma}\rbrace _{\alpha,\,\beta\gamma}\\
		&+\lambda \lbrace q_{\alpha}^{2}\left( x\right) ,\lbrace R_{\beta}(q_{\beta}\left( y\right)
		),p_{\gamma}\left( z\right) \rbrace _{\beta,\,\gamma}\rbrace _{\alpha,\,\beta\gamma}+\lambda \lbrace q_{\alpha}^{2}\left( x\right)
		,\lbrace q_{\beta}\left( y\right) ,R_{\gamma}(p_{\gamma}\left(z\right) )\rbrace _{\beta,\,\gamma}\rbrace _{\alpha,\,\beta\gamma}\\
		&+\lambda \lbrace q_{\alpha}^{2}\left( x\right) ,\lambda \lbrace q_{\beta}\left( y\right) ,p_{\gamma}\left(z\right) \rbrace _{\beta,\,\gamma}\rbrace _{\alpha,\,\beta\gamma}\\
		&+\lbrace R_{\beta}(q_{\beta}^{2}\left( y\right) ),\lbrace R_{\gamma}(q_{\gamma}\left( z\right) ),p_{\alpha}\left( x\right) \rbrace _{\gamma,\alpha}\rbrace _{\beta,\,\gamma\alpha}+\lbrace R_{\beta}(q_{\beta}^{2}\left( y\right)
		),\lbrace q_{\gamma}\left( z\right) ,R_{\alpha}(p_{\alpha}\left(x\right) )\rbrace _{\gamma,\alpha}\rbrace _{\beta,\,\gamma\alpha}\\
		&+\lbrace R_{\beta}(q_{\beta}^{2}\left(y\right) ),\lambda \lbrace q_{\gamma}\left( z\right) ,p_{\alpha}\left( x\right) \rbrace _{\gamma,\alpha}\rbrace _{\beta,\,\gamma\alpha}+
		\lbrace q_{\beta}^{2}(y),\lbrace  R_{\gamma}(q_{\gamma}\left( z\right)) ,R_{\alpha}(p_{\alpha}\left( x\right))
		\rbrace _{\gamma,\alpha} \rbrace _{\beta,\,\gamma\alpha} \\
		&+\lambda \lbrace q_{\beta}^{2}\left( y\right) ,\lbrace R_{\gamma}(q_{\gamma} \left( z\right)
		),p_{\alpha}\left( x\right) \rbrace _{\gamma,\alpha}\rbrace _{\beta,\,\gamma\alpha}+\lambda \lbrace q_{\beta}^{2}\left( y\right)
		,\lbrace q_{\gamma}\left( z\right) ,R_{\alpha}(p_{\alpha}\left(x\right) )\rbrace _{\gamma,\alpha}\rbrace _{\beta,\,\gamma\alpha}\\
		&+\lambda \lbrace q_{\beta}^{2}\left( y\right) ,\lambda \lbrace q_{\gamma}\left( z\right) ,p_{\alpha}\left(x\right) \rbrace _{\gamma,\alpha}\rbrace _{\beta,\,\gamma\alpha}\\
		&+\lbrace R_{\gamma}(q_{\gamma}^{2}\left( z\right) ),\lbrace R_{\alpha}(q_{\alpha}\left( x\right) ),p_{\beta}\left( y\right) \rbrace _{\alpha,\,\beta}\rbrace _{\gamma,\alpha\beta}+\lbrace R_{\gamma}(q_{\gamma}^{2}\left( z\right)
		),\lbrace q_{\alpha}\left( x\right) ,R_{\beta}(p_{\beta}\left(y\right) )\rbrace _{\alpha,\,\beta}\rbrace _{\gamma,\alpha\beta}\\
		&+\lbrace R_{\gamma}(q_{\gamma}^{2}\left(z\right) ),\lambda \lbrace q_{\alpha} \left( x\right) ,p_{\beta}\left( y\right) \rbrace _{\alpha,\,\beta}\rbrace _{\gamma,\alpha\beta}+
		\lbrace q_{\gamma}^{2}(z),\lbrace  R_{\alpha}(q_{\alpha}\left( x\right)) ,R_{\beta}(p_{\beta}\left( y\right))		\rbrace _{\alpha,\,\beta} \rbrace _{\gamma,\alpha\beta} \\
		&+\lambda \lbrace q_{\gamma}^{2}\left( z\right) ,\lbrace R_{\alpha}(q_{\alpha} \left( x\right)
		),p_{\beta}\left( y\right) \rbrace _{\alpha,\,\beta}\rbrace _{\gamma,\alpha\beta}+\lambda \lbrace q_{\gamma}^{2}\left( z\right)
		,\lbrace q_{\alpha}\left( x\right) ,R_{\beta}(p_{\beta}\left(y\right) )\rbrace _{\alpha,\,\beta}\rbrace _{\gamma,\alpha\beta}\\
		&+\lambda \lbrace q_{\gamma}^{2}\left( z\right) ,\lambda \lbrace q_{\alpha} \left( x\right) ,p_{\beta} \left(y\right) \rbrace _{\alpha,\,\beta}\rbrace _{\gamma,\alpha\beta}\\
		=&\lbrace q_{\alpha}^{2}(R_{\alpha}\left( x\right) ),\lbrace q_{\beta}(R_{\beta}(y)),p_{\gamma}\left(z\right) \rbrace _{\beta,\,\gamma}\rbrace _{\alpha,\,\beta\gamma}+\lbrace q_{\alpha}^{2}\left( R_{\alpha}(x)\right) ,\lbrace q_{\beta}\left(y\right) ,p_{\gamma}(R_{\gamma}(z))\rbrace _{\beta,\,\gamma}\rbrace _{\alpha,\,\beta\gamma}\\
		&+\lbrace q_{\alpha}^{2}R_{\alpha}(x),\lambda\lbrace q_{\beta}(y),p_{\gamma}(z)\rbrace _{\beta,\,\gamma}\rbrace _{\alpha,\,\beta\gamma}+\lbrace q_{\alpha}^2(x),\lbrace q_{\beta}R_{\beta}(y),p_{\gamma}R_{\gamma}(z)\rbrace _{\beta,\,\gamma}\rbrace _{\alpha,\,\beta\gamma}\\
		&+\lambda\lbrace q_{\alpha}^2(x),\lbrace q_{\beta}R_{\beta}(y),p_{\gamma}(z)\rbrace _{\beta,\,\gamma}\rbrace _{\alpha,\,\beta\gamma}+\lambda\lbrace q_{\alpha}^2(x),\lbrace q_{\beta}(y),p_{\gamma}R_{\gamma}(z)\rbrace _{\beta,\,\gamma}\rbrace _{\alpha,\,\beta\gamma}\\
		&+\lambda\lbrace q_{\alpha}^2(x),\lambda\lbrace q_{\beta}(y),p_{\gamma}(z)\rbrace _{\beta,\,\gamma}\rbrace _{\alpha,\,\beta\gamma}+\lbrace q_{\beta}^2R_{\beta}(y),\lbrace q_{\gamma}R_{\gamma}(z),p_{\alpha}(x)\rbrace _{\gamma,\alpha}\rbrace _{\beta,\,\gamma\alpha}\\
		&+\lbrace q_{\beta}^2R_{\beta}(y),\lbrace q_{\gamma}(z),p_{\alpha}R_{\alpha}(x)\rbrace _{\gamma,\alpha}\rbrace _{\beta,\,\gamma\alpha}+\lbrace q_{\beta}^2R_{\beta}(y),\lambda\lbrace q_{\gamma}(z),p_{\alpha}(x)\rbrace _{\gamma,\alpha}\rbrace _{\beta,\,\gamma\alpha}\\
		&+\lbrace q_{\beta}^2(y),\lbrace q_{\gamma}R_{\gamma}(z),p_{\alpha}R_{\alpha}(x)\rbrace _{\gamma,\alpha}\rbrace _{\beta,\,\gamma\alpha}+\lbrace q_{\beta}^2(y),\lbrace q_{\gamma}R_{\gamma}(z),p_{\alpha}(x)\rbrace _{\gamma,\alpha}\rbrace _{\beta,\,\gamma\alpha}\\
		&+\lambda\lbrace q_{\beta}^2(y),\lbrace q_{\gamma}(z),p_{\alpha}R_{\alpha}(x)\rbrace _{\gamma,\alpha}\rbrace _{\beta,\,\gamma\alpha}+\lambda\lbrace q_{\beta}^2(y),\lambda\lbrace q_{\gamma}(z),p_{\alpha}(x)\rbrace _{\gamma,\alpha}\rbrace _{\beta,\,\gamma\alpha}\\
		&+\lbrace q_{\gamma}^2R_{\gamma}(z),\lbrace q_{\alpha}R_{\alpha}(x),p_{\beta}(y)\rbrace _{\alpha,\,\beta}\rbrace _{\gamma,\alpha\beta}+\lbrace q_{\gamma}^2R_{\gamma}(z),\lbrace q_{\alpha}(x),p_{\beta}R_{\beta}(y)\rbrace _{\alpha,\,\beta}\rbrace _{\gamma,\alpha\beta}\\
		&+\lbrace q_{\gamma}^2R_{\gamma}(z),\lambda\lbrace q_{\alpha}(x),p_{\beta}(y)\rbrace _{\alpha,\,\beta}\rbrace _{\gamma,\alpha\beta}+\lbrace q_{\gamma}^2(z),\lbrace q_{\alpha}R_{\alpha}(x),p_{\beta}R_{\beta}(y)\rbrace _{\alpha,\,\beta}\rbrace _{\gamma,\alpha\beta}\\
		&+\lambda\lbrace q_{\gamma}^2(z),\lbrace q_{\alpha}R_{\alpha}(x),p_{\beta}(y)\rbrace _{\alpha,\,\beta}\rbrace _{\gamma,\alpha\beta}+\lambda\lbrace q_{\gamma}^2(z),\lbrace q_{\alpha}(x),p_{\beta}R_{\beta}(y)\rbrace _{\alpha,\,\beta}\rbrace _{\gamma,\alpha\beta}\\
		&+\lambda\lbrace q_{\gamma}^2(z),\lambda\lbrace q_{\alpha}(x),p_{\beta}(y)\rbrace _{\alpha,\,\beta}\rbrace _{\gamma,\alpha\beta}\hspace{1cm} \text{(by $ R_{\alpha}\circ p_{\alpha}=p_{\alpha}\circ R_{\alpha} ,\, R_{\alpha}\circ q_{\alpha}=q_{\alpha}\circ R_{\alpha}$)}\\
	%	=&0+0+0+0+0+0+0\hspace{1cm} \text{(by the BiHom-$\Omega$-Jacobi condition in Definition \ref{Def-BiHomLie})}\\
		=&0.\hspace{1cm}\text{(by Eq.~(\ref{BHO-Jacobi}))}
	\end{align*}
\end{proof}

\begin{coro}
	Let $(L,[\cdot ,\cdot ]_{\alpha,\,\beta})_{\alpha,\,\beta\in \Omega}$ be an $\Omega$-Lie algebra. If $(R_{\alpha})_{\alpha\in \Omega}$ is a Rota-Baxter family of weight $\lambda $ on $L$.
	Define a new multiplication on $L$ by
	\begin{equation*}
		[x,y]'_{\alpha,\,\beta}:=[R_{\alpha}(x),y]_{\alpha,\,\beta}+[x,R_{\beta}(y)]_{\alpha,\,\beta}+\lambda \lbrack x,y]_{\alpha,\,\beta},
	\end{equation*}
	for all $x,y\in L,\,\alpha,\,\beta\in \Omega,$ then $(L,[ \cdot , \cdot ]'_{\alpha,\,\beta})_{\alpha,\,\beta\in \Omega}$ is an $\Omega$-Lie algebra.
\end{coro}
\begin{proof}
	It follows from Theorem \ref{Proposition:BiHLie} by taking $p_{\alpha}=q_{\alpha}=\text{id}_{L}$ for $\alpha\in \Omega$.
\end{proof}

In~\cite{aguiar}, Aguiar has proved that a pre-Lie algebra induced by the Rota-Baxter family of weight zero on a Lie algebra, and the BiHom and Hom analogue of Aguiar's result were studied in~\cite{bihomprelie,makhloufrotabaxter}, now we generalize the classical results to BiHom-$\Omega$ version.

\begin{theorem}\label{lrprelie}
	Let $(L, \lbrace \cdot,\cdot \rbrace _{\alpha,\,\beta}, p_{\alpha}, q_{\alpha} )_{\alpha,\,\beta\in \Omega}$ be a BiHom-$\Omega$-Lie algebra and $(R_{\alpha})_{\alpha\in \Omega}$ be a
	Rota-Baxter family of weight $0$ on $L$ such that \[ R_{\alpha}\circ p_{\alpha} =p_{\alpha} \circ R_{\alpha} \; \text{and} \; R_{\alpha}\circ q_{\alpha} =q_{\alpha} \circ R_{\alpha}.\]
	Define the operation on $L$ by
	\[ x\blacktriangleright_{\alpha,\,\beta} y:=\lbrace R_{\alpha}(x), y\rbrace _{\alpha,\,\beta},\]
	for all $ x, y\in L,\,\alpha,\,\beta\in\Omega.$	Then $(L, \blacktriangleright_{\alpha,\,\beta} , p_{\alpha} , q_{\alpha} )_{\alpha,\,\beta\in \Omega}$ is a BiHom-$\Omega$-pre-Lie  algebra.
\end{theorem}

\begin{proof}
	For any $x,y,z\in L,\,\alpha,\,\beta,\,\gamma\in \Omega,$ we have
	\begin{align*}
		p_{\alpha\,\beta}(x\blacktriangleright_{\alpha,\,\beta}y)=&p_{\alpha\,\beta}(\lbrace R_{\alpha}(x),y\rbrace _{\alpha,\,\beta})\\
		=&\lbrace p_{\alpha}R_{\alpha}(x),p_{\beta}(y)\rbrace _{\alpha,\,\beta}\hspace{1cm}\text{(by $p_{\alpha\,\beta}$ satisfying Eq.~(\ref{Omega-Lie-morphism})})\\
		=&\lbrace R_{\alpha}p_{\alpha}(x),p_{\beta}(y)\rbrace _{\alpha,\,\beta}\hspace{1cm}\text{(by $R_{\alpha}\circ p_{\alpha}=p_{\alpha}\circ R_{\alpha}$)}\\
		=&p_{\alpha}(x)\blacktriangleright_{\alpha,\,\beta}p_{\beta}(y).
	\end{align*}
	Similarly, we get
	$q_{\alpha\,\beta}(x\blacktriangleright_{\alpha,\,\beta}y)=q_{\alpha}(x)\blacktriangleright_{\alpha,\,\beta}q_{\beta}(y). $	Next, we only need to prove Eq.~(\ref{BHO-prelie-alg}). On the one hand, we have
	\begin{align*}
		p&_{\alpha}q_{\alpha}(x)\blacktriangleright_{\alpha,\,\beta\,\gamma}(p_{\beta}(y)\blacktriangleright_{\beta,\,\gamma}z)-(q_{\alpha}(x)\blacktriangleright_{\alpha,\,\beta}p_{\beta}(y))\blacktriangleright_{\alpha\,\beta,\,\gamma}q_{\gamma}(z)\\
		=&p_{\alpha}q_{\alpha}(x)\blacktriangleright_{\alpha,\,\beta\,\gamma}\lbrace R_{\beta}p_{\beta}(y),z\rbrace _{\beta,\,\gamma}-\lbrace R_{\alpha}q_{\alpha}(x),p_{\beta}(y)\rbrace _{\alpha,\,\beta}\blacktriangleright_{\alpha\,\beta,\,\gamma}q_{\gamma}(z)\\
		=&\lbrace R_{\alpha}p_{\alpha}q_{\alpha}(x),\lbrace R_{\beta}p_{\beta}(y),z\rbrace _{\beta,\,\gamma}\rbrace _{\alpha,\,\beta\,\gamma}-\lbrace R_{\alpha\,\beta}(\lbrace R_{\alpha}q_{\alpha}(x),p_{\beta}(y)\rbrace _{\alpha,\,\beta}),q_{\gamma}(z)\rbrace _{\alpha\,\beta,\,\gamma}\\
		=&\lbrace R_{\alpha}p_{\alpha}q_{\alpha}(x),\lbrace R_{\beta}p_{\beta}(y),z\rbrace _{\beta,\,\gamma}\rbrace _{\alpha,\,\beta\,\gamma}+\lbrace R_{\alpha\,\beta}(\lbrace q_{\beta}(y),p_{\alpha}R_{\alpha}(x)\rbrace _{\beta,\,\alpha}),q_{\gamma}(z)\rbrace _{\beta\,\alpha,\,\gamma}\\
		&\hspace{1cm}\text{(by Eq.~(\ref{BHO-skew})})\\
		=&\lbrace q_{\alpha}p_{\alpha}R_{\alpha}(x),\lbrace p_{\beta}R_{\beta}(y),z\rbrace _{\beta,\,\gamma}\rbrace _{\alpha,\,\beta\,\gamma}+\lbrace R_{\alpha\,\beta}(\lbrace q_{\beta}(y),p_{\alpha}R_{\alpha}(x)\rbrace _{\beta,\,\alpha}),q_{\gamma}(z)\rbrace _{\beta\,\alpha,\,\gamma}.\\
		&\hspace{1cm}\text{(by $p_{\alpha},q_{\alpha},R_{\alpha}$ commuting with each other})
	\end{align*}
	
\noindent	On the other hand, we have
		\begin{align*}
		p&_{\beta}q_{\beta}(y)\blacktriangleright_{\beta,\,\alpha\,\gamma}(p_{\alpha}(x)\blacktriangleright_{\alpha,\,\gamma}z)-(q_{\beta}(y)\blacktriangleright_{\beta,\,\alpha}p_{\alpha}(x))\blacktriangleright_{\beta\,\alpha,\,\gamma}q_{\gamma}(z)\\
		=&p_{\beta}q_{\beta}(y)\blacktriangleright_{\beta,\,\alpha\,\gamma}\lbrace R_{\alpha}p_{\alpha}(x),z\rbrace _{\alpha,\,\gamma}-\lbrace R_{\beta}q_{\beta}(y),p_{\alpha}(x)\rbrace _{\beta,\,\alpha}\blacktriangleright_{\beta\,\alpha,\,\gamma}q_{\gamma}(z)\\
		=&\lbrace R_{\beta}p_{\beta}q_{\beta}(y),\lbrace R_{\alpha}p_{\alpha}(x),z\rbrace _{\alpha,\,\gamma}\rbrace _{\beta,\,\alpha\,\gamma}-\lbrace R_{\beta\,\alpha}(\lbrace R_{\beta}q_{\beta}(y),p_{\alpha}(x)\rbrace _{\beta,\,\alpha}),q_{\gamma}(z)\rbrace _{\beta\,\alpha,\,\gamma}\\
		=&\lbrace R_{\beta}p_{\beta}q_{\beta}(y),\lbrace R_{\alpha}p_{\alpha}(x),z\rbrace _{\alpha,\,\gamma}\rbrace _{\beta,\,\alpha\,\gamma}-\lbrace \lbrace R_{\beta}q_{\beta}(y),R_{\alpha}p_{\alpha}(x)\rbrace _{\beta,\,\alpha},q_{\gamma}(z)\rbrace _{\beta\,\alpha,\,\gamma}\\
		&+\lbrace R_{\beta\,\alpha}(\lbrace q_{\beta}(y),R_{\alpha}p_{\alpha}(x)\rbrace _{\beta,\,\alpha}),q_{\gamma}(z)\rbrace _{\beta\,\alpha,\,\gamma}\hspace{1cm}\text{(by Eq.~(\ref{RB-Eq}))}\\
		=&\lbrace q_{\beta}p_{\beta}R_{\beta}(y),\lbrace p_{\alpha}R_{\alpha}(x),z\rbrace _{\alpha,\,\gamma}\rbrace _{\beta,\,\alpha\,\gamma}-\lbrace \lbrace q_{\beta}R_{\beta}(y),p_{\alpha}R_{\alpha}(x)\rbrace _{\beta,\,\alpha},q_{\gamma}(z)\rbrace _{\beta\,\alpha,\,\gamma}\\
		&+\lbrace R_{\alpha\,\beta}(\lbrace q_{\beta}(y),p_{\alpha}R_{\alpha}(x)\rbrace _{\beta,\,\alpha}),q_{\gamma}(z)\rbrace _{\beta\,\alpha,\,\gamma}\hspace{1cm}\text{(by $p_{\alpha},q_{\alpha},R_{\alpha}$ commuting with each other)}\\
		=&\lbrace q_{\beta}p_{\beta}R_{\beta}(y),\lbrace p_{\alpha}R_{\alpha}(x),z\rbrace _{\alpha,\,\gamma}\rbrace _{\beta,\,\alpha\,\gamma}+\lbrace \lbrace q_{\alpha}R_{\alpha}(x),p_{\beta}R_{\beta}(y)\rbrace _{\alpha,\,\beta},q_{\gamma}(z)\rbrace _{\alpha\,\beta,\,\gamma}\\
		&+\lbrace R_{\alpha\,\beta}(\lbrace q_{\beta}(y),p_{\alpha}R_{\alpha}(x)\rbrace _{\beta,\,\alpha}),q_{\gamma}(z)\rbrace _{\beta\,\alpha,\,\gamma}\hspace{1cm}\text{(by Eq.~(\ref{BHO-skew})})\\
		=&\lbrace q_{\beta}p_{\beta}R_{\beta}(y),\lbrace p_{\alpha}R_{\alpha}(x),z\rbrace _{\alpha,\,\gamma}\rbrace _{\beta,\,\alpha\,\gamma}+\lbrace q_{\alpha\,\beta}q_{\alpha\,\beta}^{-1}(\lbrace q_{\alpha}R_{\alpha}(x),p_{\beta}R_{\beta}(y)\rbrace _{\alpha,\,\beta}),p_{\gamma}q_{\gamma}p_{\gamma}^{-1}(z)\rbrace _{\alpha\,\beta,\,\gamma}\\
		&+\lbrace R_{\alpha\,\beta}(\lbrace q_{\beta}(y),p_{\alpha}R_{\alpha}(x)\rbrace _{\beta,\,\alpha}),q_{\gamma}(z)\rbrace _{\beta\,\alpha,\,\gamma}\\
		=&\lbrace q_{\beta}p_{\beta}R_{\beta}(y),\lbrace p_{\alpha}R_{\alpha}(x),z\rbrace _{\alpha,\,\gamma}\rbrace _{\beta,\,\alpha\,\gamma}-\lbrace q_{\gamma}^{2}p_{\gamma}^{-1}(z),p_{\alpha\,\beta}q_{\alpha\,\beta}^{-1}(\lbrace q_{\alpha}R_{\alpha}(x),p_{\beta}R_{\beta}(y)\rbrace _{\alpha,\,\beta})\rbrace _{\gamma,\,\alpha\,\beta}\\
		&+\lbrace R_{\alpha\,\beta}(\lbrace q_{\beta}(y),p_{\alpha}R_{\alpha}(x)\rbrace _{\beta,\,\alpha}),q_{\gamma}(z)\rbrace _{\beta\,\alpha,\,\gamma}\hspace{1cm}\text{(by Eq.~(\ref{BHO-skew})})\\
		=&\lbrace q_{\beta}p_{\beta}R_{\beta}(y),\lbrace p_{\alpha}R_{\alpha}(x),z\rbrace _{\alpha,\,\gamma}\rbrace _{\beta,\,\alpha\,\gamma}-\lbrace q_{\gamma}^{2}p_{\gamma}^{-1}(z),\lbrace p_{\alpha}R_{\alpha}(x),p_{\beta}^{2}q_{\beta}^{-1}R_{\beta}(y)\rbrace _{\alpha,\,\beta}\rbrace _{\gamma,\,\alpha\,\beta}\\
		&+\lbrace R_{\alpha\,\beta}(\lbrace q_{\beta}(y),p_{\alpha}R_{\alpha}(x)\rbrace _{\beta,\,\alpha}),q_{\gamma}(z)\rbrace _{\beta\,\alpha,\,\gamma}\hspace{1cm}\text{(by $p_{\alpha\,\beta},q_{\alpha\,\beta}^{-1}$ satisfying Eq.~(\ref{Omega-Lie-morphism})})\\
		=&\lbrace q_{\beta}p_{\beta}R_{\beta}(y),\lbrace p_{\alpha}R_{\alpha}(x),z\rbrace _{\alpha,\,\gamma}\rbrace _{\beta,\,\alpha\,\gamma}+\lbrace q_{\alpha}^{2}q_{\alpha}^{-1}p_{\alpha}R_{\alpha}(x),\lbrace q_{\beta}p_{\beta}q_{\beta}^{-1}R_{\beta}(y),p_{\gamma}p_{\gamma}^{-1}(z)\rbrace _{\beta,\,\gamma}\rbrace _{\alpha,\,\beta\,\gamma}\\
		&+\lbrace q_{\beta}^{2}p_{\beta}q_{\beta}^{-1}R_{\beta}(y),\lbrace q_{\gamma}p_{\gamma}^{-1}(z),p_{\alpha}q_{\alpha}^{-1}p_{\alpha}R_{\alpha}(x)\rbrace _{\gamma,\,\alpha}\rbrace _{\beta,\,\gamma\,\alpha}+\lbrace R_{\alpha\,\beta}(\lbrace q_{\beta}(y),p_{\alpha}R_{\alpha}(x)\rbrace _{\beta,\,\alpha}),q_{\gamma}(z)\rbrace _{\beta\,\alpha,\,\gamma}\\
		&\hspace{7.4cm}\text{(by Eq.~(\ref{BHO-Jacobi})})\\
		=&\lbrace q_{\beta}p_{\beta}R_{\beta}(y),\lbrace p_{\alpha}R_{\alpha}(x),z\rbrace _{\alpha,\,\gamma}\rbrace _{\beta,\,\alpha\,\gamma}+\lbrace q_{\alpha}p_{\alpha}R_{\alpha}(x),\lbrace p_{\beta}R_{\beta}(y),z\rbrace _{\beta,\,\gamma}\rbrace _{\alpha,\,\beta\,\gamma}\\
		&+\lbrace q_{\beta}p_{\beta}R_{\beta}(y),\lbrace q_{\gamma}p_{\gamma}^{-1}(z),p_{\alpha}^{2}q_{\alpha}^{-1}R_{\alpha}(x)\rbrace _{\gamma,\,\alpha}\rbrace _{\beta,\,\gamma\,\alpha}+\lbrace R_{\alpha\,\beta}(\lbrace q_{\beta}(y),p_{\alpha}R_{\alpha}(x)\rbrace _{\beta,\,\alpha}),q_{\gamma}(z)\rbrace _{\beta\,\alpha,\,\gamma}\\
		=&\lbrace q_{\beta}p_{\beta}R_{\beta}(y),\lbrace p_{\alpha}R_{\alpha}(x),z\rbrace _{\alpha,\,\gamma}\rbrace _{\beta,\,\alpha\,\gamma}+\lbrace q_{\alpha}p_{\alpha}R_{\alpha}(x),\lbrace p_{\beta}R_{\beta}(y),z\rbrace _{\beta,\,\gamma}\rbrace _{\alpha,\,\beta\,\gamma}\\
		&-\lbrace q_{\beta}p_{\beta}R_{\beta}(y),\lbrace q_{\alpha}p_{\alpha}q_{\alpha}^{-1}R_{\alpha}(x),p_{\gamma}p_{\gamma}^{-1}(z)\rbrace _{\alpha,\,\gamma}\rbrace _{\beta,\,\alpha\,\gamma}+\lbrace R_{\alpha\,\beta}(\lbrace q_{\beta}(y),p_{\alpha}R_{\alpha}(x)\rbrace _{\beta,\,\alpha}),q_{\gamma}(z)\rbrace _{\beta\,\alpha,\,\gamma}\\
		&\hspace{7.4cm}\text{(by Eq.~(\ref{BHO-skew}))}\\
		=&\lbrace q_{\beta}p_{\beta}R_{\beta}(y),\lbrace p_{\alpha}R_{\alpha}(x),z\rbrace _{\alpha,\,\gamma}\rbrace _{\beta,\,\alpha\,\gamma}+\lbrace q_{\alpha}p_{\alpha}R_{\alpha}(x),\lbrace p_{\beta}R_{\beta}(y),z\rbrace _{\beta,\,\gamma}\rbrace _{\alpha,\,\beta\,\gamma}\\
		&-\lbrace q_{\beta}p_{\beta}R_{\beta}(y),\lbrace p_{\alpha}R_{\alpha}(x),z\rbrace _{\alpha,\,\gamma}\rbrace _{\beta,\,\alpha\,\gamma}+\lbrace R_{\alpha\,\beta}(\lbrace q_{\beta}(y),p_{\alpha}R_{\alpha}(x)\rbrace _{\beta,\,\alpha}),q_{\gamma}(z)\rbrace _{\beta\,\alpha,\,\gamma}\\
		=&\lbrace q_{\alpha}p_{\alpha}R_{\alpha}(x),\lbrace p_{\beta}R_{\beta}(y),z\rbrace _{\beta,\,\gamma}\rbrace _{\alpha,\,\beta\,\gamma}+\lbrace R_{\alpha\,\beta}(\lbrace q_{\beta}(y),p_{\alpha}R_{\alpha}(x)\rbrace _{\beta,\,\alpha}),q_{\gamma}(z)\rbrace _{\beta\,\alpha,\,\gamma}.
	\end{align*}
	By comparing the items of both sides, we get
	\begin{align*}
		p&_{\alpha}q_{\alpha} (x)\blacktriangleright_{\alpha,\,\beta\gamma} (p_{\beta} (y)\blacktriangleright_{\beta,\,\gamma} z)-(q_{\alpha} (x)\blacktriangleright_{\alpha,\,\beta} p_{\beta} (y))\blacktriangleright_{\alpha\beta,\,\gamma} q_{\gamma} (z)\\
		=&p_{\beta} q_{\beta} (y)\blacktriangleright_{\beta,\,\alpha\gamma} (p_{\alpha}(x)\blacktriangleright_{\alpha,\gamma} z)-(q_{\beta}(y)\blacktriangleright_{\beta,\,\alpha} p_{\alpha}(x))\blacktriangleright_{\beta\alpha,\gamma} q_{\gamma}(z).
	\end{align*}
	This completes the proof.
\end{proof}

%%%%%%%%%%%%%%%%%%%%%%%%%%%%%%%%%%%%%%%%%%%%
%%%%%%%%%%%%%%%%%%%%%%%%%%%%%%
\section{BiHom-$\Omega$-PostLie algebras and BiHom-$\Omega$-pre-Possion algebras}\label{sec5}
%%%%%%%%%%%%%%%%%%%%%%%%%%%%%
%\setcounter{equation}{0}
%%%%%%%%%%%%%%%%%%%%%%%%%%%%
In this section, we continue to assume that $\Omega$ is a commutative semigroup.
\subsection{BiHom-$\Omega$-PostLie algebras}
In this subsection, we mainly study the relationship between BiHom-$\Omega$-PostLie algebras and BiHom-$\Omega$-Lie algebras. First, we generlize the concept of PostLie algebras~\cite{bai,vallette} to $\Omega$ version.
\begin{defn}\label{Omega-PostLie}
	An $\mathbf{\Omega}$-$\mathbf{PostLie \, algebra}$ $(L, [\cdot , \cdot ]_{\alpha,\,\beta}, \rhd_{\alpha,\,\beta} )_{\alpha,\,\beta\in\Omega}$ is a vector space $L$ equipped with two families of bilinear operations $[\cdot , \cdot ]_{\alpha,\,\beta}$,
	$\rhd_{\alpha,\,\beta} :L\times L\rightarrow L$, such that $(L, [\cdot , \cdot ]_{\alpha,\,\beta})_{\alpha,\,\beta\in\Omega}$ is an $\Omega$-Lie  algebra and
	\begin{align*}
		[x,y]_{\alpha,\,\beta}\rhd_{\alpha\beta,\,\gamma}z=&x\rhd_{\alpha,\,\beta\gamma}(y\rhd_{\beta,\,\gamma}z)-(x\rhd_{\alpha,\,\beta}y)\rhd_{\alpha\beta,\,\gamma}z-y\rhd_{\beta,\,\alpha\gamma}(x\rhd_{\alpha,\gamma}z)+(y\rhd_{\beta,\,\alpha}x)\rhd_{\beta\alpha,\gamma}z,\\
		x\rhd_{\alpha,\,\beta\gamma}[y,z]_{\beta,\,\gamma}=&[x\rhd_{\alpha,\,\beta}y,z]_{\alpha\beta,\,\gamma}+[y,x\rhd_{\alpha,\gamma}z]_{\beta,\,\alpha\gamma},
	\end{align*}
 for all $x, y, z\in L,\,\alpha,\,\beta,\,\gamma\in \Omega$.
\end{defn}

\begin{defn}
	Let $ (L,[\cdot,\cdot]_{\alpha,\,\beta},\rhd_{\alpha,\,\beta})_{\alpha,\,\beta\in \Omega} $ and $ (L',[\cdot,\cdot]'_{\alpha,\,\beta},\rhd'_{\alpha,\,\beta})_{\alpha,\,\beta\in \Omega} $ be two $\Omega$-PostLie algebras. A family of linear maps $ (f_{\alpha})_{\alpha\in \Omega}: L\rightarrow L' $ is called an $\mathbf{\Omega}$-$\mathbf{PostLie \, algebra \, morphism}$ if \[f_{\alpha\,\beta}[x,y]_{\alpha,\,\beta}=[f_{\alpha}(x),f_{\beta}(y)]'_{\alpha,\,\beta},\]
	\[f_{\alpha\,\beta}(x\rhd_{\alpha,\,\beta}y)=f_{\alpha}(x)\rhd'_{\alpha,\,\beta}f_{\beta}(y),\]
	for all $ x,y\in L,\,\alpha,\,\beta\in \Omega. $
\end{defn}

\begin{remark}\label{Post-pre(-Lie)}
	Let $ (L,[\cdot,\cdot]_{\alpha,\,\beta},\rhd_{\alpha,\,\beta})_{\alpha,\,\beta\in \Omega} $ be an $\Omega$-PostLie algebra.
		\begin{enumerate}
		\item If $ [x,y]_{\alpha,\,\beta}=0,$ for all $x,y\in L,\;\alpha,\,\beta\in\Omega $, then we get \[x\rhd_{\alpha,\,\beta\,\gamma}(y\rhd_{\beta,\,\gamma}z)-(x\rhd_{\alpha,\,\beta}y)\rhd_{\alpha\,\beta,\,\gamma}z=y\rhd_{\beta,\,\alpha\,\gamma}(x\rhd_{\alpha,\,\gamma}z)-(y\rhd_{\beta,\,\alpha}x)\rhd_{\beta\,\alpha,\,\gamma}z,\] for all $ x,y,z\in L,\,\alpha,\,\beta,\,\gamma \in \Omega,$ that is $ (L,\rhd_{\alpha,\,\beta})_{\alpha,\,\beta\in \Omega} $ is an $\Omega$-pre-Lie algebra.
		\item If $ x\rhd_{\alpha,\,\beta}y=0, $ for all $ x,y\in L,\;\alpha,\,\beta\in \Omega $, then $ (L,[\cdot,\cdot]_{\alpha,\,\beta})_{\alpha,\,\beta\in \Omega} $ is an $\Omega$-Lie algebra.
	\end{enumerate}
\end{remark}

The concept of BiHom-PostLie algebras was given in \cite{BiHom-PostLie}. Now we introduce a more general version of Definition~\ref{Omega-PostLie}.
\begin{defn} \label{BiHom-Ome-PostLie}
	A $\mathbf{BiHom}$-$\mathbf{\Omega}$-$\mathbf{PostLie \, algebra}$ $(L, \{\cdot , \cdot \}_{\alpha,\,\beta}, \blacktriangleright_{\alpha,\,\beta}, p_{\alpha} , q_{\alpha} )_{\alpha,\,\beta\in \Omega}$ is a vector space $L$ equipped with two families of bilinear maps $\{\cdot , \cdot \}_{\alpha,\,\beta} ,\,\blacktriangleright_{\alpha,\,\beta}: L\times L\rightarrow L$, and two commuting families of linear maps $(p_{\alpha})_{\alpha\in \Omega} , (q_{\alpha})_{\alpha\in \Omega} :L\rightarrow L$ such that $(L, \{\cdot , \cdot \}_{\alpha,\,\beta}, p_{\alpha} , q_{\alpha} )_{\alpha,\,\beta\in \Omega}$ is a BiHom-$\Omega$-Lie algebra and
	\[p _{\alpha\,\beta}(x\blacktriangleright_{\alpha ,\,\beta}y)=p_{\alpha} (x)\blacktriangleright_{\alpha ,\, \beta}p_{\beta} (y),\;\; q _{\alpha\,\beta}(x\blacktriangleright_{\alpha ,\, \beta}y)=q_{\alpha} (x)\blacktriangleright_{\alpha ,\, \beta}q_{\beta}(y) ,\]
		\begin{align*}
		\lbrace q_{\alpha}(x),p_{\beta}(y)\rbrace _{\alpha,\,\beta}\blacktriangleright_{\alpha\beta,\,\gamma}q_{\gamma}(z)
		=&p_{\alpha}q_{\alpha}(x)\blacktriangleright_{\alpha,\,\beta\gamma}(p_{\beta}(y)\blacktriangleright_{\beta,\,\gamma}z)-(q_{\alpha}(x)\blacktriangleright_{\alpha ,\, \beta}p_{\beta}(y))\blacktriangleright_{\alpha\beta,\,\gamma}q_{\gamma}(z)\\
		&-p_{\beta}q_{\beta}(y)\blacktriangleright_{\beta,\,\alpha\gamma}(p_{\alpha}(x)\blacktriangleright_{\alpha,\gamma}z)+(q_{\beta}(y)\blacktriangleright_{\beta,\,\alpha}p_{\alpha}(x))\blacktriangleright_{\beta\alpha,\gamma}q_{\gamma}(z), %\label{PostLie1}
	\end{align*}
\begin{align*}
		\quad p_{\alpha}q_{\alpha} (x)\blacktriangleright_{\alpha,\,\beta\gamma}\{y, z\}_{\beta,\,\gamma}=&\{q_{\alpha}(x)\blacktriangleright_{\alpha ,\, \beta} y, q_{\gamma}(z)\}_{\alpha\beta,\,\gamma}+\{q_{\beta} (y), p_{\alpha} (x)\blacktriangleright_{\alpha,\gamma} z\}_{\beta,\,\alpha\gamma},%\label{PostLie2}
	\end{align*}
for all $x, y, z\in L,\,\alpha,\,\beta,\,\gamma\in \Omega$. The maps $p_{\alpha},q_{\alpha}$ (in this order) are called the structure maps of $L$.
\end{defn}
%Similar to Remark~\ref{Post-pre(-Lie)}, we get the following results.

\begin{remark}\label{PostLie-Lie}
\begin{enumerate}
\item \label{item:prelie} If $ \{x,y\}_{\alpha,\,\beta}=0,$ for all $ x,y\in L,\;\alpha,\,\beta\in \Omega $, then $ (L,\blacktriangleright_{\alpha,\,\beta},p_{\alpha},q_{\alpha})_{\alpha,\,\beta\in \Omega} $ is a BiHom-$\Omega$-pre-Lie algebra.

\item \label{item:lie} If $ x\blacktriangleright_{\alpha,\,\beta}y=0, $ for all $ x,y\in L,\;\alpha,\,\beta\in \Omega $, then $ (L,\{\cdot,\cdot\}_{\alpha,\,\beta},p_{\alpha},q_{\alpha})_{\alpha,\,\beta\in \Omega} $ is a BiHom-$\Omega$-Lie algebra.
    \end{enumerate}
\end{remark}

Now we introduce the Yau twisting procedure for BiHom-$\Omega$-PostLie algebras and the proof is similar to Proposition \ref{Yautwist-preLie} and \ref{Ome-Lie-Yautwist}.

\begin{prop}
	Let $(L, [\cdot , \cdot ]_{\alpha,\,\beta}, \rhd_{\alpha,\,\beta} )_{\alpha,\,\beta\in \Omega}$ be an $\Omega$-PostLie algebra and let $p_{\alpha} , q_{\alpha} :L\rightarrow L$ be
	two commuting $\Omega$-PostLie algebra morphisms. Define two operations on $L$ by
	 \[\{x, y\}_{\alpha,\,\beta}:=[p_{\alpha}(x), q_{\beta} (y)]_{\alpha,\,\beta}\;\text{and} \; x\blacktriangleright_{\alpha ,\, \beta}y:=p_{\alpha}(x)\rhd_{\alpha,\,\beta} q_{\beta}(y),\]
	 for all $x, y\in L,\,\alpha,\,\beta\in\Omega$. Then $(L, \{\cdot , \cdot \}_{\alpha,\,\beta}, \blacktriangleright_{\alpha ,\,\beta}, p_{\alpha} , q_{\alpha})_{\alpha,\,\beta\in \Omega}$ is a BiHom-$\Omega$-PostLie algebra, called the Yau twist of  $(L, [\cdot , \cdot ]_{\alpha,\,\beta}, \rhd_{\alpha,\,\beta} )_{\alpha,\,\beta\in \Omega}$.
\end{prop}

Now we give a way to construct a BiHom-$\Omega$-Lie algebra by defining a new operation on the BiHom-$\Omega$-PostLie algebra.

\begin{prop}\label{interm}
	Let $(L, \{\cdot , \cdot \}_{\alpha,\,\beta}, \blacktriangleright_{\alpha ,\, \beta}, p_{\alpha}, q_{\alpha})_{\alpha,\,\beta\in \Omega}$ be a BiHom-$\Omega$-PostLie algebra. If $p_{\alpha},\,q_{\alpha} $ are
	bijective. Define a new multiplication on $L$  by
	\[<x, y>_{\alpha,\,\beta}:=x\blacktriangleright_{\alpha , \,\beta}y-(p_{\beta} ^{-1}q_{\beta}(y))\blacktriangleright_{\beta,\,\alpha}(p_{\alpha} q_{\alpha}^{-1}(x))+\{x, y\}_{\alpha,\,\beta},\]
	for all $x, y\in L,\,\alpha,\,\beta\in \Omega,$	then $(L, <\cdot , \cdot >_{\alpha,\,\beta}, p_{\alpha}, q_{\alpha} )_{\alpha,\,\beta\in \Omega}$ is a BiHom-$\Omega$-Lie algebra.
\end{prop}
\begin{proof}
	For all $ x, y, z \in L,\,\alpha,\,\beta,\,\gamma\in \Omega,$ owing to the commutativity, we get $ p_{\alpha}\circ q_{\alpha}=q_{\alpha}\circ p_{\alpha}$ and we have \[p_{\alpha\,\beta}<x,y>_{\alpha,\,\beta}=<p_{\alpha}(x),p_{\beta}(y)>_{\alpha,\,\beta},\quad q_{\alpha\,\beta}<x,y>_{\alpha,\,\beta}=<q_{\alpha}(x),q_{\beta}(y)>_{\alpha,\,\beta}.\]
	First, we prove the BiHom-$\Omega$-skew-symmetry, we have
	\begin{align*}
		<q_{\alpha}(x),p_{\beta}(y)>_{\alpha,\,\beta}=&q_{\alpha}(x)\blacktriangleright_{\alpha ,\,\beta}p_{\beta}(y)-p_{\beta}^{-1}q_{\beta}(p_{\beta}(y))\blacktriangleright_{\beta,\,\alpha}p_{\alpha}q_{\alpha}^{-1}q_{\alpha}(x)+\lbrace q_{\alpha}(x),p_{\beta}(y)\rbrace_{\alpha,\,\beta}\\
		=&q_{\alpha}(x)\blacktriangleright_{\alpha,\,\beta}p_{\beta}(y)-q_{\beta}(y)\blacktriangleright_{\beta,\,\alpha}p_{\alpha}(x)+\lbrace q_{\alpha}(x),p_{\beta}(y)\rbrace_{\alpha,\,\beta}\\
		=&-(q_{\beta}(y)\blacktriangleright_{\beta,\,\alpha}p_{\alpha}(x)-q_{\alpha}(x)\blacktriangleright_{\alpha,\,\beta}p_{\beta}(y)+\lbrace q_{\beta}(y),p_{\alpha}(x)\rbrace_{\beta,\,\alpha})\\
		&\hspace{1cm} \text{(by Eq.~(\ref{BHO-skew})}\\
		=&-(q_{\beta}(y)\blacktriangleright_{\beta,\,\alpha}p_{\alpha}(x)-p_{\alpha}^{-1}q_{\alpha}p_{\alpha}(x)\blacktriangleright_{\alpha,\,\beta}p_{\beta}q_{\beta}^{-1}q_{\beta}(y)+\lbrace q_{\beta}(y),p_{\alpha}(x)\rbrace_{\beta,\,\alpha})\\
		=&-<q_{\beta}(y),p_{\alpha}(x)>_{\beta,\,\alpha}.
	\end{align*}
	Next, we prove the BiHom-$\Omega$-Jacobi condition, we have
	\begin{align*}
		<&q_{\alpha}^{2}(x),<q_{\beta}(y),p_{\gamma}(z)>_{\beta,\,\gamma}>_{\alpha,\,\beta\,\gamma}+<q_{\beta}^{2}(y),<q_{\gamma}(z),p_{\alpha}(x)>_{\gamma,\,\alpha}>_{\beta,\,\gamma\,\alpha}\\
		&+<q_{\gamma}^{2}(z),<q_{\alpha}(x),p_{\beta}(y)>_{\alpha,\,\beta}>_{\gamma,\,\alpha\,\beta}\\
		=&<q_{\alpha}^{2}(x),q_{\beta}(y)\blacktriangleright_{\beta,\,\gamma}p_{\gamma}(z)>_{\alpha,\,\beta\,\gamma}-<q_{\alpha}^{2}(x),q_{\gamma}(z)\blacktriangleright_{\gamma,\,\beta}p_{\beta}(y)>_{\alpha,\,\beta\,\gamma}\\
		&+<q_{\alpha}^{2}(x),\lbrace q_{\beta}(y),p_{\gamma}(z)\rbrace_{\beta,\,\gamma}>_{\alpha,\,\beta\,\gamma}+<q_{\beta}^{2}(y),q_{\gamma}(z)\blacktriangleright_{\gamma,\,\alpha}p_{\alpha}(x)>_{\beta,\,\gamma\,\alpha}\\
		&-<q_{\beta}^{2}(y),q_{\alpha}(x)\blacktriangleright_{\alpha,\,\gamma}p_{\gamma}(z)>_{\beta,\,\gamma\,\alpha}+<q_{\beta}^{2}(y),\lbrace q_{\gamma}(z),p_{\alpha}(x)\rbrace_{\gamma,\,\alpha}>_{\beta,\,\gamma\,\alpha}\\
		&+<q_{\gamma}^{2}(z),q_{\alpha}(x)\blacktriangleright_{\alpha,\,\beta}p_{\beta}(y)>_{\gamma,\,\alpha\,\beta}-<q_{\gamma}^{2}(z),q_{\beta}(y)\blacktriangleright_{\beta,\,\alpha}p_{\alpha}(x)>_{\gamma,\,\alpha\,\beta}\\
		&+<q_{\gamma}^{2}(z),\lbrace q_{\alpha}(x),p_{\beta}(y)\rbrace_{\alpha,\,\beta}>_{\gamma,\,\alpha\,\beta}\\
		=&q_{\alpha}^{2}(x)\blacktriangleright_{\alpha,\,\beta\gamma}(q_{\beta}(y)\blacktriangleright_{\beta,\,\gamma}p_{\gamma}(z))+\lbrace q_{\alpha}^{2}(x),q_{\beta}(y)\blacktriangleright_{\beta,\,\gamma}p_{\gamma}(z)\rbrace_{\alpha,\,\beta\,\gamma}-q_{\alpha}^{2}(x)\blacktriangleright_{\alpha,\,\beta\gamma}(q_{\gamma}(z)\blacktriangleright_{\gamma,\,\beta}p_{\beta}(y))\\
		&-\lbrace q_{\alpha}^{2}(x),q_{\gamma}(z)\blacktriangleright_{\gamma,\,\beta}p_{\gamma}(z)\rbrace_{\alpha,\,\beta\,\gamma}+q_{\alpha}^{2}(x)\blacktriangleright_{\alpha,\,\beta\gamma}\lbrace q_{\beta}(y),p_{\gamma}(z)\rbrace_{\beta,\gamma}+\lbrace q_{\alpha}^{2}(x),\lbrace q_{\beta}(y),p_{\gamma}(z)\rbrace_{\beta,\,\gamma}\rbrace_{\alpha,\,\beta\gamma}\\
		&+q_{\beta}^{2}(y)\blacktriangleright_{\beta,\,\gamma\,\alpha}(q_{\gamma}(z)\blacktriangleright_{\gamma,\,\alpha}p_{\alpha}(x))+\lbrace q_{\beta}^{2}(y),q_{\gamma}(z)\blacktriangleright_{\gamma,\,\alpha}p_{\alpha}(x)\rbrace_{\beta,\,\gamma\,\alpha}-q_{\beta}^{2}(y)\blacktriangleright_{\beta,\,\gamma\,\alpha}(q_{\alpha}(x)\blacktriangleright_{\alpha,\,\gamma}p_{\gamma}(z))\\
		&-\lbrace q_{\beta}^{2}(y),q_{\alpha}(x)\blacktriangleright_{\alpha,\,\gamma}p_{\gamma}(z)\rbrace_{\beta,\,\gamma\,\alpha}+q_{\beta}^{2}(y)\blacktriangleright_{\beta,\,\gamma\,\alpha}\lbrace q_{\gamma}(z),p_{\alpha}(x)\rbrace_{\gamma,\,\alpha}+\lbrace q_{\beta}^{2}(y),\lbrace q_{\gamma}(z),p_{\alpha}(x)\rbrace_{\gamma,\,\alpha}\rbrace_{\beta,\,\gamma\,\alpha}\\
		&+q_{\gamma}^{2}(z)\blacktriangleright_{\gamma,\,\alpha\,\beta}(q_{\alpha}(x)\blacktriangleright_{\alpha ,\,\beta}p_{\beta}(y))+\lbrace q_{\gamma}^{2}(z),q_{\alpha}(x)\blacktriangleright_{\alpha,\,\beta}p_{\beta}(y)\rbrace_{\gamma,\,\alpha\,\beta}-q_{\gamma}^{2}(z)\blacktriangleright_{\gamma,\,\alpha\,\beta}(q_{\beta}(y)\blacktriangleright_{\beta,\,\alpha}p_{\alpha}(x))\\
		&-\lbrace q_{\gamma}^{2}(z),q_{\beta}(y)\blacktriangleright_{\beta,\,\alpha}p_{\alpha}(x)\rbrace_{\gamma,\,\beta\,\alpha}+q_{\gamma}^{2}(z)\blacktriangleright_{\gamma,\,\alpha\,\beta}\lbrace q_{\alpha}(x),p_{\beta}(y)\rbrace_{\alpha,\,\beta}+\lbrace q_{\gamma}^{2}(z),\lbrace q_{\alpha}(x),p_{\beta}(y)\rbrace_{\alpha,\,\beta}\rbrace_{\gamma,\,\alpha\,\beta}\\
		&-p_{\beta\,\gamma}^{-1}q_{\beta\,\gamma}(q_{\beta}(y)\blacktriangleright_{\beta,\,\gamma}p_{\gamma}(z))\blacktriangleright_{\beta\,\gamma,\,\alpha}p_{\alpha}q_{\alpha}^{-1}q_{\alpha}^{2}(x)+p_{\gamma\,\beta}^{-1}q_{\gamma\,\beta}(q_{\gamma}(z)\blacktriangleright_{\gamma,\,\beta}p_{\beta}(y))\blacktriangleright_{\gamma\,\beta,\,\alpha}p_{\alpha}q_{\alpha}(x)\\
		&-p_{\beta\,\gamma}^{-1}q_{\beta\,\gamma}\lbrace q_{\beta}(y),p_{\gamma}(z)\rbrace_{\beta,\,\gamma}\blacktriangleright_{\beta\,\gamma,\,\alpha}p_{\alpha}q_{\alpha}(x)-p_{\gamma\,\alpha}^{-1}q_{\gamma\,\alpha}(q_{\gamma}(z)\blacktriangleright_{\gamma,\,\alpha}p_{\alpha}(x))\blacktriangleright_{\gamma\,\alpha,\,\beta}p_{\beta}q_{\beta}(y)\\
		&+p_{\alpha\,\gamma}^{-1}q_{\alpha\,\gamma}(q_{\alpha}(x)\blacktriangleright_{\alpha,\,\gamma}p_{\gamma}(z))\blacktriangleright_{\alpha\,\gamma,\,\beta}p_{\beta}q_{\beta}(y)-p_{\gamma\,\alpha}^{-1}q_{\gamma\,\alpha}\lbrace q_{\gamma}(z),p_{\alpha}(x)\rbrace_{\gamma,\,\alpha}\blacktriangleright_{\gamma\,\alpha,\,\beta}p_{\beta}q_{\beta}(y)\\
		&-p_{\alpha\,\beta}^{-1}q_{\alpha\,\beta}(q_{\alpha}(x)\blacktriangleright_{\alpha,\,\beta}p_{\beta}(y))\blacktriangleright_{\alpha\,\beta,\,\gamma}p_{\gamma}q_{\gamma}(z)+p_{\beta\,\alpha}^{-1}q_{\beta\,\alpha}(q_{\beta}(y)\blacktriangleright_{\beta,\,\alpha}p_{\alpha}(x))\blacktriangleright_{\beta\,\alpha,\gamma}p_{\gamma}q_{\gamma}(z)\\
		&-p_{\alpha\,\beta}^{-1}q_{\alpha\,\beta}\lbrace q_{\alpha}(x),p_{\beta}(y)\rbrace_{\alpha,\,\beta}\blacktriangleright_{\alpha\,\beta,\,\gamma}p_{\gamma}q_{\gamma}(z)\\
		=&q_{\alpha}^{2}(x)\blacktriangleright_{\alpha,\,\beta\gamma}(q_{\beta}(y)\blacktriangleright_{\beta,\,\gamma}p_{\gamma}(z))-(p_{\beta}^{-1}q_{\beta}^{2}(y)\blacktriangleright_{\beta,\,\gamma}q_{\gamma}(z))\blacktriangleright_{\beta\,\gamma,\,\alpha}p_{\alpha}q_{\alpha}(x)+\lbrace q_{\alpha}^{2}(x),q_{\beta}(y)\blacktriangleright_{\beta,\,\gamma}p_{\gamma}(z)\rbrace_{\alpha,\,\beta\,\gamma}\\
		&-q_{\alpha}^{2}(x)\blacktriangleright_{\alpha,\,\beta\gamma}(q_{\gamma}(z)\blacktriangleright_{\gamma,\,\beta}p_{\beta}(y))+(p_{\gamma}^{-1}q_{\gamma}^{2}(z)\blacktriangleright_{\gamma,\,\beta}q_{\beta}(y))\blacktriangleright_{\gamma\,\beta,\,\alpha}p_{\alpha}q_{\alpha}(x)-\lbrace q_{\alpha}^{2}(x),q_{\gamma}(z)\blacktriangleright_{\gamma,\,\beta}p_{\beta}(y)\rbrace_{\alpha,\,\beta\,\gamma}\\
		&+q_{\alpha}^{2}(x)\blacktriangleright_{\alpha,\,\beta\gamma}\lbrace q_{\beta}(y),p_{\gamma}(z)\rbrace_{\beta,\,\gamma}-\lbrace p_{\beta}^{-1}q_{\beta}^{2}(y),q_{\gamma}(z)\rbrace_{\beta,\,\gamma}\blacktriangleright_{\beta\,\gamma,\,\alpha}p_{\alpha}q_{\alpha}(x)+\lbrace q_{\alpha}^{2}(x),\lbrace q_{\beta}(y),p_{\gamma}(z)\rbrace_{\beta,\,\gamma}\rbrace_{\alpha,\,\beta\,\gamma}\\
		&+q_{\beta}^{2}(y)\blacktriangleright_{\beta,\,\gamma\,\alpha}(q_{\gamma}(z)\blacktriangleright_{\gamma,\,\alpha}p_{\alpha}(x))-(p_{\gamma}^{-1}q_{\gamma}^{2}(z)\blacktriangleright_{\gamma,\,\alpha}q_{\alpha}(x))\blacktriangleright_{\gamma\,\alpha,\,\beta}p_{\beta}q_{\beta}(y)+\lbrace q_{\beta}^{2}(y),q_{\gamma}(z)\blacktriangleright_{\gamma,\,\alpha}p_{\alpha}(x)\rbrace_{\beta,\,\gamma\,\alpha}\\
		&-q_{\beta}^{2}(y)\blacktriangleright_{\beta,\,\gamma\,\alpha}(q_{\beta}(x)\blacktriangleright_{\alpha,\,\gamma}p_{\gamma}(z))+(p_{\alpha}^{-1}q_{\alpha}^{2}(x)\blacktriangleright_{\alpha,\,\gamma}q_{\gamma}(z))\blacktriangleright_{\alpha\,\gamma,\,\beta}p_{\beta}q_{\beta}(y)-\lbrace q_{\beta}^{2}(y),q_{\alpha}(x)\blacktriangleright_{\alpha,\,\gamma}p_{\gamma}(z)\rbrace_{\beta,\,\gamma\,\alpha}\\
		&+q_{\beta}^{2}(y)\blacktriangleright_{\beta,\,\gamma\,\alpha}\lbrace q_{\gamma}(z),p_{\alpha}(x)\rbrace_{\gamma,\,\alpha}-\lbrace p_{\gamma}^{-1}q_{\gamma}^{2}(z),q_{\alpha}(x)\rbrace_{\gamma,\,\alpha}\blacktriangleright_{\gamma\,\alpha,\,\beta}p_{\beta}q_{\beta}(y)+\lbrace q_{\beta}^{2}(y),\lbrace q_{\gamma}(z),p_{\alpha}(x)\rbrace_{\gamma,\,\alpha}\rbrace_{\beta,\,\gamma\,\alpha}\\
		&+q_{\gamma}^{2}(z)\blacktriangleright_{\gamma,\,\alpha\,\beta}(q_{\alpha}(x)\blacktriangleright_{\alpha,\,\beta}p_{\beta}(y))-(p_{\alpha}^{-1}q_{\alpha}^{2}(x)\blacktriangleright_{\alpha,\,\beta}q_{\beta}(y))\blacktriangleright_{\alpha\,\beta,\,\gamma}p_{\gamma}q_{\gamma}(z)+\lbrace q_{\gamma}^{2}(z),q_{\alpha}(x)\blacktriangleright_{\alpha,\,\beta}p_{\beta}(y)\rbrace_{\gamma,\,\alpha\,\beta}\\
		&-q_{\gamma}^{2}(z)\blacktriangleright_{\gamma,\,\alpha\,\beta}(q_{\beta}(y)\blacktriangleright_{\beta,\,\alpha}p_{\alpha}(x))+(p_{\beta}^{-1}q_{\beta}^{2}(y)\blacktriangleright_{\beta,\,\alpha}q_{\alpha}(x))\blacktriangleright_{\beta\,\alpha,\,\gamma}p_{\gamma}q_{\gamma}(z)-\lbrace q_{\gamma}^{2}(z),q_{\beta}(y)\blacktriangleright_{\beta,\,\alpha}p_{\alpha}(x)\rbrace_{\gamma,\,\beta\,\alpha}\\
		&+q_{\gamma}^{2}(z)\blacktriangleright_{\gamma,\,\alpha\,\beta}\lbrace q_{\alpha}(x),p_{\beta}(y)\rbrace_{\alpha,\,\beta}-\lbrace p_{\alpha}^{-1}q_{\alpha}^{2}(x),q_{\beta}(y)\rbrace_{\alpha,\,\beta}\blacktriangleright_{\alpha\,\beta,\,\gamma}p_{\gamma}q_{\gamma}(z)+\lbrace q_{\gamma}^{2}(z),\lbrace q_{\alpha}(x),p_{\beta}(y)\rbrace_{\alpha,\,\beta}\rbrace_{\gamma,\,\alpha\,\beta}\\
		=&0.\hspace{1cm}\text{(by $(L, \{\cdot , \cdot \}_{\alpha,\,\beta}, \blacktriangleright_{\alpha ,\, \beta}, p_{\alpha}, q_{\alpha})_{\alpha,\,\beta\in \Omega}$ being a BiHom-$\Omega$-PostLie algebra)}
	\end{align*}
\end{proof}

Motivated by~\cite[Corollary 5.6]{bai}, we have the following result.
%The next result is the BiHom-$\Omega$ version of~\cite[Corollary 5.6]{bai}.
\begin{prop} \label{comgen}
	Let $(L, \lbrace \cdot , \cdot \rbrace _{\alpha,\,\beta}, p_{\alpha} , q_{\alpha} )_{\alpha,\,\beta\in \Omega}$ be a BiHom-$\Omega$-Lie algebra and let $(R_{\alpha})_{\alpha\in \Omega}$ be a Rota-Baxter family of weight $\lambda $ on $L$ such that
	\[ R_{\alpha}\circ p_{\alpha} =p_{\alpha} \circ R_{\alpha} \;\text{and} \; R_{\alpha}\circ q_{\alpha} =q_{\alpha}\circ R_{\alpha}.\]
  We define two operations
	on $L$ by
		\[\langle x, y\rangle_{\alpha,\,\beta}:=\lambda \lbrace x, y\rbrace _{\alpha,\,\beta} \;\text{and}\; x\blacktriangleright_{\alpha ,\, \beta}y:=\lbrace R_{\alpha}(x), y\rbrace _{\alpha,\,\beta},\]
for all $x, y\in L,\,\alpha,\,\beta\in \Omega.$	Then $(L, \langle\cdot , \cdot \rangle_{\alpha,\,\beta}, \blacktriangleright_{\alpha ,\, \beta}, p_{\alpha} , q_{\alpha})_{\alpha,\,\beta\in \Omega}$ is a BiHom-$\Omega$-PostLie algebra.
\end{prop}
\begin{proof}
	 By $(L, \lbrace \cdot , \cdot \rbrace _{\alpha,\,\beta}, p_{\alpha} , q_{\alpha} )_{\alpha,\,\beta\in \Omega}$ being a BiHom-$\Omega$-Lie algebra, we get $(L,\langle\cdot,\cdot\rangle_{\alpha,\,\beta},p_{\alpha},q_{\alpha})_{\alpha,\,\beta\in \Omega} $ is a BiHom-$\Omega$-Lie algebra, where $ \langle x, y\rangle_{\alpha,\,\beta}:=\lambda \lbrace x, y\rbrace _{\alpha,\,\beta} $. For $ x,y,z\in L,\,\alpha,\,\beta,\,\gamma\in \Omega$, we have
	\begin{align*}
	    p_{\alpha\,\beta}(x\blacktriangleright_{\alpha,\,\beta}y)=&p_{\alpha\,\beta}(\lbrace R_{\alpha}(x),y\rbrace _{\alpha,\,\beta})=\lbrace p_{\alpha}R_{\alpha}(x),p_{\beta}(y)\rbrace _{\alpha,\,\beta}\\
	    =&\lbrace R_{\alpha}p_{\alpha}(x),p_{\beta}(y)\rbrace _{\alpha,\,\beta}\hspace{1cm}\text{(by $p_{\alpha}\circ R_{\alpha}=R_{\alpha}\circ p_{\alpha}$)}\\
	    =&p_{\alpha}(x)\blacktriangleright_{\alpha,\,\beta}p_{\beta}(y).
	\end{align*}
Similarly, we get $ q_{\alpha\,\beta}(x\blacktriangleright_{\alpha,\,\beta}y)=q_{\alpha}(x)\blacktriangleright_{\alpha,\,\beta}q_{\beta}(y). $ Next, we have
\begin{align*}
	\langle &q_{\alpha}(x),p_{\beta}(y)\rangle_{\alpha,\,\beta}\blacktriangleright_{\alpha\,\beta,\,\gamma}q_{\gamma}(z)+(q_{\alpha}(x)\blacktriangleright_{\alpha,\,\beta}p_{\beta}(y))\blacktriangleright_{\alpha\,\beta,\,\gamma}q_{\gamma}(z)+p_{\beta}q_{\beta}(y)\blacktriangleright_{\beta,\,\alpha\,\gamma}(p_{\alpha}(x)\blacktriangleright_{\alpha,\,\gamma}z)\\
	&-p_{\alpha}q_{\alpha}(x)\blacktriangleright_{\alpha,\,\beta\,\gamma}(p_{\beta}(y)\blacktriangleright_{\beta,\,\gamma}z)-(q_{\beta}(y)\blacktriangleright_{\beta,\,\alpha}p_{\alpha}(x))\blacktriangleright_{\beta\,\alpha,\,\gamma}q_{\gamma}(z)\\
	=&\lambda\lbrace R_{\alpha\,\beta}\lbrace q_{\alpha}(x),p_{\beta}(y)\rbrace _{\alpha,\,\beta},q_{\gamma}(z)\rbrace _{\alpha\,\beta,\,\gamma}+\lbrace R_{\alpha\,\beta}\lbrace R_{\alpha}q_{\alpha}(x),p_{\beta}(y)\rbrace _{\alpha,\,\beta},q_{\gamma}(z)\rbrace _{\alpha\,\beta,\,\gamma}\\
	&+\lbrace R_{\beta}p_{\beta}q_{\beta}(y),\lbrace R_{\alpha}p_{\alpha}(x),z\rbrace _{\alpha,\,\gamma}\rbrace _{\beta,\,\alpha\,\gamma}-\lbrace R_{\alpha}p_{\alpha}q_{\alpha}(x),\lbrace R_{\beta}p_{\beta}(y),z\rbrace _{\beta,\,\gamma}\rbrace _{\alpha,\,\beta\,\gamma}\\
	&-\lbrace R_{\beta\,\alpha}\lbrace R_{\beta}q_{\beta}(y),p_{\alpha}(x)\rbrace _{\beta,\,\alpha},q_{\gamma}(z)\rbrace _{\beta\,\alpha,\,\gamma}\\
	=&\lbrace \lambda R_{\alpha\,\beta}\lbrace q_{\alpha}(x),p_{\beta}(y)\rbrace _{\alpha,\,\beta}+R_{\alpha\,\beta}\lbrace R_{\alpha}q_{\alpha}(x),p_{\beta}(y)\rbrace _{\alpha,\,\beta}+R_{\alpha\,\beta}\lbrace q_{\alpha}(x),R_{\beta}p_{\beta}(y)\rbrace _{\alpha,\,\beta},q_{\gamma}(z)\rbrace _{\alpha\,\beta,\,\gamma}\\
	&+\lbrace R_{\beta}p_{\beta}q_{\beta}(y),\lbrace R_{\alpha}p_{\alpha}(x),z\rbrace _{\alpha,\,\gamma}\rbrace _{\beta,\,\alpha\,\gamma}-\lbrace R_{\alpha}p_{\alpha}q_{\alpha}(x),\lbrace R_{\beta}p_{\beta}(y),z\rbrace _{\beta,\,\gamma}\rbrace _{\alpha,\,\beta\,\gamma}\hspace{1cm}\text{(by Eq.~(\ref{BHO-skew}))}\\
	=&\lbrace \lbrace R_{\alpha}q_{\alpha}(x),R_{\beta}p_{\beta}(y)\rbrace _{\alpha,\,\beta},q_{\gamma}(z)\rbrace _{\alpha\,\beta,\,\gamma}+\lbrace R_{\beta}p_{\beta}q_{\beta}(y),\lbrace R_{\alpha}p_{\alpha}(x),z\rbrace _{\alpha,\,\gamma}\rbrace _{\beta,\,\alpha\,\gamma}\\
	&-\lbrace R_{\alpha}p_{\alpha}q_{\alpha}(x),\lbrace R_{\beta}p_{\beta}(y),z\rbrace _{\beta,\,\gamma}\rbrace _{\alpha,\,\beta\,\gamma}\hspace{1cm}\text{(by Eq.~(\ref{RB-Eq}))}\\
	=&\lbrace \lbrace q_{\alpha}R_{\alpha}(x),R_{\beta}p_{\beta}(y)\rbrace _{\alpha,\,\beta},q_{\gamma}(z)\rbrace _{\alpha\,\beta,\,\gamma}+\lbrace q_{\beta}R_{\beta}p_{\beta}(y),\lbrace p_{\alpha}R_{\alpha}(x),z\rbrace _{\alpha,\,\gamma}\rbrace _{\beta,\,\alpha\,\gamma}\\
	&-\lbrace p_{\alpha}q_{\alpha}R_{\alpha}(x),\lbrace R_{\beta}p_{\beta}(y),z\rbrace _{\beta,\,\gamma}\rbrace _{\alpha,\,\beta\,\gamma}\\
	=&\lbrace q_{\alpha\,\beta}q_{\alpha\,\beta}^{-1}\lbrace q_{\alpha}R_{\alpha}(x),R_{\beta}p_{\beta}(y)\rbrace _{\alpha,\,\beta},p_{\gamma}p_{\gamma}^{-1}q_{\gamma}(z)\rbrace _{\alpha\,\beta,\,\gamma}+\lbrace q_{\beta}R_{\beta}p_{\beta}(y),\lbrace q_{\alpha}q_{\alpha}^{-1}p_{\alpha}R_{\alpha}(x),p_{\gamma}p_{\gamma}^{-1}(z)\rbrace _{\alpha,\,\gamma}\rbrace _{\beta,\,\alpha\,\gamma}\\
	&-\lbrace p_{\alpha}q_{\alpha}R_{\alpha}(x),\lbrace R_{\beta}p_{\beta}(y),z\rbrace _{\beta,\,\gamma}\rbrace _{\alpha,\,\beta\,\gamma}\\
	=&-\lbrace q_{\gamma}p_{\gamma}^{-1}q_{\gamma}(z),p_{\alpha\,\beta}q_{\alpha\,\beta}^{-1}\lbrace q_{\alpha}R_{\alpha}(x),R_{\beta}p_{\beta}(y)\rbrace _{\alpha,\,\beta}\rbrace _{\gamma,\,\alpha\,\beta}-\lbrace q_{\beta}R_{\beta}p_{\beta}(y),\lbrace q_{\gamma}p_{\gamma}^{-1}(z),p_{\alpha}q_{\alpha}^{-1}p_{\alpha}R_{\alpha}(x)\rbrace _{\gamma,\,\alpha}\rbrace _{\beta,\,\gamma\,\alpha}\\
	&-\lbrace p_{\alpha}q_{\alpha}R_{\alpha}(x),\lbrace R_{\beta}p_{\beta}(y),z\rbrace _{\beta,\,\gamma}\rbrace _{\alpha,\,\beta\,\gamma}\hspace{1cm}\text{(by Eq.~(\ref{BHO-skew}))}\\
	=&-\lbrace q_{\alpha}^{2}q_{\alpha}^{-1}p_{\alpha}R_{\alpha}(x),\lbrace q_{\beta}q_{\beta}^{-1}R_{\beta}p_{\beta}(y),p_{\gamma}p_{\gamma}^{-1}(z)\rbrace _{\beta,\,\gamma}\rbrace _{\alpha,\,\beta\,\gamma}-\lbrace q_{\beta}^{2}q_{\beta}^{-1}R_{\beta}p_{\beta}(y),\lbrace q_{\gamma}p_{\gamma}^{-1}(z),p_{\alpha}q_{\alpha}^{-1}p_{\alpha}R_{\alpha}(x)\rbrace _{\gamma,\,\alpha}\rbrace _{\beta,\,\gamma\,\alpha}\\
	&-\lbrace q_{\gamma}^{2}p_{\gamma}^{-1}(z),\lbrace q_{\alpha}q_{\alpha}^{-1}p_{\alpha}R_{\alpha}(x),p_{\beta}q_{\beta}^{-1}R_{\beta}p_{\beta}(y)\rbrace _{\alpha,\,\beta}\rbrace _{\gamma,\,\alpha\,\beta}\hspace{1cm}\text{(by $p_{\alpha\,\beta}, q_{\alpha\,\beta}^{-1}$ satisfying Eq.~(\ref{Omega-Lie-morphism}))}\\
	=&0.\hspace{1cm}\text{(by Eq.~(\ref{BHO-Jacobi}))}
\end{align*}
Similarly, we have
\[ \langle q_{\alpha}(x)\blacktriangleright_{\alpha,\,\beta}y,q_{\gamma}(z)\rangle_{\alpha\,\beta,\,\gamma}+\langle q_{\beta}(y),p_{\alpha}(x)\blacktriangleright_{\alpha,\,\gamma}z\rangle_{\beta,\,\alpha\,\gamma}=p_{\alpha}q_{\alpha}(x)\blacktriangleright_{\alpha,\,\beta\,\gamma}\langle y,z\rangle_{\beta,\,\gamma}.\]
This completes the proof.
\end{proof}

\begin{remark}
If $\lambda =0$, then Proposition \ref{comgen} reduces to Theorem \ref{lrprelie}.
\end{remark}

\begin{remark}\label{relation}%\yuan{I don't understand}
	By the link among Theorem \ref{Proposition:BiHLie}, Proposition \ref{interm} and Proposition~\ref{comgen}, we have the following commutative diagram.
	
	\[\xymatrix{
		&\txt{BiHom-$\Omega$-Lie algebra\\$ (L,\lbrace \cdot,\cdot\rbrace_{\alpha,\,\beta},p_{\alpha},q_{\alpha})_{\alpha,\,\beta\in\Omega} $}\ar@<.5ex>[rr]^{\text{Proposition \ref{comgen}}}_{\langle x, y\rangle_{\alpha,\,\beta}:=\lambda \lbrace x, y\rbrace _{\alpha,\,\beta} \;\text{and}\; x\blacktriangleright_{\alpha ,\, \beta}y:=\lbrace R_{\alpha}(x), y\rbrace _{\alpha,\,\beta}}\ar@<.5ex>[rdd]^{\text{Theorem \ref{Proposition:BiHLie}}}_{<x,y>_{\alpha,\,\beta}:=\lbrace R_{\alpha}(x),y\rbrace _{\alpha,\,\beta}+\lbrace x,R_{\beta}(y)\rbrace _{\alpha,\,\beta}+\lambda \lbrace x,y\rbrace _{\alpha,\,\beta}\qquad\;\;}&\;& \txt{BiHom-$\Omega$-PostLie algebra\\$(L, \langle\cdot , \cdot \rangle_{\alpha,\,\beta}, \blacktriangleright_{\alpha ,\, \beta}, p_{\alpha} , q_{\alpha})_{\alpha,\,\beta\in \Omega}$}\ar@<.5ex>[ldd]_{\text{Proposition \ref{interm}}}^{\qquad\quad\;\;\;<x, y>_{\alpha,\,\beta}:=x\blacktriangleright_{\alpha , \,\beta}y-(p_{\beta} ^{-1}q_{\beta}(y))\blacktriangleright_{\beta,\,\alpha}(p_{\alpha} q_{\alpha}^{-1}(x))+\langle x, y\rangle_{\alpha,\,\beta}}\\
		\;& \;& \;\\
	\;&\;&\txt{BiHom-$\Omega$-Lie algebra\\$(L, <\cdot , \cdot >_{\alpha,\,\beta}, p_{\alpha}, q_{\alpha} )_{\alpha,\,\beta\in \Omega}.$}
	}\]\\
	More precisely, we have
	\begin{align*}
		<x, y>_{\alpha,\,\beta}=&x\blacktriangleright_{\alpha ,\, \beta}y-(p_{\beta}^{-1}q_{\beta}(y))\blacktriangleright_{\beta,\,\alpha}(p_{\alpha} q_{\alpha}^{-1}(x))+\langle x, y\rangle_{\alpha,\,\beta}\\
		=&\{R_{\alpha}(x), y\}_{\alpha,\,\beta}-\{R_{\beta}p_{\beta}^{-1}q_{\beta} (y), p_{\alpha}q_{\alpha} ^{-1}(x)\}_{\beta,\,\alpha}+\lambda \{x, y\}_{\alpha,\,\beta}\\
		=&\{R_{\alpha}(x), y\}_{\alpha,\,\beta}-\{q_{\beta}p_{\beta} ^{-1}R_{\beta}(y), p_{\alpha}q_{\alpha} ^{-1}(x)\}_{\beta,\,\alpha}+\lambda \{x, y\}_{\alpha,\,\beta}\\
		=&\{R_{\alpha}(x), y\}_{\alpha,\,\beta}+\{q_{\alpha}q_{\alpha}^{-1}(x), p_{\beta}p_{\beta}^{-1}R_{\beta}(y)\}_{\alpha,\,\beta}+\lambda \{x, y\}_{\alpha,\,\beta}\\
		&\hspace{1cm} \text{(by Eq.~(\ref{BHO-skew}))}\\
		=&\{R_{\alpha}(x), y\}_{\alpha,\,\beta}+\{x, R_{\beta}(y)\}_{\alpha,\,\beta}+\lambda \{x, y\}_{\alpha,\,\beta},
	\end{align*}
for all $ x,y\in L,\,\alpha,\,\beta\in \Omega. $
\end{remark}

%%%%%%%%%%%%%%%%%%%%%%%%%%%%%%

%%%%%%%%%%%%%%%%%%%
\subsection{BiHom-$\Omega$-pre-Possion algebras}
%%%%%%%%%%%%%%%%%%%
In this subsection, we generlize the relationship between pre-Lie algebras and pre-Possion algebras to BiHom-$\Omega$ version.
First of all, let's recall the concepts of $\Omega$-zinbiel algebras and $\Omega$-pre-Possion algebras.

\begin{defn}\cite{Aguiar}
   	\begin{enumerate}
		\item An $\bf{\Omega}$-$\bf{zinbiel \, algebra}$ $ (Z,\ast_{\alpha,\,\beta})_{\alpha,\,\beta\in \Omega} $  is a vector space $Z$ equipped with a family of binary operations $ (\ast_{\alpha,\,\beta}: Z\times Z\rightarrow Z)_{\alpha,\,\beta\in \Omega} $ such that
		\begin{align}\label{Omega-zinbiel}
			x\ast_{\alpha,\,\beta\,\gamma}(y\ast_{\beta,\,\gamma}z)=(x\ast_{\alpha,\,\beta}y)\ast_{\alpha\,\beta,\,\gamma}z+(y\ast_{\beta,\,\alpha}x)\ast_{\beta\,\alpha,\,\gamma}z,
		\end{align}
for all $ x,y,z\in Z,\,\alpha,\,\beta,\,\gamma\in \Omega. $
\item An $\bf{\Omega}$-$\bf{pre}$-$\bf{Possion \, algebra}$ $ (B,\rhd_{\alpha,\,\beta},\ast_{\alpha,\,\beta})_{\alpha,\,\beta\in \Omega} $ is a vector space $B$ equipped with two families of bilinear operations $ \rhd_{\alpha,\,\beta},\ast_{\alpha,\,\beta}: B\times B\rightarrow B $ such that $ (B,\rhd_{\alpha,\,\beta})_{\alpha,\,\beta\in \Omega} $ is an $\Omega$-pre-Lie algebra, $ (B,\ast_{\alpha,\,\beta})_{\alpha,\,\beta\in \Omega} $ is an $\Omega$-zinbiel algebra and
\begin{align*}
	(x\rhd_{\alpha,\,\beta}y-y\rhd_{\beta,\,\alpha}x)\ast_{\alpha\,\beta,\,\gamma}z=x\rhd_{\alpha,\,\beta\,\gamma}(y\ast_{\beta,\,\gamma}z)-y\ast_{\beta,\,\alpha\,\gamma}(x\rhd_{\alpha,\,\gamma}z),\\%\label{Omega-pre-Possion-1}\\
	(x\ast_{\alpha,\,\beta}y+y\ast_{\beta,\,\alpha}x)\rhd_{\alpha\,\beta,\,\gamma}z=x\ast_{\alpha,\,\beta\,\gamma}(y\rhd_{\beta,\,\gamma}z)+y\ast_{\beta,\,\alpha\,\gamma}(x\rhd_{\alpha,\,\gamma}z),%\label{Omega-pre-Possion-2},
\end{align*}
for all $ x,y,z\in B,\,\alpha,\,\beta,\,\gamma\in \Omega. $
	\end{enumerate}
\end{defn}

\begin{defn}
	\begin{enumerate}
\item	Let $ (Z,\ast_{\alpha,\,\beta})_{\alpha,\,\beta\in \Omega} $ and $ (Z',\ast_{\alpha,\,\beta}')_{\alpha,\,\beta\in \Omega} $ be two $\Omega$-zinbiel algebras. A family of linear maps $ (f_{\alpha})_{\alpha\in \Omega} : Z\rightarrow Z'$ is called an $\bf{\Omega}$-$\bf{zinbiel \, algebra \, morphism}$ if
	\begin{align}\label{Omega-zinbiel-morphism}
		f_{\alpha\,\beta}(x\ast_{\alpha,\,\beta}y)=f_{\alpha}(x)\ast'_{\alpha,\,\beta}f_{\beta}(y),
	\end{align}
	for all $ x,y\in Z,\,\alpha,\,\beta\in \Omega. $
	\item Let $ (B,\rhd_{\alpha,\,\beta},\ast_{\alpha,\,\beta})_{\alpha,\,\beta\in \Omega} $ and $ (B',\rhd_{\alpha,\,\beta}',\ast_{\alpha,\,\beta}')_{\alpha,\,\beta\in \Omega} $ be two $\Omega$-pre-Possion algebras. A family of linear maps $(f_{\alpha})_{\alpha\in \Omega}: B\rightarrow B'$ is called an $\bf{\Omega}$-$\bf{pre}$-$\bf{Possion \, algebra \, morphism}$ if
	\[f_{\alpha\,\beta}(x\rhd_{\alpha,\,\beta}y)=f_{\alpha}(x)\rhd'_{\alpha,\,\beta}f_{\beta}(y),\]
	\[f_{\alpha\,\beta}(x\ast_{\alpha,\,\beta}y)=f_{\alpha}(x)\ast'_{\alpha,\,\beta}f_{\beta}(y),\]
	for all $ x,y\in B,\,\alpha,\,\beta\in \Omega. $
	\end{enumerate}	
\end{defn}
Below, we will generlize the above definitions to the BiHom version.

\begin{defn}
	A $\bf{BiHom}$-$\bf{\Omega}$-$\bf{zinbiel \, algebra} $ $ (Z,\circledast_{\alpha,\,\beta},p_{\alpha},q_{\alpha})_{\alpha,\,\beta\in \Omega} $ is a vector space $Z$ equipped with a family of binary operations $ (\circledast_{\alpha,\,\beta}: Z\times Z\rightarrow Z)_{\alpha,\,\beta\in \Omega} $ and two commuting $\Omega$-zinbiel algebra morphisms $p_{\alpha},q_{\alpha}: Z\rightarrow Z$ such that
	\begin{align*}
		p_{\alpha}q_{\alpha}(x)\circledast_{\alpha,\,\beta\,\gamma}(p_{\beta}(y)\circledast_{\beta,\,\gamma}z)=(q_{\alpha}(x)\circledast_{\alpha,\,\beta}p_{\beta}(y))\circledast_{\alpha\,\beta,\,\gamma}q_{\gamma}(z)+(q_{\beta}(y)\circledast_{\beta,\,\alpha}p_{\alpha}(x))\circledast_{\beta\,\alpha,\,\gamma}q_{\gamma}(z),
	\end{align*}
for all $ x,y,z\in Z,\,\alpha,\,\beta,\gamma\in \Omega. $ The maps $p_{\alpha}$ and $q_{\alpha}$ (in this order) are called the structure maps of $Z$.
\end{defn}

Combining BiHom-$\Omega$-zinbiel algebras and BiHom-$\Omega$-pre-Lie algebras, we propose the following definition.
%For the $\Omega$-pre-Possion algebras, we have their BiHom version as follows.

\begin{defn}\label{pre-Possion}
	A $\bf{BiHom\text{-}}\bf{\Omega}\text{-}\bf{pre\text{-}Possion \, algebra}$ $ (B,\blacktriangleright_{\alpha,\,\beta},\circledast_{\alpha,\,\beta},p_{\alpha},q_{\alpha})_{\alpha,\,\beta\in \Omega} $ is a vector space $B$ equipped with two families of bilinear maps $ \blacktriangleright_{\alpha,\,\beta},\circledast_{\alpha,\,\beta}: B \times B\rightarrow B $ and two commuting families of linear maps $ (p_{\alpha})_{\alpha\in \Omega},(q_{\alpha})_{\alpha\in \Omega} : B \rightarrow B$ such that $ (B,\blacktriangleright_{\alpha,\,\beta},p_{\alpha},q_{\alpha})_{\alpha,\,\beta\in \Omega} $ is a BiHom-$\Omega$-pre-Lie algebra, $ (B,\circledast_{\alpha,\,\beta},p_{\alpha},q_{\alpha})_{\alpha,\,\beta\in \Omega} $ is a BiHom-$\Omega$-zinbiel algebra and

	$	(q_{\alpha}(x)\blacktriangleright_{\alpha,\,\beta}p_{\beta}(y)-q_{\beta}(y)\blacktriangleright_{\beta,\,\alpha}p_{\alpha}(x))\circledast_{\alpha\,\beta,\,\gamma}q_{\gamma}(z)$
		\begin{align*}%\label{BiHom-O-pre-Possion-1}
	=p_{\alpha}q_{\alpha}(x)\blacktriangleright_{\alpha,\,\beta\,\gamma}(p_{\beta}(y)\circledast_{\beta,\,\gamma}z)-p_{\beta}q_{\beta}(y)\circledast_{\beta,\,\alpha\,\gamma}(p_{\alpha}(x)\blacktriangleright_{\alpha,\,\gamma}z),
		\end{align*}
	$	(q_{\alpha}(x)\circledast_{\alpha,\,\beta}p_{\beta}(y)+q_{\beta}(y)\circledast_{\beta,\,\alpha}p_{\alpha}(x))\blacktriangleright_{\alpha\,\beta,\,\gamma}q_{\gamma}(z)$
	\begin{align*}%\label{BiHom-O-pre-Possion-2}
		=p_{\alpha}q_{\alpha}(x)\circledast_{\alpha,\,\beta\,\gamma}(p_{\beta}(y)\blacktriangleright_{\beta,\,\gamma}z)+p_{\beta}q_{\beta}(y)\circledast_{\beta,\,\alpha\,\gamma}(p_{\alpha}(x)\blacktriangleright_{\alpha,\,\gamma}z),
	\end{align*}
for all $ x,y,z\in B,\,\alpha,\,\beta,\,\gamma\in \Omega .$ The maps $p_{\alpha}$ and $q_{\alpha}$ (in this order) are called the structure maps of $B$.
\end{defn}

Next, we introduce the relationship between BiHom-$\Omega$-pre-Possion algebras and BiHom-$\Omega$-pre-Lie algebras.

\begin{remark}
	Let $ (B,\blacktriangleright_{\alpha,\,\beta},\circledast_{\alpha,\,\beta},p_{\alpha},q_{\alpha})_{\alpha,\,\beta\in \Omega} $ be a BiHom-$\Omega$-pre-Possion algebra.
	\begin{enumerate}
		\item If $ x\circledast_{\alpha,\,\beta}y=0, $ for all $ x,y\in B,\;\alpha,\,\beta\in \Omega$, then $ (B,\blacktriangleright_{\alpha,\,\beta},p_{\alpha},q_{\alpha})_{\alpha,\,\beta\in \Omega} $ is a BiHom-$\Omega$-pre-Lie algebra.
		\item If $ x\blacktriangleright_{\alpha,\,\beta}y=0, $ for all $ x,y\in B,\;\alpha,\,\beta\in \Omega$, then $ (B,\circledast_{\alpha,\,\beta},p_{\alpha},q_{\alpha})_{\alpha,\,\beta\in \Omega} $ is a BiHom-$\Omega$-zinbiel algebra.
	\end{enumerate}
\end{remark}

As usual, we characterize the Yau twisting procedure for BiHom-$\Omega$-zinbiel algebras as follows.

\begin{prop}\label{Yautwist-zinbiel}
	Let $(Z, \ast_{\alpha,\,\beta} )_{\alpha,\,\beta \in \Omega}$ be an $\Omega$-zinbiel algebra. If $p_{\alpha} , q_{\alpha} :Z\rightarrow Z$ are two commuting $\Omega$-zinbiel algebra morphisms and we define the multiplication on $Z$ by
	\[x\circledast_{\alpha,\,\beta}y:=p_{\alpha}(x)\ast_{\alpha,\,\beta} q_{\beta}(y), \quad \text{ for all }x, y\in Z,\,\alpha,\,\beta\in \Omega.\] Then $(Z, \circledast_{\alpha,\,\beta} , p_{\alpha} , q_{\alpha} )_{\alpha,\,\beta \in \Omega}$ is a BiHom-$\Omega$-zinbiel algebra,
	called the Yau twist of $(Z, \ast_{\alpha,\,\beta} )_{\alpha,\,\beta \in\Omega}$.
\end{prop}
\begin{proof}
	For $ x,y,z\in Z,\;\,\alpha,\,\beta,\,\gamma\in \Omega, $ we have\\
$ \phantom{1111111}\;\;\, p_{\alpha}q_{\alpha}(x)\circledast_{\alpha,\,\beta\,\gamma}(p_{\beta}(y)\circledast_{\beta,\,\gamma}z) $
	\begin{align*}
		=&p_{\alpha}^{2}q_{\alpha}(x)\ast_{\alpha,\,\beta\,\gamma}q_{\beta\,\gamma}(p_{\beta}(y)\circledast_{\beta,\,\gamma}z)\\
		=&p_{\alpha}^{2}q_{\alpha}(x)\ast_{\alpha,\,\beta\,\gamma}q_{\beta\,\gamma}(p_{\beta}^{2}(y)\ast_{\beta,\,\gamma}q_{\gamma}(z))\\
		=&p_{\alpha}^{2}q_{\alpha}(x)\ast_{\alpha,\,\beta\,\gamma}(q_{\beta}p_{\beta}^{2}(y)\ast_{\beta,\,\gamma}q_{\gamma}^{2}(z))\hspace{1cm}\text{(by $q_{\beta\,\gamma}$ satisfying Eq.~(\ref{Omega-zinbiel-morphism})})\\
		=&(p_{\alpha}^{2}q_{\alpha}(x)\ast_{\alpha,\,\beta}q_{\beta}p_{\beta}^{2}(y))\ast_{\alpha\,\beta,\,\gamma}q_{\gamma}^{2}(z)+(q_{\beta}q_{\beta}^{2}(y)\ast_{\beta,\,\alpha}p_{\alpha}^{2}q_{\alpha}(x))\ast_{\beta\,\alpha,\,\gamma}q_{\gamma}^{2}(z)\\
		&\hspace{1cm}\text{(by Eq.~(\ref{Omega-zinbiel})})\\
		=&(p_{\alpha}^{2}q_{\alpha}(x)\ast_{\alpha,\,\beta}p_{\beta}q_{\beta}p_{\beta}(y))\ast_{\alpha\,\beta,\,\gamma}q_{\gamma}^{2}(z)+(p_{\beta}^{2}q_{\beta}(y)\ast_{\beta,\,\alpha}p_{\alpha}q_{\alpha}p_{\alpha}(x))\ast_{\beta\,\alpha,\,\gamma}q_{\gamma}^{2}(z)\\
		&\hspace{1cm}\text{(by $p_{\alpha},q_{\alpha}$ commuting with each other })\\
		=&p_{\beta\,\alpha}(p_{\alpha}q_{\alpha}(x)\ast_{\alpha,\,\beta}q_{\beta}p_{\beta}(y))\ast_{\alpha\,\beta,\,\gamma}q_{\gamma}^{2}(z)+p_{\beta\,\alpha}(p_{\beta}q_{\beta}(y)\ast_{\beta,\,\alpha}q_{\alpha}p_{\alpha}(x))\ast_{\beta\,\alpha,\,\gamma}q_{\gamma}^{2}(z)\\
		&\hspace{1cm}\text{(by $p_{\alpha\,\beta}$ satisfying Eq.~(\ref{Omega-zinbiel-morphism})})\\
		=&(q_{\alpha}(x)\circledast_{\alpha,\,\beta}p_{\beta}(y))\circledast_{\alpha\,\beta,\,\gamma}q_{\gamma}(z)+(q_{\beta}(y)\circledast_{\beta,\,\alpha}p_{\alpha}(x))\circledast_{\beta\,\alpha,\,\gamma}q_{\gamma}(z).
		\end{align*}
\end{proof}

The following result is the Yau twisting procedure for BiHom-$\Omega$-pre-Possion algebras and the proof is similar to Proposition~\ref{Yautwist-zinbiel}.
\begin{prop}
	Let $ (B,\rhd_{\alpha,\,\beta},\ast_{\alpha,\,\beta})_{\alpha,\,\beta\in \Omega} $ be an $\Omega$-pre-Possion algebra and let $ p_{\alpha},q_{\alpha}: B\rightarrow B $ be two commuting $\Omega$-pre-Possion algebra morphisms. Define two operations on $B$ by
	\[x\blacktriangleright_{\alpha,\,\beta}y:=p_{\alpha}(x)\rhd_{\alpha,\,\beta}q_{\beta}(y),\;\qquad x\circledast_{\alpha,\,\beta}y:=p_{\alpha}(x)\ast_{\alpha,\,\beta}q_{\beta}(y),\]
	for all $ x,y\in B,\,\alpha,\,\beta\in \Omega. $ Then $ (B,\blacktriangleright_{\alpha,\,\beta},\circledast_{\alpha,\,\beta},p_{\alpha},q_{\alpha})_{\alpha,\,\beta\in \Omega} $ is a BiHom-$\Omega$-pre-Possion algebra, called the Yau twist of $ (B,\rhd_{\alpha,\,\beta},\ast_{\alpha,\,\beta})_{\alpha,\,\beta\in \Omega}. $
\end{prop}

\end{document}